\documentclass{amsart}
\usepackage{amssymb,amsmath}
\usepackage[american]{babel}
\usepackage{amsopn}
\usepackage{mathrsfs}
\usepackage{color}
\usepackage{wasysym}
\usepackage{vmargin}
\usepackage{amscd}
\usepackage{pdfsync}
\usepackage{verbatim}
\newtheorem{theorem}{Theorem}[section]
\newtheorem{lemma}[theorem]{Lemma}

\numberwithin{equation}{section}
\def\p{\partial}
\def\N{\mathbb N}

\def\R{\mathbb R}
\def\C{\mathbb C}

\newcommand\RR{{{\mathbb R}}}

\newcommand{\rr}{\mathbb{R}}
\newcommand{\eps}{\varepsilon}
\newcommand{\nn}{\mathbb{N}}
\newcommand{\cc}{\mathbb{C}}

\def\un{{\mathrm{1~\hspace{-1.4ex}l}}}
\def\poscal#1#2{\langle#1,#2\rangle}

\def\val#1{\vert#1\vert}

\def\valjp#1{\langle#1\rangle}

\def\l2{L^2(\R^{n})}
\def\L2{L^2(\R^{2n})}

\def\vs{\vskip.3cm}

\let \dis=\displaystyle

\let \dis=\displaystyle

\def\mat22#1#2#3#4{\begin{pmatrix}#1&#2\\ #3&#4\end{pmatrix}}

\def\finp{\operatorname{fp}}

\def\io{{\infty}}

\def\R{\mathbb R}\def\C{\mathbb C}\def\N{\mathbb N}

\def\D{\partial}\def\eps{\varepsilon}\def\phi{\varphi}
\newcommand{\real}{\mathbb{R}}

\def\wrt{with respect to }

\begin{document}


\title[Gelfand-Shilov and Gevrey smoothing effect for the Kac equation]{Gelfand-Shilov and Gevrey smoothing effect for the spatially inhomogeneous non-cutoff Kac equation}
\author{N. Lerner, Y. Morimoto, K. Pravda-Starov  \& C.-J. Xu}
\date{\today}
\address{\noindent \textsc{N. Lerner, Institut de Math\'ematiques de Jussieu,
Universit\'e Pierre et Marie Curie,
4 Place Jussieu,
75252 Paris cedex 05,
France}}
\email{nicolas.lerner@imj-prg.fr}
\address{\noindent \textsc{Y. Morimoto, Graduate School of Human and Environmental Studies,
Kyoto University, Kyoto 606-8501, Japan}}
\email{morimoto@math.h.kyoto-u.ac.jp }
\address{\noindent \textsc{K. Pravda-Starov,
IRMAR, CNRS UMR 6625, Universit\'e de Rennes 1,
Campus de Beaulieu, 263 avenue du G\'en\'eral Leclerc, CS 74205,
35042 Rennes cedex, France}}
\email{karel.pravda-starov@univ-rennes1.fr}
\address{\noindent \textsc{C.-J. Xu, Universit\'e de Rouen, CNRS UMR 6085, Laboratoire de Math\'ematiques, 76801 Saint-Etienne du Rouvray, France\\
and\\
School of Mathematics, Wuhan university 430072, Wuhan, P.R. China
}}
\email{Chao-Jiang.Xu@univ-rouen.fr}

\begin{abstract}
We consider the spatially inhomogeneous non-cutoff Kac's model of the Boltzmann equation.  We prove that the Cauchy problem for the fluctuation around the Maxwellian distribution enjoys Gelfand-Shilov regularizing properties with respect to the velocity variable and Gevrey regularizing properties with respect to the position variable.
\end{abstract}

\maketitle

\section{Introduction}

Kinetic equations with long range interactions, such as the Boltzmann equation without angular cutoff, are known to enjoy smoothing effects for the solutions of the associated  Cauchy problems. There have been recently a series of works studying the $C^\infty$ smoothing properties of the spatially inhomogeneous non-cutoff Boltzmann equation
(see the articles by
Alexandre, Morimoto, Ukai, Xu \&Yang 
\cite{AMUXY2,AMUXY1,AM}). These studies were inspired by a pioneer work by Desvillettes \& Wennberg~\cite{DW}, together with previous results~\cite{AS1, AS2,  AMUXY-KJM,  HMUY,MUXY-DCDS} for the spatially homogeneous Boltzmann equation and an earlier work in the mid-nineties for a model equation of the radially symmetric spatially homogeneous Boltzmann equation given by the Kac equation~\cite{D95}
(see also \cite{MR2594923,MR1407542,MR1720101,MR1711273}).

Regarding the Gevrey smoothing features and following the work \cite{MR2556716},
 we studied in the recent article~\cite{LMPX3} the Gelfand-Shilov
regularizing properties of the radially symmetric spatially homogeneous non-cutoff Boltzmann equation and we established that the Cauchy problem for small fluctuations around the Maxwellian distribution enjoys the very same smoothing properties as the linear evolution equation associated to a fractional power of the harmonic oscillator
\begin{equation}\label{fracsmooth}
\begin{cases}
\partial_tg+\mathcal{H}^sg=0,\\
g|_{t=0}=g_0  \in L^2(\rr_v^d),
\end{cases}
\end{equation}
with $0<s<1$, $\mathcal{H}=-\triangle_v+\frac{|v|^2}{4}$ and $d=3$.
This result shows that the radially symmetric spatially homogeneous Boltzmann equation, which reduces to the spatially homogeneous Kac equation, enjoys a Gelfand-Shilov smoothing effect in the space $S^{1/2s}_{1/2s}(\rr_v^d)$ for all positive time $t>0$, where the Gelfand-Shilov spaces $S_{\nu}^{\mu}(\rr_v^d)$, with $\mu,\nu>0$, $\mu+\nu\geq 1$, are defined as the spaces of smooth functions $f \in C^{\infty}(\rr_v^d)$ satisfying
$$\exists C \geq 1, \forall \alpha,\beta \in \nn^d, \quad \sup_{v \in \rr^d}|v^{\alpha}\partial_v^{\beta}f(v)| \leq C^{|\alpha|+|\beta|+1}(\alpha !)^{\nu}(\beta !)^{\mu}.$$
The Gelfand-Shilov spaces  $S_{\nu}^{\mu}(\rr^d)$ may also be characterized as the spaces of Schwartz functions belonging to the Gevrey space $G^{\mu}(\rr^d)$, whose Fourier transforms belong to the Gevrey space $G^{\nu}(\rr^d)$.

The analysis of the Gevrey regularizing properties of spatially inhomogeneous kinetic equations with respect to both position and velocity variables is more intricated. There are up to now only very few results except for a very simplified model of the linearized spatially inhomogeneous non-cutoff Boltzmann equation given by
the generalized Kolmogorov equation
\begin{equation}\label{H2}
\begin{cases}
\partial_tg+v \cdot \nabla_x g+(-\triangle_v)^s g=0,\\
g|_{t=0}=g_0\in L^2(\RR^{2d}_{x,v}),
\end{cases}
\end{equation}
with $0<s<1$, for which the second and the last authors established in~\cite{MX} that the solution to the Cauchy problem \eqref{H2} satisfies
\begin{equation}\label{frc2}
\exists c>0, \forall t>0, \quad e^{c(t^{2s+1}(-\triangle_x)^s +t\,(-\triangle_v)^s)} g(t)\in L^2(\RR^{2d}_{x,v}).
\end{equation}
This result indicates that the generalized Kolmogorov equation enjoys a $G^{\frac{1}{2s}}(\rr_{x,v}^{2d})$ Gevrey smoothing effect with respect to both position and velocity variables, despite the fact that diffusion only occurs in the velocity variables. This phenomenon of hypoellipticity is due to non-commutation and non-trivial interactions between the transport part $v \cdot\nabla_x$ and the diffusion part $(-\triangle_v)^s$ in this evolution equation.
The occurrence of hypoelliptic properties for kinetic equations was used and pointed out in many recent works,
such as the paper by Ars\'enio \& Saint-Raymond \cite{MR2832590}, as well as Golse's survey \cite{MR2976422}.
The work by  Alexandre, Morimoto, Ukai, Xu \& Yang \cite{AMUXY1}
highlighted the importance of regularization effects for Boltzmann equation
(see also \cite{AMUXY2,MR2877343,DFT,MR3050180}).
It served as a motivation for us to explore
more completely the behaviour of solutions of Kac's equation, 
a somewhat simplified model of Boltzmann equation
but still keeping some of the main features of Boltzmann's.
Studying whether this type of Gevrey smoothing features does hold, or not, for the spatially inhomogeneous non-cutoff Boltzmann equation is a challenging problem in mathematical physics. 
The models \eqref{fracsmooth} and \eqref{H2} are linear equations hopefully capturing some of the features of the Boltzmann equation regularizing properties.
We aim here at studying these regularizing properties for a non-linear model 
close to the Boltzmann equation. As an attempt for further understanding of the Gevrey smoothing features of
the Boltzmann equation, we study in this article the Gevrey regularizing properties of the spatially inhomogeneous Kac's model of the non-cutoff Boltzmann equation.

The spatially inhomogeneous Kac equation reads as the kinetic equation
\begin{equation}\label{w1}
\begin{cases}
\partial_tf+v\partial_{x}f=K(f,f),\\
f|_{t=0}=f_0,
\end{cases}
\end{equation}
for the density distribution of particles $f=f(t,x,v)$ at time $t$, having position $x \in \rr$ and velocity $v \in \rr$.
The Kac collision operator is defined as
$$K(g,f)=\int_{\val \theta\le \frac{\pi}{4}}\beta(\theta)\left(\int_{\RR} (g'_* f'-g_*f )dv_*\right)d\theta,$$
with the standard shorthand $f_*'=f(t,x,v_*')$, $f'=f(t,x,v')$, $f_*=f(t,x,v_*)$, $f=f(t,x,v)$, where the relations between pre and post collisional velocities
$$v'+iv'_{*}=e^{i\theta}(v+iv_{*}), \textrm{ i.e.,} \quad v'=v\cos\theta - v_*\sin\theta, \quad v'_* =v\sin\theta +v_*\cos\theta, \quad v,v_* \in \rr,$$
follow from the conservation of the kinetic energy in the binary collisions
$$v^2+v_{\ast}^2=v'^2+v_{\ast}'^2.$$
In this definition, the cross section is assumed to be an even non-negative function satisfying
$$\beta \geq 0, \quad \beta\in L^1_{\textrm{loc}}(]0,1[), \quad \beta(-\theta)=\beta(\theta),$$
with a non-integrable singularity for grazing collisions
$$\int_{-\frac{\pi}{4}}^{\frac{\pi}{4}}\beta(\theta)d\theta=+\infty.$$
This non-integrability plays a major role regarding the qualitative behaviour of the solutions of the Kac equation and this feature is essential for the smoothing effect to be present. Indeed, as first observed by Desvillettes~\cite{D95}, the grazing collisions accounting for the non-integrability of the cross section near $\theta=0$ do induce smoothing effects for the solutions of the non-cutoff Kac equation, or more generally for the solutions of the non-cutoff Boltzmann equation. On the other hand, these solutions are at most as regular as the initial data, see e.g.~\cite{36}, when the cross section is assumed to be integrable, or after removing the singularity by using a cutoff function (Grad's angular cutoff assumption).

We consider a cross section with a non-integrable singularity of the type
\begin{equation}\label{w5}
\beta(\theta) \ \substack{ \\ \dis\approx \\ \theta \to 0} |\theta|^{-1-2s},
\end{equation}
for\footnote{The notation $a\approx b$ means $a/b$
is bounded from above and below by fixed positive constants.} some given parameter $0 < s <1$. Under this assumption, the Kac collision operator may be defined as a finite part integral. We refer the reader to the appendix (Section~\ref{kacsection}) for details about this definition as a finite part integral.
Details on the physics background may be found in the extensive expositions~\cite{17,villani2} and the references therein.

We study the Kac equation in a close to equilibrium framework and consider the fluctuation
$$f=\mu+\sqrt{\mu}g,$$
around the normalized Maxwellian distribution
\begin{equation}\label{cv1}
\mu(v)=(2\pi)^{-\frac{1}{2}}e^{-\frac{v^2}{2}}, \quad v \in \rr.
\end{equation}
Since $K(\mu,\mu)=0$ by conservation of the kinetic energy, the Kac equation (\ref{w1}) for the fluctuation reads as
\begin{equation}\label{w7}
\begin{cases}
\partial_tg+v\partial_{x}g+\mathcal{K}g=\Gamma(g, g),\\
g|_{t=0}=g_0,
\end{cases}
\end{equation}
where $\mathcal{K}$ stands for the linearized Kac operator
$$\mathcal{K}g=-\mu^{-1/2}K(\mu,\mu^{1/2}g)-\mu^{-1/2}K(\mu^{1/2}g,\mu),$$
with
\begin{equation}\label{wr1}
\Gamma(f, g)=\mu^{-1/2}K(\mu^{1/2}f,\mu^{1/2}g).
\end{equation}
The linearized Kac operator was studied in the work~\cite{LMPX1}. We recall from this work that $\mathcal{K}$ is a non-negative unbounded operator on $L^2(\rr_v)$ with domain
$$\mathcal D=\Big\{u\in L^2({\rr_v}),\ \sum_{k\ge 0} k^{2s}\Vert{\mathbb P_{k}u}\Vert_{L^2}^2< +\infty\Big\}=\{u\in L^2({\rr_v}),\ \mathcal H^s u\in L^2(\rr_v)\},$$
where
\begin{align*}
\mathcal{H}=-\Delta_v+\frac{v^2}{4}  =
\sum_{k\ge 0} \Big( k + \frac{1}{2}\Big) \mathbb P_{k},
\end{align*}
stands for the harmonic oscillator and $\mathbb P_{k}$ denote the orthogonal projections
onto the Hermite basis described in Section \ref{6.sec.harmo}.
The fractional harmonic oscillator
$$\mathcal{H}^s=\sum_{k\ge 0} \Big( k + \frac{1}{2}\Big)^s \mathbb P_{k},$$
is defined through  functional calculus.
The linearized Kac operator
is diagonal in the Hermite basis
\begin{equation}\label{sal1}
\mathcal{K}=\sum_{k\ge 1}\lambda_{k}\mathbb P_{k},
\end{equation}
with a spectrum only composed by the non-negative eigenvalues
$$\lambda_{2k+1}=\int^{\frac{\pi}{4}}_{-\frac{\pi}{4}}\beta(\theta)
\big(1-(\cos\theta)^{2k+1}\big)d\theta \geq 0, \quad  k \geq 0,$$
$$\lambda_{2k}=\int^{\frac{\pi}{4}}_{-\frac{\pi}{4}}\beta(\theta)
\big(1-(\cos\theta)^{2k}-(\sin \theta)^{2k}\big)d\theta \geq 0, \quad  k \geq 1,$$
satisfying the asymptotic estimates
\begin{equation}\label{3.llkkb}
\lambda_{k} \approx k^s,
\end{equation}
when $k\rightarrow+\io$.
We notice that
$$0=\lambda_2 \leq \lambda_{2k} < \lambda_{2l}, \quad 0< \lambda_1 < \lambda_{2k+1} < \lambda_{2l+1},$$
when $1 \leq k <l$, and that $\lambda_1$ is the lowest positive eigenvalue.
The linearized Kac operator enjoys the coercive estimates
\begin{equation}\label{tr4}
\exists C>0, \forall f \in \mathscr{S}(\rr_v), \quad   \frac{1}{C}\|\mathcal{H}^{\frac{s}{2}}f\|_{L^2}^2 \leq  (\mathcal{K}f,f)_{L^2}+\|f\|_{L^2}^2  \leq C \|\mathcal{H}^{\frac{s}{2}}f\|_{L^2}^2,
\end{equation}
with a kernel given by
$$\textrm{Ker }\mathcal{K}=\textrm{Span}\{\psi_0,\psi_2\}.$$
The definition of the Hermite basis $(\psi_n)_{n \geq 0}$ is recalled in Section \ref{6.sec.harmo}.
We also recall the phase space properties of the linearized Kac operator established in~\cite{LMPX1}.
To that end, we make the following choice for the cross section
\begin{equation}\label{sing}
\beta(\theta)=\frac{|\cos\frac{\theta}{2}|}{|\sin\frac{\theta}{2}|^{1+2s}},\quad
\val \theta\leq\frac{\pi}{4}.
\end{equation}
This choice of cross section is made for simplicity in order to use directly the results of~\cite{LMPX1}. Notice that these results may be extended to a wider class of cross sections with the non-integrable singularity
$$\beta(\theta) \ \substack{ \\ \dis\sim \\ \theta \to 0} |\theta|^{-1-2s}.$$
With that choice, the eigenvalues satisfy the asymptotic equivalent
$$\lambda_{k}
\begin{matrix}{\scriptscriptstyle~}\\\sim\\{\scriptstyle k\rightarrow+\io} 
\end{matrix}\frac{2^{1+s}}{s}\mathbf{\Gamma}(1-s)k^s,$$ 
where $\mathbf{\Gamma}$ denotes the Gamma function. Furthermore, the linearized  Kac operator
\begin{equation}\label{cp4}
\mathcal K u=l^w(v,D_v) u =\frac{1}{2\pi}\int_{\rr^{2}}e^{i(v-y)  \eta}\, l\Big(\frac{v+y}{2},\eta\Big)u(y)dyd\eta,
\end{equation}
is a pseudodifferential operator whose Weyl symbol belongs to $\mathbf{S}^s(\rr^{2})$,
where for $m \in \rr$, the symbol class $\mathbf{S}^m(\rr^{2})$ is defined as  the set of smooth functions $a : \R^2 \rightarrow \C$ satisfying
$$\forall (\alpha,\beta) \in \nn^{2}, \exists C_{\alpha\beta}>0, \forall (v,\eta) \in \rr^{2},\quad
|\partial_v^{\alpha}\partial_{\eta}^{\beta}a(v,\eta)| \leq C_{\alpha,\beta} \langle (v,\eta) \rangle^{2m-|\alpha|-|\beta|},$$
with $\valjp{(v,\eta)}=\sqrt{1+\val v^2+\val \eta^2}$.
More specifically, the Weyl symbol $l(v,\eta)$ admits the following asymptotic expansion
\begin{multline}\label{cp5}
\forall N \geq 1, \quad  l(v,\eta)\equiv \frac{2^{1+s}}{s}\mathbf{\Gamma}(1-s) \Big(1+\eta^2+\frac{v^2}{4}\Big)^s-\frac{2^{1+s}(2+\sqrt{2})^s}{s} \\
+\sum_{k=1}^Nc_k\Big(1+\eta^2+\frac{v^2}{4}\Big)^{s-k} \textrm{ mod } \mathbf{S}^{s-N-1}(\rr^2),
\end{multline}
where $(c_k)_{k \geq 1}$ is a sequence of real numbers.
\vs
By using the above analysis, we established in~\cite{LMPX3} that the Cauchy problem associated to the spatially homogeneous Kac equation
$$
\begin{cases}
\partial_tg+\mathcal{K}g=\Gamma(g, g),\\
g|_{t=0}=g_0 \in L^2(\rr_v),
\end{cases}$$
enjoys exactly the same regularizing properties  
\begin{equation}\label{frc1}
\forall t>0, \exists C>1, \forall p,q \in \nn, \quad \sup_{v \in \rr}|v^p\partial_v^qg(t)| \leq C^{p+q+1}(p !)^{\frac{1}{2s}}(q !)^{\frac{1}{2s}},
\end{equation}
as the evolution equation~\eqref{fracsmooth}.
In this article, we consider the spatially inhomogeneous case and show that the Cauchy problem (\ref{w7}) is locally well-posed for sufficiently small initial data in the Sobolev space $H^{(1,0)}(\rr_{x,v}^2)$, where
$$H^{(r_1,r_2)}(\rr_{x,v}^2)=\{u \in L^2(\rr_{x,v}^2) : \langle D_x \rangle^{r_1}\langle D_v \rangle^{r_2} u \in L^2(\rr_{x,v}^2)\},$$
equipped with the dot product
$$(f,g)_{(r_1,r_2)}=\big(\langle D_x \rangle^{r_1}\langle D_v \rangle^{r_2}f,\langle D_x \rangle^{r_1}\langle D_v \rangle^{r_2}g\big)_{L^2(\rr_{x,v}^2)},$$
where $r_1,r_2 \in \rr$,  $\langle\, \cdot\, \rangle=\sqrt{1+|\cdot|^2}$, with $|\cdot|$ the Euclidean norm.
By taking advantage of the hypoelliptic properties of the linear operator
$$P=v\partial_x+\mathcal{K},$$
we prove that this Cauchy problem enjoys the same Gelfand-Shilov regularizing effect in the velocity variable and Gevrey regularizing effect in the position variable as the following evolution equation
$$\begin{cases}
\partial_tg+(\sqrt{\mathcal{H}}+\langle D_x\rangle)^{\frac{2s}{2s+1}}g=0,\\
g|_{t=0}=g_0 \in H^{(1,0)}(\rr_{x,v}^2).
\end{cases}$$
The following theorem is the main result contained in this article:

\medskip

\begin{theorem}\label{qw13.66ee}
Let $0<T<+\infty$. We assume that the collision cross section satisfies \eqref{sing} with $0<s<1$. Then,
there exist some positive constants $\eps_0>0$, $c_0 > 1$ such that for all $g_0 \in H^{(1,0)}(\rr_{x,v}^2)$ satisfying
$$\|g_0\|_{(1,0)} \leq \eps_0,$$
the Cauchy problem associated to the spatially inhomogeneous Kac equation
$$\begin{cases}
\partial_tg+v\partial_{x}g+\mathcal{K}g=\Gamma(g,g),\\
g|_{t=0}=g_0,
\end{cases}$$
admits a unique weak solution $g \in L^{\infty}([0,T],H^{(1,0)}(\rr_{x,v}^2))$
satisfying
$$\|g\|_{L^{\infty}([0,T],H^{(1,0)}(\rr_{x,v}^2))}+\|\mathcal{H}^{\frac{s}{2}}g\|_{L^{2}([0,T],H^{(1,0)}(\rr_{x,v}^2))} \leq c_0\|g_0\|_{(1,0)},$$
with $\mathcal{H}=-\triangle_v+\frac{v^2}{4}$.
Furthermore, this solution is smooth for all positive time $0<t \leq T$,
and satisfies  the Gelfand-Shilov and Gevrey type estimates:
$$\exists C>1, \forall 0<t \leq T, \forall k \geq 0, \quad \|(\sqrt{\mathcal{H}}+\langle D_x \rangle)^kg(t)\|_{(1,0)}\leq \frac{1}{t^{\frac{2s+1}{2s}k}}C^{k+1} (k!)^{\frac{2s+1}{2s}}\|g_0\|_{(1,0)},$$
in particular
\begin{multline*}
\forall \delta>0, \exists C>1, \forall 0<t \leq T, \forall k,l,p \geq 0, \\  \|v^k\partial_v^l\partial_x^pg(t)\|_{L^{\infty}(\rr_{x,v}^2)}
 \leq   \frac{C^{k+l+p+1}}{t^{\frac{2s+1}{2s}(k+l+p+3)+\delta}} (k!)^{\frac{2s+1}{2s}}(l!)^{\frac{2s+1}{2s}}(p!)^{\frac{2s+1}{2s}}\|g_0\|_{(1,0)}.
\end{multline*}
\end{theorem}

\medskip

This result establishes a Gelfand-Shilov smoothing effect in the velocity variable and a Gevrey smoothing effect in the position variable for the spatially inhomogeneous Kac equation 
$$\forall t>0,  \forall x\in\rr, \quad g(t, x,\,\cdot\,)\in S^{1+\frac{1}{2s}}_{1+\frac{1}{2s}}(\rr_v), \quad \forall t>0, \forall v\in\rr, \quad g(t,\,\cdot,\,v)\in G^{1+\frac{1}{2s}}(\rr_x).$$
We underline that in addition to unveiling this Gelfand-Shilov and Gevrey regularizing effects, the result of Theorem~\ref{qw13.66ee} also provides an explicit control of the Gelfand-Shilov and Gevrey semi-norms of the solutions for small times $t>0$.
The result of Theorem 1.1
is much more precise than the basic results controlling the moments for the solutions of the Kac equation since the Maxwellian solution belongs to the Gelfand-Shilov class $S^{\frac 12}_{\frac 12}$,
the regularity $S^{1+\frac{1}{2s}}_{1+\frac{1}{2s}}$ of the perturbation $g$ can be extended
to the solution itself.
In the result of Theorem~\ref{qw13.66ee}, we also notice that the Gelfand-Shilov and Gevrey regularity indices directly depend on the hypoelliptic properties of the linear operator
$$\|\mathcal{H}^sf\|_{L^2(\rr_{x,v}^2)}+\|\langle D_x \rangle^{\frac{2s}{2s+1}}f\|_{L^2(\rr_{x,v}^2)} \lesssim \|v\partial_{x}f+\mathcal{K}f\|_{L^2(\rr_{x,v}^2)},$$
with respect to the position variable. These a priori hypoelliptic estimates are known to be sharp. However, it is still open to determine whether, or not, the Gelfand-Shilov and Gevrey regularity indices in Theorem~\ref{qw13.66ee} are sharp. Indeed, we notice that the regularity with respect to the velocity variable in the spatially inhomogeneous case is weaker than the one obtained in the spatially homogeneous case (\ref{frc1}). Furthermore, the results for the simplified model given by the generalized Kolmogorov equation (\ref{frc2}) may indicate that stronger Gelfand-Shilov and Gevrey regularizing effects can possibly hold. It would be most interesting to understand further this optimality since the solutions to the Cauchy problem (\ref{w7}) would be analytic for cross sections with strong singularity $1/2 \leq s <1$, in the case when these stronger smoothing results hold.

Remark that it seems in principle possible to tackle similar questions
for the Boltzmann equation.
However several new difficulties are occurring in this case.
In the first place the linearized Boltzmann operator is more complicated than  Kac's linearization
(the latter is simply a fractional power of the harmonic oscillator). 
For Boltzmann equation, one should first introduce
the Landau operator, say in three dimensions,
$$
\mathscr L_{L}= 2\bigl(-\Delta_{v}+\frac{\val v^{2}}{4}-\frac32\bigr)-\Delta_{\mathbb S^{2}}+\text{finite rank operator.}
$$
The linearized Boltzmann operator $\mathscr L_{B}$ appears essentially as the $s$-power of  $\mathscr L_{L}$.
As a result, the preliminary study of $\mathscr L_{L}^{s}$ is more complicated to handle,
although optimal coercive estimates can be proven, e.g. \cite{MR3063534}.
Handling the non-linear perturbations and getting a global existence result
for initial data close enough to the Maxwellian 
along with regularization properties 
seems a quite realistic program,
but the complexity of the trilinear estimates
is dramatically increasing (see \cite{GLX1})  and 
some technical problems remain to be overcome.

The article is organized as follows. Section~\ref{preli} is devoted to preliminary results on the Kac collision operator. We first show that the bilinear operator (\ref{wr1}) may be computed explicitly along the Hermite basis (Lemma~\ref{proposition1a}). By taking advantage of these algebraic properties, we establish key trilinear estimates with exponential weights satisfied by the non-linear collision term (Lemma~\ref{proposition222}). In Section~\ref{fok} (Proposition~\ref{KFP2}), we make the explicit construction of a multiplier to derive the hypoelliptic properties of the linear operator
$$P=v\partial_{x}+\mathcal{K}.$$
This linear model has the specific structure 
$$\textrm{Transport part in the }(v,x) \textrm{ variables}\ + \textrm{ Elliptic part in the }v \textrm{ variable}.$$ 
The non-commutation of the transport part $v\partial_{x}$ with the diffusive part $\mathcal{K}$ accounts for the hypoelliptic properties of this linear operator.
Section~\ref{loun} is devoted to the proof of the local existence and uniqueness result (Theorems~\ref{qw13} and~\ref{qw13.66}) for the Cauchy problem (\ref{w7}), whereas Section~\ref{gelf} provides the proof of the Gelfand-Shilov and Gevrey smoothing effects. Finally, the last section is an appendix (Section~\ref{appendix}) providing instrumental estimates satisfied by Hermite functions (Section~\ref{6.sec.harmo}), some reminders about the Gelfand-Shilov regularity (Section~\ref{regularity}), the definition of the Kac collision operator as a finite part integral (Section~\ref{kacsection}) and properties of metrics on the phase space (Section~\ref{symbolic}).

\section{Some computations and estimates on the Kac collision operator}\label{preli}

\subsection{Computations of the Kac collision operator along the Hermite basis}
This section shows that the non-linear Kac collision operator enjoys specific algebraic features and that
the bilinear operator (\ref{wr1})
may be computed explicitly along the Hermite basis $(\psi_n)_{n \geq 0}$,
$$
\Gamma(\psi_{k}, \psi_{l})=\alpha_{k,l}\psi_{k+l}, \quad \alpha_{k,l} \in \rr.
$$
The following  lemma extends  the result of  \cite{LMPX3} (Lemma 3.3) to all Hermite functions with odd indices.


\begin{lemma}\label{proposition1a}
Let $(\psi_{n})_{n\geq 0}$ be the Hermite basis of $L^2(\R)$ described in Section~\ref{6.sec.harmo}.
We have
$$
\Gamma(\psi_{k}, \psi_{l})=\alpha_{k, l}\psi_{k+l}, \quad  k, l\geq 0,
$$
with
$$\alpha_{2n, m}=\sqrt{C_{2n+m}^{2n}}\int_{-\frac\pi{4}}^{\frac\pi{4}}\beta(\theta)(\sin \theta)^{2n} (\cos \theta)^md\theta, \quad n\geq 1, \ m \geq 0,$$
$$\alpha_{0, m}=\int_{-\frac\pi{4}}^{\frac\pi{4}}\beta(\theta)\big((\cos \theta)^m-1\big)d\theta, \quad m \geq 1; \quad \alpha_{0,0}=\alpha_{2n+1,m}=0, \quad n, m \geq 0,$$
where $C_n^k=\frac{n!}{k!(n-k)!}$ stands for the binomial coefficients.
\end{lemma}


\begin{proof}
We deduce from (\ref{cv1}), (\ref{Ar1}) and (\ref{Ar2}) that for all $n \geq 0$,
$$\mu^{1/2}(v)\psi_n(v)=\frac{1}{(4\pi)^{\frac{1}{4}}}\phi_n\Big(\frac{v}{\sqrt{2}}\Big)e^{-\frac{v^2}{4}}=\frac{(-1)^n}{\sqrt{2^{n+1} n! \pi}} \Big[\frac{d^n}{dx^n}(e^{-x^2})\Big]\Big|_{x=\frac{v}{\sqrt{2}}}
=\frac{(-1)^ni^n}{\sqrt{2 n! \pi}}D_v^n\big(e^{-\frac{v^2}{2}}\big).$$
It follows that
$$\widehat{\mu^{1/2}\psi_n}(\xi)=\frac{(-1)^ni^n}{\sqrt{2 n! \pi}}\xi^n\widehat{e^{-\frac{v^2}{2}}}(\xi)=\frac{(-1)^ni^n}{\sqrt{n!}}\xi^ne^{-\frac{\xi^2}{2}},$$
since
$$\widehat{\big(e^{-\frac{\alpha}{2} \val v^2}\big)}(\xi)=\int_{\R^{d}}e^{-\frac{\alpha}{2} \val v^2}e^{-iv\cdot\xi}dv=\frac{(2\pi)^{\frac{d}{2}}}{\alpha^{\frac{d}{2}}}e^{-\frac{\val\xi^2}{2\alpha}}, \quad \alpha>0.$$
We have
\begin{equation}\label{yu1}
\widehat{\breve{\overbrace{\mu^{1/2}\psi_{2n}}}}(\xi)=\widehat{\mu^{1/2}\psi_{2n}}(\xi), \qquad \widehat{\breve{\overbrace{\mu^{1/2}\psi_{2n+1}}}}(\xi)=0,
\end{equation}
where $\breve{f}$ stands for the even part of the function $f$, 
$$\breve{f}(x)=\frac{1}{2}\big(f(x)+f(-x)\big),$$
since the function $\psi_{n}$ has the same parity than the integer $n$.
We notice that
\begin{align}\label{fh1}
\widehat{\mu^{1/2}\psi_n}(\xi \sin \theta)\widehat{\mu^{1/2}\psi_m}(\xi \cos \theta)= & \ \frac{(-1)^{n+m}i^{n+m}}{\sqrt{n!m!}}\xi^{n+m}(\sin \theta)^n (\cos \theta)^m e^{-\frac{\xi^2}{2}}\\
= & \ \sqrt{C_{n+m}^n}(\sin \theta)^n (\cos \theta)^m\widehat{\mu^{1/2}\psi_{n+m}}(\xi) \notag
\end{align}
and
\begin{equation}\label{fh2}
\widehat{\mu^{1/2}\psi_n}(0)=(\mu^{1/2},\psi_n)_{L^2}=(\psi_0,\psi_n)_{L^2}=\delta_{n,0},
\end{equation}
where $\delta_{i,j}$ stands for the Kronecker delta.
It follows from (\ref{yu1}), the Bobylev formula (\ref{cl11}) and the Fourier inversion formula that
\begin{multline*}
\Gamma(\psi_{2n},\psi_m)=\mu^{-1/2}K(\mu^{1/2}\psi_{2n},\mu^{1/2}\psi_m)\\
=\frac{\mu^{-1/2}}{2 \pi}
\int_{\rr}\int_{-\frac\pi{4}}^{\frac\pi{4}}
\beta(\theta) \left[\widehat{\mu^{1/2}\psi_{2n}}(\eta \sin{\theta})\widehat{\mu^{1/2}\psi_m}(\eta\cos{\theta})-\widehat{\mu^{1/2}\psi_{2n}}(0)\widehat{\mu^{1/2}\psi_m}(\eta)\right]e^{iv\eta}d\eta d\theta.
\end{multline*}
When $n=0$, we deduce from (\ref{fh1}) and (\ref{fh2}) that
\begin{align*}
\Gamma(\psi_0,\psi_m)= &  \frac{\mu^{-1/2}}{2 \pi}\int_{\rr}\int_{-\frac\pi{4}}^{\frac\pi{4}}\hspace{-3mm}\beta(\theta) \big((\cos{\theta})^m-1\big)\widehat{\mu^{1/2}\psi_m}(\eta)e^{iv\eta}d\eta d\theta\\
= & \Big(\int_{-\frac\pi{4}}^{\frac\pi{4}}\beta(\theta) \big((\cos{\theta})^m-1\big)d\theta\Big)\psi_m.
\end{align*}
When $n \geq 1$, we deduce from (\ref{fh1}) and (\ref{fh2}) that
$$\Gamma(\psi_{2n},\psi_m)=
\sqrt{C_{2n+m}^{2n}}\Big(\int_{-\frac\pi{4}}^{\frac\pi{4}}
\beta(\theta)(\sin \theta)^{2n} (\cos \theta)^md\theta\Big)\psi_{2n+m}.$$
On the other hand, it follows from (\ref{yu1}), (\ref{fh2}), the Bobylev formula (\ref{cl11}) and the Fourier inversion formula that
\begin{multline*}
\Gamma(\psi_{2n+1},\psi_m)=\mu^{-1/2}K(\mu^{1/2}\psi_{2n+1},\mu^{1/2}\psi_m)=\\
\frac{\mu^{-1/2}}{2 \pi}
\int_{\rr}\int_{-\frac\pi{4}}^{\frac\pi{4}}
\beta(\theta) \Big[\widehat{\breve{\overbrace{\mu^{1/2}\psi_{2n+1}}}}(\eta \sin{\theta})\widehat{\mu^{1/2}\psi_m}(\eta\cos{\theta})-\widehat{\mu^{1/2}\psi_{2n+1}}(0)\widehat{\mu^{1/2}\psi_m}(\eta)\Big]e^{iv\eta}d\eta d\theta=0,
\end{multline*}
when $n,m \geq 0$. This ends the proof of Lemma~\ref{proposition1a}.
\end{proof}

\subsection{Uniform weighted trilinear estimates for the Kac collision operator}

We begin by proving some instrumental estimates.
We define
$$\Lambda_{n,m}=\int_{-\frac{\pi}{4}}^{\frac{\pi}{4}}\beta(\theta)(\sin \theta)^{2n}(\cos \theta)^{m}d\theta, \quad n \geq 1,\ m \geq 0.$$
We notice that
\begin{equation}\label{ekkk1}
\forall n \geq 1, m \geq 0, \quad \frac{1}{\sqrt{2}}\Lambda_{n,2m} \leq \Lambda_{n,2m+1} \leq \Lambda_{n,2m}.
\end{equation}
The following lemma extends the estimates obtained in~\cite{LMPX3} (Lemma~3.4):


\begin{lemma}\label{lemkk0}
We assume that the cross section satisfies \eqref{sing} with $0<s<1$. Then, there exists a positive constant $C>0$ such that for all $n \geq 1$, $m \geq 0$,
$$0 \leq \alpha_{2n,m}=\sqrt{C_{2n+m}^{2n}}\int_{-\frac{\pi}{4}}^{\frac{\pi}{4}}\beta(\theta)(\sin \theta)^{2n}(\cos \theta)^{m}d\theta \leq  \frac{C}{n^{\frac{3}{4}}}\tilde{\mu}_{n,m},$$
where
$$\tilde{\mu}_{n,m}=\Big(1+\frac{m}{n}\Big)^{s}\Big(1+\frac{n}{m+1}\Big)^{\frac{1}{4}}.$$
\end{lemma}


\begin{proof}
Lemma~\ref{lemkk0} is a direct consequence of the following estimates:
\begin{itemize}
\item[$(i)$] $ \displaystyle \alpha_{2n,2m} \lesssim  \frac{1}{n^{\frac{3}{4}}} \Big(1+\frac{m}{n}\Big)^{s},$
when $m \gg 1$, $n \gg 1$
\item[$(ii)$] $ \displaystyle  \alpha_{2n,2m+1} \lesssim  \frac{1}{n^{\frac{3}{4}}} \Big(1+\frac{m}{n}\Big)^{s}\Big(1+\frac{n}{m}\Big)^{\frac{1}{4}},$
when $m \gg 1$, $n \gg 1$
\item[$(iii)$] $\alpha_{2n,m} \lesssim  m^{s},$
when $m \gg 1$, $1 \leq n \leq n_0$
\item[$(iv)$] $\displaystyle \alpha_{2n,2m} \lesssim  \frac{1}{n}, \quad  \alpha_{2n,2m+1} \lesssim  \frac{1}{n^{\frac{1}{2}}},$
when $0 \leq m \leq m_0$, $n \gg 1$
\end{itemize}
In order to establish these estimates, we beginning by noticing from (\ref{sing}) that
$$\Lambda_{n,2m} \approx \int_{0}^{\frac{\pi}{4}}(\sin \theta)^{2n-1-2s}(\cos \theta)^{2m+1}d\theta.$$
By using the substitution rule with $t=\sin^2 \theta$, we obtain that
$$\int_{0}^{\frac{\pi}{4}}(\sin \theta)^{2n-1-2s}(\cos \theta)^{2m+1}d\theta=\frac{1}{2}\int_0^{\frac{1}{2}}t^{n-1-s}(1-t)^mdt.$$
This implies that
$$ \Lambda_{n,2m} \approx \int_0^{\frac{1}{2}}t^{n-1-s}(1-t)^mdt.$$
By recalling the identity satisfied by the beta function
$$B(x,y)=\int_0^1t^{x-1}(1-t)^{y-1}dt=\frac{\Gamma(x)\Gamma(y)}{\Gamma(x+y)}, \quad \textrm{Re }x>0, \ \textrm{Re }y>0,$$
we obtain that for all $n \geq 1$, $m \geq 0$,
$$\int_0^{\frac{1}{2}}t^{n-1-s}(1-t)^mdt \leq B(n-s,m+1)=\frac{\Gamma(n-s)\Gamma(m+1)}{\Gamma(m+n+1-s)}.$$
It follows that
\begin{equation}\label{Ar0}
\sqrt{C_{2n+2m}^{2n}}\Lambda_{n,2m} \lesssim \sqrt{\frac{(2n+2m)!}{(2n)!(2m)!}}\frac{\Gamma(n-s)\Gamma(m+1)}{\Gamma(m+n+1-s)}.
\end{equation}
By using the Stirling equivalent
\begin{equation}\label{Ar0.1}
\Gamma(x+1) 
\begin{matrix}{\scriptscriptstyle~}\\\sim\\{x \to +\infty} 
\end{matrix} \sqrt{2\pi x}\Big(\frac{x}{e}\Big)^x, \quad \Gamma(n+1)=n!,
\end{equation}
we deduce that
\begin{multline*}
\sqrt{C_{2n+2m}^{2n}}\Lambda_{n,2m} \lesssim \Big(\frac{n+m}{n m}\Big)^{\frac{1}{4}}\Big(\frac{2n+2m}{e}\Big)^{n+m}\Big(\frac{e}{2n}\Big)^{n}\Big(\frac{e}{2m}\Big)^{m}\\
\times \sqrt{\frac{n m}{n+m}}\Big(\frac{n-s-1}{e}\Big)^{n-s-1}\Big(\frac{m}{e}\Big)^{m}\Big(\frac{e}{m+n-s}\Big)^{m+n-s},
\end{multline*}
when $m \gg 1$, $n \gg 1$. It follows that
\begin{align}\label{Ar3}
\sqrt{C_{2n+2m}^{2n}}\Lambda_{n,2m} \lesssim & \  \Big(\frac{n m}{n+m}\Big)^{\frac{1}{4}}\frac{(n+m)^{n+m}(n-s-1)^{n-s-1}}{(m+n-s)^{m+n-s}n^n} \\ \notag
 \lesssim & \ \Big(\frac{n m}{n+m}\Big)^{\frac{1}{4}}\frac{(m+n)^{s}}{(n-s-1)^{s+1}}\Big(1-\frac{s+1}{n}\Big)^{n}\Big(1+\frac{s}{m+n-s}\Big)^{m+n-s}\\ \notag
  \lesssim & \ m^{\frac{1}{4}}\frac{(m+n)^{s-\frac{1}{4}}}{n^{s+\frac{3}{4}}}  \lesssim  \frac{1}{n^{\frac{3}{4}}} \Big(1+\frac{m}{n}\Big)^{s},
\end{align}
when $m \gg 1$, $n \gg 1$, since
\begin{equation}\label{ekkk0}
\forall r \geq 1, \forall x >-r, \quad \Big(1+\frac{x}{r}\Big)^r \leq e^x.
\end{equation}
By using that
\begin{equation}\label{ekkk2}
C_{2n+2m+1}^{2n}=\frac{2n+2m+1}{2m+1}C_{2n+2m}^{2n},
\end{equation}
we obtain from (\ref{ekkk1}), (\ref{Ar3}) and (\ref{ekkk2}) that
$$\sqrt{C_{2n+2m+1}^{2n}}\Lambda_{n,2m+1} \lesssim m^{\frac{1}{4}}\frac{(m+n)^{s-\frac{1}{4}}}{n^{s+\frac{3}{4}}}  \sqrt{\frac{n+m}{m}}
\lesssim  \frac{1}{n^{\frac{3}{4}}}\Big(1+\frac{m}{n}\Big)^{s} \Big(1+\frac{n}{m}\Big)^{\frac{1}{4}},$$
when $m \gg 1$, $n \gg 1$. When $m \gg 1$, $1 \leq n \leq n_0$, it follows from (\ref{Ar0}) that
$$\sqrt{C_{2n+2m}^{2n}}\Lambda_{n,2m} \lesssim \sqrt{\frac{(2n+2m)!}{(2m)!}}\frac{\Gamma(m+1)}{\Gamma(m+n+1-s)}.$$
By using the Stirling equivalent (\ref{Ar0.1}), this implies that
\begin{multline*}
\sqrt{C_{2n+2m}^{2n}}\Lambda_{n,2m} \lesssim \Big(\frac{n+m}{ m}\Big)^{\frac{1}{4}}\Big(\frac{2n+2m}{e}\Big)^{n+m}\Big(\frac{e}{2m}\Big)^{m}\\
\times \sqrt{\frac{m}{n+m-s}}\Big(\frac{m}{e}\Big)^{m}\Big(\frac{e}{m+n-s}\Big)^{m+n-s},
\end{multline*}
when $m \gg 1$, $1 \leq n \leq n_0$. We deduce from (\ref{ekkk0}) that
$$\sqrt{C_{2n+2m}^{2n}}\Lambda_{n,2m} \lesssim   \frac{(n+m)^{n+m}}{(n+m-s)^{m+n-s}}=\Big(1+\frac{s}{n+m-s}\Big)^{m+n-s}(n+m)^{s} \lesssim m^s,$$
when $m \gg 1$, $1 \leq n \leq n_0$. It follows from (\ref{ekkk1}) and (\ref{ekkk2}) that
$$\sqrt{C_{2n+2m+1}^{2n}}\Lambda_{n,2m+1} \lesssim  m^s,$$
when $m \gg 1$, $1 \leq n \leq n_0$.
When $0 \leq m \leq m_0$, $n \gg 1$, it follows from (\ref{Ar0}) that
$$\sqrt{C_{2n+2m}^{2n}}\Lambda_{n,2m} \lesssim \sqrt{\frac{(2n+2m)!}{(2n)!}}\frac{\Gamma(n-s)}{\Gamma(m+n+1-s)}.$$
By using the Stirling equivalent (\ref{Ar0.1}), this implies that
\begin{multline*}
\sqrt{C_{2n+2m}^{2n}}\Lambda_{n,2m} \lesssim \Big(\frac{n+m}{n}\Big)^{\frac{1}{4}}\Big(\frac{2n+2m}{e}\Big)^{n+m}\Big(\frac{e}{2n}\Big)^{n}\\
\times \sqrt{\frac{n-s-1}{n+m-s}}\Big(\frac{n-s-1}{e}\Big)^{n-s-1}\Big(\frac{e}{m+n-s}\Big)^{m+n-s},
\end{multline*}
when $0 \leq m \leq m_0$, $n \gg 1$. We deduce from (\ref{ekkk0}) that
\begin{multline*}
\sqrt{C_{2n+2m}^{2n}}\Lambda_{n,2m} \lesssim  \frac{(n+m)^{n+m}(n-s-1)^{n-s-1}}{n^n(m+n-s)^{m+n-s}} \\ \lesssim   \frac{(n+m)^{m}}{(n-s-1)^{m+1}}
\Big(1+\frac{m}{n}\Big)^{n}\Big(1-\frac{m+1}{m+n-s}\Big)^{m+n-s} \lesssim \frac{1}{n},
\end{multline*}
when $0 \leq m \leq m_0$, $n \gg 1$. It follows from (\ref{ekkk1}) and (\ref{ekkk2}) that
$$\sqrt{C_{2n+2m+1}^{2n}}\Lambda_{n,2m+1} \lesssim \frac{1}{n^{\frac{1}{2}}},$$
when $0 \leq m \leq m_0$, $n \gg 1$.
This ends the proof of Lemma~\ref{lemkk0}.
\end{proof}

In order to estimate the non-linear collision operator, we shall use the following lemma:

\begin{lemma}\label{L1}
Let $r>1/2$. Then, there exists a positive constant $C_r>0$ such that for all $f,g \in \mathscr{S}(\rr_x)$, $t \geq 0$, $0<\delta \leq 1$, $m,n \geq 0$,
$$\big\|\mathscr{M}_{\delta, m+n}(t)\big(\big[\big(\mathscr{M}_{\delta,m}(t)\big)^{-1}f\big]\big[\big(\mathscr{M}_{\delta,n}(t)\big)^{-1}g\big]\big)\big\|_{H^r} \leq C_r \|f\|_{H^r}\|g\|_{H^r},$$
with the Fourier multiplier
$$\mathscr{M}_{\delta,n}(t)=\frac{\exp\big(t\big(\sqrt{n+\frac{1}{2}}+\langle D_x \rangle\big)^{\frac{2s}{2s+1}}\big)}{1+\delta\exp\big(t\big(\sqrt{n+\frac{1}{2}}+\langle D_x \rangle\big)^{\frac{2s}{2s+1}}\big)},$$
where $\|\cdot\|_{H^r}$ stands for the Sobolev norm $H^r(\rr_x)$.
\end{lemma}

\noindent
\begin{proof}
We begin by noticing that the operator $\mathscr{M}_{\delta,n}(t)$ is a bounded isomorphism of $L^2(\rr_x)$ such that
$$(\mathscr{M}_{\delta,n}(t))^{-1}=\frac{1+\delta\exp\big(t\big(\sqrt{n+\frac{1}{2}}+\langle D_x \rangle\big)^{\frac{2s}{2s+1}}\big)}{\exp\big(t\big(\sqrt{n+\frac{1}{2}}+\langle D_x \rangle\big)^{\frac{2s}{2s+1}}\big)}.$$
Setting
$$h=\mathscr{M}_{\delta, m+n}(t)\big(\big[\big(\mathscr{M}_{\delta,m}(t)\big)^{-1}f\big]\big[\big(\mathscr{M}_{\delta,n}(t)\big)^{-1}g\big]\big),$$
we have
\begin{multline}\label{lk0}
\widehat{h}(\xi)=  \frac{\exp\big(t\big(\sqrt{m+n+\frac{1}{2}}+\langle \xi \rangle\big)^{\frac{2s}{2s+1}}\big)}{1+\delta\exp\big(t\big(\sqrt{m+n+\frac{1}{2}}+\langle \xi \rangle\big)^{\frac{2s}{2s+1}}\big)}\mathcal{F}\big(
\big[\big(\mathscr{M}_{\delta,m}(t)\big)^{-1}f\big]\big[\big(\mathscr{M}_{\delta,n}(t)\big)^{-1}g\big]\big)\\
=  \frac{1}{2\pi} \frac{\exp\big(t\big(\sqrt{m+n+\frac{1}{2}}+\langle \xi \rangle\big)^{\frac{2s}{2s+1}}\big)}{1+\delta\exp\big(t\big(\sqrt{m+n+\frac{1}{2}}+\langle \xi \rangle\big)^{\frac{2s}{2s+1}}\big)}\mathcal{F}\big(\big(\mathscr{M}_{\delta,m}(t)\big)^{-1}f\big) \ast \mathcal{F}\big(\big(\mathscr{M}_{\delta,n}(t)\big)^{-1}g\big),
\end{multline}
where $\mathcal{F}$ denotes the Fourier transform.
For all $0<\sigma<1$,
we notice that
$$\forall 0 \leq \theta \leq 1, \quad \phi(\theta)=(1-\theta)^{\sigma}+\theta^{\sigma} \geq 1,$$
since
$$\phi(0)=\phi(1)=1, \quad \forall 0 < \theta < 1, \quad \phi''(\theta)=-\sigma(1-\sigma)\big((1-\theta)^{\sigma-2}+\theta^{\sigma-2}\big) \leq 0.$$
This implies that
\begin{equation}\label{cl3}
\forall 0<\sigma<1, \forall a,b \geq 0, \quad a^{\sigma}+b^{\sigma} \geq (a+b)^{\sigma}.
\end{equation}
We also notice that for all $x,y \in \rr^d$, with $d \geq 1$,
\begin{equation}\label{qa0}
\langle x+y\rangle=(1+|x+y|^2)^{\frac{1}{2}}=\|(1,x+y)\|_{2} \leq \|(1/2,x)\|_{2}+\|(1/2,y)\|_{2} \leq \langle x \rangle + \langle y \rangle,
\end{equation}
where $\|\cdot\|_2$ stands for the Euclidean norm on $\rr^{d+1}$.
It follows from (\ref{cl3}) and (\ref{qa0}) that for all $m,n \geq 0$, $\eta,\xi \in \rr$,
\begin{multline}\label{lk1}
\Big(\sqrt{m+n+\frac{1}{2}}+\langle \xi \rangle\Big)^{\frac{2s}{2s+1}} \leq \Big(\sqrt{m+\frac{1}{2}}+\sqrt{n+\frac{1}{2}}+\langle \eta \rangle+\langle \xi-\eta \rangle\Big)^{\frac{2s}{2s+1}} \\
\leq \Big(\sqrt{m+\frac{1}{2}}+\langle \eta \rangle\Big)^{\frac{2s}{2s+1}}+\Big(\sqrt{n+\frac{1}{2}}+\langle \xi-\eta \rangle\Big)^{\frac{2s}{2s+1}}.
\end{multline}
We notice that the increasing function
$$F(x)=\frac{e^x}{1+\delta e^x},$$
satisfies the inequality
\begin{equation}\label{lk2}
\forall x,y \geq 0, \quad F(x+y) \leq 3 F(x)F(y),
\end{equation}
since
$$\forall x,y \geq 0, \quad \frac{F(x+y)}{F(x)F(y)}=\delta+\frac{1-\delta}{1+\delta e^{x+y}}+\frac{\delta(e^x+e^y)}{1+\delta e^{x+y}} \leq 1-\delta+\delta+\frac{1}{e^x}+\frac{1}{e^y} \leq 3.$$
It follows from (\ref{lk1}) and (\ref{lk2}) that
\begin{align}\label{lk3}
& \ \frac{\exp\big(t\big(\sqrt{m+n+\frac{1}{2}}+\langle \xi \rangle\big)^{\frac{2s}{2s+1}}\big)}{1+\delta \exp\big(t\big(\sqrt{m+n+\frac{1}{2}}+\langle \xi \rangle\big)^{\frac{2s}{2s+1}}\big)} \\ \notag
\leq & \  \frac{ \exp\big(t\big(\sqrt{m+\frac{1}{2}}+\langle \eta \rangle\big)^{\frac{2s}{2s+1}}+t\big(\sqrt{n+\frac{1}{2}}+\langle \xi-\eta \rangle\big)^{\frac{2s}{2s+1}}\big)}{1+\delta  \exp\big(t\big(\sqrt{m+\frac{1}{2}}+\langle \eta \rangle\big)^{\frac{2s}{2s+1}}+t\big(\sqrt{n+\frac{1}{2}}+\langle \xi-\eta \rangle\big)^{\frac{2s}{2s+1}}\big)}\\ \notag
\leq & \ 3\frac{ \exp\big(t\big(\sqrt{m+\frac{1}{2}}+\langle \eta \rangle\big)^{\frac{2s}{2s+1}}\big)}{1+\delta  \exp\big(t\big(\sqrt{m+\frac{1}{2}}+\langle \eta \rangle\big)^{\frac{2s}{2s+1}}\big)}\times \frac{ \exp\big(t\big(\sqrt{n+\frac{1}{2}}+\langle \xi-\eta \rangle\big)^{\frac{2s}{2s+1}}\big)}{1+\delta  \exp\big(t\big(\sqrt{n+\frac{1}{2}}+\langle \xi-\eta \rangle\big)^{\frac{2s}{2s+1}}\big)}.
\end{align}
We deduce from (\ref{lk0}) and (\ref{lk3}) that
\begin{equation}\label{cl7}
|\widehat{h}(\xi)| \leq \frac{3}{2\pi}\int_{\rr}|\hat{f}(\eta)||\hat{g}(\xi-\eta)|d\eta= \frac{3}{2\pi}(|\hat{f}| \ast |\hat{g}|)(\xi).
\end{equation}
We notice that $|\hat{f}|, |\hat{g}| \in L^1(\rr_{\xi}) \cap L^2(\rr_{\xi})$, since $f, g \in \mathscr{S}(\rr_x)$. This implies that
$$F=|\hat{f}| \ast |\hat{g}| \in L^1(\rr_{\xi}) \cap L^2(\rr_{\xi}).$$
The Sobolev space $H^r(\rr)$, with $r>1/2$, is an algebra for the usual product
\begin{equation}\label{cl6}
\forall r>1/2, \exists C_r>0, \forall f,g \in H^r(\rr_x), \quad \|fg\|_{H^r} \leq C_r\|f\|_{H^r}\|g\|_{H^r}.
\end{equation}
We deduce from (\ref{cl7}) and (\ref{cl6}) that
\begin{multline*}
\|h\|_{H^r} \leq \frac{3}{(2\pi)^{\frac{3}{2}}}\|\langle \xi \rangle^r(|\hat{f}| \ast |\hat{g}|) \|_{L^2}=\frac{3}{(2\pi)^{\frac{3}{2}}}\|\langle \xi \rangle^r\mathcal{F}\mathcal{F}^{-1}(|\hat{f}| \ast |\hat{g}|) \|_{L^2}
=\frac{3}{2\pi}\|\mathcal{F}^{-1}(|\hat{f}| \ast |\hat{g}|) \|_{H^r}\\
=3\|\mathcal{F}^{-1}(|\hat{f}|)\mathcal{F}^{-1}(|\hat{g}|) \|_{H^r} \leq 3C_r\|\mathcal{F}^{-1}(|\hat{f}|)\|_{H^r}\|\mathcal{F}^{-1}(|\hat{g}|) \|_{H^r}=3C_r\|f\|_{H^r}\|g \|_{H^r},
\end{multline*}
since $\|\mathcal{F}^{-1}(\val{\hat{f}} )\|_{H^r}=\|\valjp{\xi}^{r}\val{\hat{f}(\xi)}\|_{L^{2}}=\|f\|_{H^{r}}$. This ends the proof of Lemma~\ref{L1}.
\end{proof}

By elaborating on the result of the previous lemma, we may adapt the proof of Lemma~3.5 in~\cite{LMPX3} to derive the following trilinear estimate on the non-linear term (\ref{wr1}):

\begin{lemma}\label{proposition222}
Let $r>1/2$. Then, there exists a positive constant $C_r>0$ such that for all $f,g,h \in \mathscr{S}(\rr_{x,v}^2)$, $t \geq 0$, $0 \leq \delta_1 \leq 1$, $0<\delta_2 \leq 1$, $j_1,j_2 \geq 0$ with $j_1+j_2 \leq 1$,
$$\big|(\Gamma(f,g),h)_{L^2(\rr_{x,v}^2)}\big| \leq C_r \|f\|_{(r,0)}\|\mathcal{H}^\frac{s}{2}g\|_{L^2(\rr_{x,v}^2)}\|\mathcal{H}^\frac{s}{2}h\|_{L^2(\rr_{x,v}^2)},$$
$$\big|(\Gamma(\langle D_x \rangle f,g),h)_{L^2(\rr_{x,v}^2)}\big| \leq C_r \|f\|_{(1,0)}\|\mathcal{H}^\frac{s}{2}g\|_{(r,0)}\|\mathcal{H}^\frac{s}{2}h\|_{L^2(\rr_{x,v}^2)},$$
\begin{multline*}
\Bigg|\Bigg(\big(1+\delta_1 \sqrt{\mathcal{H}}+\delta_1\langle D_x \rangle\big)^{-1} M_{\delta_2}(t)\Gamma\Big(\big(M_{\delta_2}(t)\big)^{-1}\big(1+\delta_1 \sqrt{\mathcal{H}}\big)^{j_1}f,\big(M_{\delta_2}(t)\big)^{-1}\big(1+\delta_1 \sqrt{\mathcal{H}}\big)^{j_2}g\Big),h\Bigg)_{(r,0)}\Bigg|  \\
\leq C_r \|f\|_{(r,0)}\|\mathcal{H}^\frac{s}{2}g\|_{(r,0)}\|\mathcal{H}^\frac{s}{2}h\|_{(r,0)},
\end{multline*}
with
\begin{equation}\label{wr3}
M_{\delta_2}(t)=\frac{\exp\big(t\big(\sqrt{\mathcal{H}}+\langle D_x \rangle\big)^{\frac{2s}{2s+1}}\big)}{1+\delta_2\exp\big(t\big(\sqrt{\mathcal{H}}+\langle D_x \rangle\big)^{\frac{2s}{2s+1}}\big)}=\sum_{n=0}^{+\infty} \mathscr{M}_{\delta_2,n}(t) \mathbb P_{n} ,
\end{equation}
where $\mathbb P_{n}$ denote the orthogonal projections onto the Hermite basis described in Section \ref{6.sec.harmo}. In particular, we also have for all $f,g,h \in \mathscr{S}(\rr_{x,v}^2)$,
$$\big|(\Gamma(f,g),h)_{(r,0)}\big| \leq C_r \|f\|_{(r,0)}\|\mathcal{H}^\frac{s}{2}g\|_{(r,0)}\|\mathcal{H}^\frac{s}{2}h\|_{(r,0)}.$$
\end{lemma}

\begin{proof}
Let $r>1/2$ and $f,g,h \in \mathscr{S}(\rr_{x,v}^2)$. We decompose these functions into the Hermite basis in the velocity variable
$$f(x,v)=\sum_{n=0}^{+\infty}f_n(x) \psi_n(v), \quad g(x,v)=\sum_{n=0}^{+\infty}g_n(x) \psi_n(v), \quad h(x,v)=\sum_{n=0}^{+\infty}h_n(x) \psi_n(v),$$
with
$$f_n(x)=(f(x,\cdot),\psi_n)_{L^2(\rr_v)}, \quad g_n(x)=(g(x,\cdot),\psi_n)_{L^2(\rr_v)}, \quad h_n(x)=(h(x,\cdot),\psi_n)_{L^2(\rr_v)}.$$
We notice that
\begin{align}\label{sake-1}
\|f\|_{(r,0)}=\Big(\sum_{n=0}^{+\infty}\|f_n\|_{H^r(\rr_x)}^2\Big)^{\frac{1}{2}}, \quad \|\mathcal{H}^mf\|_{(r,0)}=\Big(\sum_{n=0}^{+\infty}\Big(n+\frac{1}{2}\Big)^{2m}\|f_n\|_{H^r(\rr_x)}^2\Big)^{\frac{1}{2}},
\end{align}
when $m \in \rr$.
We deduce from (\ref{wr3}) and Lemma~\ref{proposition1a} that for all $t \geq 0$, $0 \leq \delta_1 \leq 1$, $0<\delta_2 \leq 1$, $j_1,j_2 \geq 0$ with $j_1+j_2 \leq 1$,
\begin{align*}
& \big(\big(1+\delta_1 \sqrt{\mathcal{H}}+\delta_1\langle D_x \rangle\big)^{-1}M_{\delta_2}(t)\Gamma\big(\big(M_{\delta_2}(t)\big)^{-1}\big(1+\delta_1 \sqrt{\mathcal{H}}\big)^{j_1}f,\big(M_{\delta_2}(t)\big)^{-1}\big(1+\delta_1 \sqrt{\mathcal{H}}\big)^{j_2}g\big),h\big)_{(r,0)}\\ \notag
= &  \int_{\rr^2}\overline{\langle D_x \rangle^r h(x,v)}\langle D_x \rangle^r\frac{M_{\delta_2}(t)}{1+\delta_1 \sqrt{\mathcal{H}}+\delta_1\langle D_x \rangle}
   \Big(\sum_{n=0}^{+\infty}\Big(\sum_{\substack{k+l=n \\ k,l \geq 0}}\alpha_{k,l} \\
&  \qquad \quad \times   \Big[\big(\mathscr{M}_{\delta_2,k}(t)\big)^{-1}\Big(1+\delta_1 \sqrt{k+\frac{1}{2}}\Big)^{j_1}f_k\Big]\Big[\big(\mathscr{M}_{\delta_2,l}(t)\big)^{-1}\Big(1+\delta_1 \sqrt{l+\frac{1}{2}}\Big)^{j_2}g_l\Big]\Big)\psi_n(v)\Big)dxdv.
\end{align*}
This implies that
\begin{align*}
&  \big(\big(1+\delta_1 \sqrt{\mathcal{H}}+\delta_1\langle D_x \rangle\big)^{-1}M_{\delta_2}(t)\Gamma\big(\big(M_{\delta_2}(t)\big)^{-1}\big(1+\delta_1 \sqrt{\mathcal{H}}\big)^{j_1}f,\big(M_{\delta_2}(t)\big)^{-1}\big(1+\delta_1 \sqrt{\mathcal{H}}\big)^{j_2}g\big),h\big)_{(r,0)}\\ \notag
& =  \sum_{n=0}^{+\infty}\sum_{\substack{k+l=n \\ k,l \geq 0}}\alpha_{k,l} \int_{\rr}   \frac{(1+\delta_1 \sqrt{k+\frac{1}{2}})^{j_1}(1+\delta_1 \sqrt{l+\frac{1}{2}})^{j_2} }{1+\delta_1 \sqrt{n+\frac{1}{2}}+\delta_1\langle D_x \rangle}\\
&  \quad \qquad \qquad \qquad \times  \langle D_x \rangle^r\mathscr{M}_{\delta_2,n}(t)
\big(
\big[\big(\mathscr{M}_{\delta_2,k}(t)\big)^{-1}f_k\big]\big[\big(\mathscr{M}_{\delta_2,l}(t)\big)^{-1}g_l\big]\big)\overline{\langle D_x \rangle^r h_n}dx.
\end{align*}
By using that
$$\Big\|\Big(1+\delta_1 \sqrt{n+\frac{1}{2}}+\delta_1\langle D_x \rangle\Big)^{-1}\Big(1+\delta_1 \sqrt{k+\frac{1}{2}}\Big)^{j_1}\Big(1+\delta_1 \sqrt{l+\frac{1}{2}}\Big)^{j_2}\Big\|_{\mathcal{L}(L^2(\rr_{x}))} \leq 1,$$
since $k+l=n$ and $j_1,j_2 \geq 0$ with $j_1+j_2 \leq 1$, it follows that
\begin{align*}
&  \big|\big(\big(1+\delta_1 \sqrt{\mathcal{H}}+\delta_1\langle D_x \rangle\big)^{-1}M_{\delta_2}(t)\Gamma\big(\big(M_{\delta_2}(t)\big)^{-1}\big(1+\delta_1 \sqrt{\mathcal{H}}\big)^{j_1}f,\big(M_{\delta_2}(t)\big)^{-1}\big(1+\delta_1 \sqrt{\mathcal{H}}\big)^{j_2}g\big),h\big)_{(r,0)}\big|\\ \notag
&  \leq  \sum_{n=0}^{+\infty}\sum_{\substack{k+l=n \\ k,l \geq 0}}|\alpha_{k,l}|\big\|\mathscr{M}_{\delta_2,k+l}(t)\big(\big[
\big(\mathscr{M}_{\delta_2,k}(t)\big)^{-1}f_k\big]
\big[\big(\mathscr{M}_{\delta_2,l}(t)\big)^{-1}g_l\big]\big)\big\|_{H^r(\rr_x)}\|h_n\|_{H^r(\rr_x)}.
\end{align*}
We deduce from Lemma~\ref{L1} that for all $t \geq 0$, $0 \leq \delta_1 \leq 1$, $0<\delta_2 \leq 1$, $j_1,j_2 \geq 0$ with $j_1+j_2 \leq 1$,
\begin{align}\label{cp3.2}
\Delta_1=&\Big|\Big(\frac{M_{\delta_2}(t)}{1+\delta_1 \sqrt{\mathcal{H}}+\delta_1\langle D_x \rangle}\Gamma\Big(\frac{(1+\delta_1 \sqrt{\mathcal{H}})^{j_1}}{M_{\delta_2}(t)}f,\frac{(1+\delta_1 \sqrt{\mathcal{H}})^{j_2}}{M_{\delta_2}(t)}g\Big),h\Big)_{(r,0)}\Big|\\ \notag
\leq  & C_r\sum_{n=0}^{+\infty}\sum_{\substack{k+l=n \\ k,l \geq 0}}|\alpha_{k,l}|\|f_k\|_{H^r(\rr_x)}\|g_l\|_{H^r(\rr_x)}\|h_n\|_{H^r(\rr_x)}.
\end{align}
By noticing from Lemma~\ref{proposition1a} that
$$\alpha_{2n+1,m}=0, \quad n, m \geq 0,$$
we obtain that
\begin{equation}\label{cp3.2n}
\Delta_1 \leq C_r\sum_{n=0}^{+\infty}\sum_{\substack{2k+l=n \\ k,l \geq 0}}|\alpha_{2k,l}|\|f_{2k}\|_{H^r(\rr_x)}\|g_l\|_{H^r(\rr_x)}\|h_n\|_{H^r(\rr_x)}.
\end{equation}
Under the assumption (\ref{sing}), we recall from formula (A.17) in~\cite{LMPX1} (Section~A.4.2) that
$$\int_{-\frac\pi{4}}^{\frac\pi{4}}\beta(\theta) \big(1-(\cos{\theta})^n\big)d\theta \sim \frac{2^{1+s}}{s}\mathbf{\Gamma}(1-s) n^s,$$
when $n \to +\infty$, where $\mathbf{\Gamma}$ denotes the Gamma function. It follows from Lemma~\ref{proposition1a}, Lemma~\ref{lemkk0} and (\ref{cp3.2n}) that there exists a positive constant $c_1>0$ such that for all $f,g,h \in \mathscr{S}(\rr_{x,v}^2)$, $t \geq 0$, $0 \leq \delta_1 \leq 1$, $0 <\delta_2 \leq 1$, $j_1,j_2 \geq 0$ with $j_1+j_2 \leq 1$,
\begin{align*}
& \ \Big|\Big(\frac{M_{\delta_2}(t)}{1+\delta_1 \sqrt{\mathcal{H}}+\delta_1\langle D_x \rangle}\Gamma\Big(\frac{(1+\delta_1 \sqrt{\mathcal{H}})^{j_1}}{M_{\delta_2}(t)}f,\frac{(1+\delta_1 \sqrt{\mathcal{H}})^{j_2}}{M_{\delta_2}(t)}g\Big),h\Big)_{(r,0)}\Big|  \\  
& \leq   c_1\|f_0\|_{H^r(\rr_x)} \sum_{n=0}^{+\infty}\Big(n+\frac{1}{2}\Big)^s\|g_n\|_{H^r(\rr_x)}\|h_n\|_{H^r(\rr_x)} \\ 
& \qquad \qquad \qquad \qquad \qquad \qquad +   c_1\sum_{n=0}^{+\infty}\|h_n\|_{H^r(\rr_x)}\Big(\sum_{\substack{2k+l=n \\ k \geq 1, l \geq 0}}
\frac{\tilde{\mu}_{k,l}}{k^{\frac{3}{4}}}\|f_{2k}\|_{H^r(\rr_x)}\|g_l\|_{H^r(\rr_x)}\Big).
\end{align*}
By using \eqref{sake-1}, we obtain that
\begin{align*}
& \ \|f_0\|_{H^r(\rr_x)} \sum_{n=0}^{+\infty}\Big(n+\frac{1}{2}\Big)^s \|g_n\|_{H^r(\rr_x)}\|h_n\|_{H^r(\rr_x)} \\
 \leq & \  \|f_0\|_{H^r(\rr_x)} \Big(\sum_{n=0}^{+\infty}\Big(n+\frac{1}{2}\Big)^s\|g_n\|_{H^r(\rr_x)}^2\Big)^{\frac{1}{2}} \Big(\sum_{n=0}^{+\infty}\Big(n+\frac{1}{2}\Big)^s\|h_n\|_{H^r(\rr_x)}^2\Big)^{\frac{1}{2}} \\
\leq & \ \|f\|_{(r,0)}\|\mathcal{H}^\frac{s}{2}g\|_{(r,0)}\|\mathcal{H}^\frac{s}{2}h\|_{(r,0)}.
\end{align*}
Furthermore, we notice that
\begin{align*}
\Delta_2= & \ \sum_{n=0}^{+\infty}\|h_n\|_{H^r(\rr_x)}\Big(\sum_{\substack{2k+l=n \\ k \geq 1,l \geq 0}}\frac{\tilde{\mu}_{k,l}}{k^{\frac{3}{4}}} \|f_{2k}\|_{H^r(\rr_x)}\|g_l\|_{H^r(\rr_x)}\Big)\\
 =& \ \sum_{k \geq 1, l \geq 0}\frac{\tilde{\mu}_{k,l}}{k^{\frac{3}{4}}}\|f_{2k}\|_{H^r(\rr_x)}\|g_l\|_{H^r(\rr_x)}\|h_{2k+l}\|_{H^r(\rr_x)} \\
 =& \ \sum_{l=0}^{+\infty}\Big(l+\frac{1}{2}\Big)^{\frac{s}{2}}\|g_l\|_{H^r(\rr_x)} \Big(\sum_{k =1}^{+\infty}\frac{\tilde{\mu}_{k,l}}{k^{\frac{3}{4}}(l+\frac{1}{2})^{\frac{s}{2}}} \|f_{2k}\|_{H^r(\rr_x)}\|h_{2k+l}\|_{H^r(\rr_x)}\Big) .
\end{align*}
We deduce that
\begin{align*}
\Delta_2 \leq & \ \|\mathcal{H}^\frac{s}{2}g\|_{(r,0)} \Big[\sum_{l=0}^{+\infty}\Big(\sum_{k =1}^{+\infty}\frac{\tilde{\mu}_{k,l}}{k^{\frac{3}{4}}(l+\frac{1}{2})^{\frac{s}{2}}}
\|f_{2k}\|_{H^r(\rr_x)}\|h_{2k+l}\|_{H^r(\rr_x)}\Big)^2\Big]^{\frac{1}{2}}\\
 \leq & \ \|\mathcal{H}^\frac{s}{2}g\|_{(r,0)} \Big[\sum_{l=0}^{+\infty}\Big(\sum_{k =1}^{+\infty}\|f_{2k}\|_{H^r(\rr_x)}^2\Big)\Big(\sum_{k =1}^{+\infty}\frac{\tilde{\mu}_{k,l}^2}{k^{\frac{3}{2}}(l+\frac{1}{2})^{s}}\|h_{2k+l}\|_{H^r(\rr_x)}^2\Big)\Big]^{\frac{1}{2}}\\
 \leq & \   \|f\|_{(r,0)}\|\mathcal{H}^\frac{s}{2}g\|_{(r,0)} \Big(\sum_{l=0}^{+\infty}\sum_{k =1}^{+\infty}\frac{\tilde{\mu}_{k,l}^2}{k^{\frac{3}{2}}(l+\frac{1}{2})^{s}}\|h_{2k+l}\|_{H^r(\rr_x)}^2\Big)^{\frac{1}{2}}.
\end{align*}
We may write that
$$\Big(\sum_{l=0}^{+\infty}\sum_{k =1}^{+\infty}\frac{\tilde{\mu}_{k,l}^2}{k^{\frac{3}{2}}(l+\frac{1}{2})^{s}}\|h_{2k+l}\|_{H^r(\rr_x)}^2\Big)^{\frac{1}{2}}=\Big[\sum_{n=0}^{+\infty}\|h_{n}\|_{H^r(\rr_x)}^2\Big(\sum_{\substack{2k+l =n\\ k\geq 1,l \geq 0}}\frac{\tilde{\mu}_{k,l}^2}{k^{\frac{3}{2}}(l+\frac{1}{2})^{s}}\Big)\Big]^{\frac{1}{2}}.$$
On the other hand, we deduce from Lemma~\ref{lemkk0} that
\begin{multline*}
\sum_{\substack{2k+l =n\\ k\geq 1,l \geq 0}}\frac{\tilde{\mu}_{k,l}^2}{k^{\frac{3}{2}}(l+\frac{1}{2})^{s}} \lesssim \sum_{\substack{2k+l =n\\ k\geq 1,l \geq 0\\ k \geq l}}\frac{k^{\frac{1}{2}}}{k^{\frac{3}{2}}(l+\frac{1}{2})^{s}}+\sum_{\substack{2k+l =n\\ k\geq 1,l \geq 0 \\ k \leq l}}\frac{(l+\frac{1}{2})^{s}}{k^{\frac{3}{2}}}
\\ \lesssim \sum_{\substack{2k+l =n\\ k\geq 1,l \geq 0\\ k \geq l}}\frac{1}{k} +\Big(n+\frac{1}{2}\Big)^{s}\sum_{k= 1}^n\frac{1}{k^{\frac{3}{2}}}
\lesssim \Big(n+\frac{1}{2}\Big)^{s}\Big(\sum_{k= 1}^{n}\frac{1}{k^{1+s}}+\sum_{k= 1}^{n}\frac{1}{k^{\frac{3}{2}}}\Big)\lesssim \Big(n+\frac{1}{2}\Big)^{s},
\end{multline*}
since
$$\tilde{\mu}_{k,l} \lesssim k^{\frac{1}{4}} \textrm{ when } k \geq l,\ k \geq 1,\ l \geq 0; \quad \tilde{\mu}_{k,l} \lesssim \Big(l+\frac{1}{2}\Big)^{s} \textrm{ when } 1 \leq k \leq l.$$
This implies that
$$\Big(\sum_{l=0}^{+\infty}\sum_{k =1}^{+\infty}\frac{\tilde{\mu}_{k,l}^2}{k^{\frac{3}{2}}(l+\frac{1}{2})^{s}}\|h_{2k+l}\|_{H^r(\rr_x)}^2\Big)^{\frac{1}{2}} \lesssim \|\mathcal{H}^\frac{s}{2}h\|_{(r,0)}.$$
We conclude that there exists a positive constant $C_r>0$ such that for all $f,g,h \in \mathscr{S}(\rr_{x,v}^2)$, $t \geq 0$, $0 \leq \delta_1 \leq 1$, $0< \delta_2 \leq 1$, $j_1,j_2 \geq 0$ with $j_1+j_2 \leq 1$,
\begin{multline*}
\big|\big(\big(1+\delta_1 \sqrt{\mathcal{H}}+\delta_1\langle D_x \rangle\big)^{-1}M_{\delta_2}(t)\Gamma\big(\big(M_{\delta_2}(t)\big)^{-1}\big(1+\delta_1 \sqrt{\mathcal{H}}\big)^{j_1}f,\big(M_{\delta_2}(t)\big)^{-1}\big(1+\delta_1 \sqrt{\mathcal{H}}\big)^{j_2}g\big),h\big)_{(r,0)}\big| \\
\leq C_r \|f\|_{(r,0)}\|\mathcal{H}^\frac{s}{2}g\|_{(r,0)}\|\mathcal{H}^\frac{s}{2}h\|_{(r,0)}.
\end{multline*}
By taking $t=0$, $\delta_1=0$, $j_1=0$ and $j_2=0$, we obtain that
$$|(\Gamma(f,g),h)_{(r,0)}| \leq \frac{C_r}{1+\delta_2} \|f\|_{(r,0)}\|\mathcal{H}^\frac{s}{2}g\|_{(r,0)}\|\mathcal{H}^\frac{s}{2}h\|_{(r,0)} \leq C_r \|f\|_{(r,0)}\|\mathcal{H}^\frac{s}{2}g\|_{(r,0)}\|\mathcal{H}^\frac{s}{2}h\|_{(r,0)}.$$
On the other hand, we deduce from Lemma~\ref{proposition1a} that
\begin{multline*}
(\Gamma(f,g),h)_{L^2(\rr_{x,v}^2)}
=  \int_{\rr^2}\Big(\sum_{n=0}^{+\infty}\Big(\sum_{\substack{k+l=n \\ k,l \geq 0}}\alpha_{k,l}f_k(x) g_l(x)\Big)\psi_n(v)\Big)\overline{h(x,v)}dxdv\\
=  \sum_{n=0}^{+\infty}\sum_{\substack{k+l=n \\ k,l \geq 0}}\alpha_{k,l} \int_{\rr}f_k(x) g_l(x) \overline{h_n(x)}dx.
\end{multline*}
It follows from the Sobolev imbedding that there exists a positive constant $c_r>0$ such that
\begin{multline*}
|(\Gamma(f,g),h)_{L^2(\rr_{x,v}^2)}\big| \leq \sum_{n=0}^{+\infty}\sum_{\substack{k+l=n \\ k,l \geq 0}}|\alpha_{k,l}|\|f_k\|_{L^{\infty}(\rr_x)}\|g_l\|_{L^2(\rr_x)}\|h_n\|_{L^2(\rr_x)}\\
\leq c_r\sum_{n=0}^{+\infty}\sum_{\substack{k+l=n \\ k,l \geq 0}}|\alpha_{k,l}|\|f_k\|_{H^r(\rr_x)}\|g_l\|_{L^2(\rr_x)}\|h_n\|_{L^2(\rr_x)}.
\end{multline*}
By substituting respectively $\|g_l\|_{L^2(\rr_x)}$ to $\|g_l\|_{H^r(\rr_x)}$ and $\|h_n\|_{L^2(\rr_x)}$ to $\|h_n\|_{H^r(\rr_x)}$ in formula (\ref{cp3.2}), the very same previous estimates allow to obtain that there exists a positive constant $c_r>0$ such that
$$|(\Gamma(f,g),h)_{L^2(\rr_{x,v}^2)}\big| \leq c_r\|f\|_{(r,0)}\|\mathcal{H}^\frac{s}{2}g\|_{L^2(\rr_{x,v}^2)}\|\mathcal{H}^\frac{s}{2}h\|_{L^2(\rr_{x,v}^2)}.$$
Lastly, we deduce from Lemma~\ref{proposition1a} that
\begin{multline*}
(\Gamma\big(\langle D_x \rangle f,g),h)_{L^2(\rr_{x,v}^2)}
=  \int_{\rr^2}\Big(\sum_{n=0}^{+\infty}\Big(\sum_{\substack{k+l=n \\ k,l \geq 0}}\alpha_{k,l}(\langle D_x \rangle f_k)(x) g_l(x)\Big)\psi_n(v)\Big)\overline{h(x,v)}dxdv\\
=  \sum_{n=0}^{+\infty}\sum_{\substack{k+l=n \\ k,l \geq 0}}\alpha_{k,l} \int_{\rr}(\langle D_x \rangle f_k)(x) g_l(x) \overline{h_n(x)}dx.
\end{multline*}
It follows from the one-dimensional Sobolev imbedding theorem that there exists a positive constant $c_r>0$ such that
\begin{align*}
|(\Gamma(\langle D_x \rangle f,g),h)_{L^2(\rr_{x,v}^2)}| \leq & \ \sum_{n=0}^{+\infty}\sum_{\substack{k+l=n \\ k,l \geq 0}}|\alpha_{k,l}|\|\langle D_x \rangle f_k\|_{L^{2}(\rr_x)}\|g_l \overline{h_n}\|_{L^2(\rr_x)}\\
\leq & \ \sum_{n=0}^{+\infty}\sum_{\substack{k+l=n \\ k,l \geq 0}}|\alpha_{k,l}|\|f_k\|_{H^1(\rr_x)}\|g_l\|_{L^{\infty}(\rr_x)}\|h_n\|_{L^2(\rr_x)}\\
\leq & \ c_r\sum_{n=0}^{+\infty}\sum_{\substack{k+l=n \\ k,l \geq 0}}|\alpha_{k,l}|\|f_k\|_{H^1(\rr_x)}\|g_l\|_{H^r(\rr_x)}\|h_n\|_{L^2(\rr_x)}.
\end{align*}
By substituting respectively $\|f_k\|_{H^1(\rr_x)}$ to $\|f_k\|_{H^r(\rr_x)}$ and $\|h_n\|_{L^2(\rr_x)}$ to $\|h_n\|_{H^r(\rr_x)}$ in formula (\ref{cp3.2}), the very same previous estimates allow to obtain that there exists a positive constant $c_r>0$ such that
$$|(\Gamma(\langle D_x \rangle f,g),h)_{L^2(\rr_{x,v}^2)}| \leq c_r\|f\|_{(1,0)}\|\mathcal{H}^\frac{s}{2}g\|_{(r,0)}\|\mathcal{H}^\frac{s}{2}h\|_{L^2(\rr_{x,v}^2)}.$$
This ends the proof of Lemma~\ref{proposition222}.
\end{proof}
We shall also need the following a priori estimates:
\begin{lemma}\label{proposition2225}
Let $r>1/2$. Then, there exists a positive constant $c_r>0$ such that for all $f,g \in \mathscr{S}(\rr_{x,v}^2)$,
$$\|\mathcal{H}^{-s}\Gamma(f,g)\|_{(r,0)} \leq c_r \|f\|_{(r,0)}\|g\|_{(r,0)}.$$
\end{lemma}
\begin{proof}
Let $r>1/2$ and $f,g, h \in \mathscr{S}(\rr_{x,v}^2)$. We decompose these functions into the Hermite basis in the velocity variable
$$f=\sum_{n=0}^{+\infty}f_n(x) \psi_n(v), \quad g=\sum_{n=0}^{+\infty}g_n(x) \psi_n(v), \quad h=\sum_{n=0}^{+\infty}h_n(x) \psi_n(v),$$
with
$$f_n(x)=(f(x,\cdot),\psi_n)_{L^2(\rr_v)}, \quad g_n(x)=(g(x,\cdot),\psi_n)_{L^2(\rr_v)}, \quad h_n(x)=(h(x,\cdot),\psi_n)_{L^2(\rr_v)}.$$
We also consider the functions
\begin{equation}\label{tg1}
\tilde{g}=\mathcal{H}^{-\frac{s}{2}}g=\sum_{n=0}^{+\infty}\tilde{g}_n(x) \psi_n(v), \quad \tilde{h}=\mathcal{H}^{-\frac{s}{2}}h=\sum_{n=0}^{+\infty}\tilde{h}_n(x) \psi_n(v),
\end{equation}
whose coefficients satisfy
$$\tilde{g}_n=\Big(n+\frac{1}{2}\Big)^{-\frac{s}{2}}g_n, \quad \tilde{h}_n=\Big(n+\frac{1}{2}\Big)^{-\frac{s}{2}}h_n.$$
We deduce from Lemma~\ref{proposition1a} and (\ref{tg1}) that
\begin{align*}
& \ (\mathcal{H}^{-s}\Gamma(f,g),h)_{(r,0)}=(\mathcal{H}^{-\frac{s}{2}}\Gamma(f,g),\tilde{h})_{(r,0)}\\ \notag
= & \ \int_{\rr^2}\langle D_x \rangle^r\mathcal{H}^{-\frac{s}{2}}\Big(\sum_{n=0}^{+\infty}\Big(\sum_{\substack{k+l=n \\ k,l \geq 0}}\alpha_{k,l}f_k g_l\Big)\psi_n(v)\Big)\overline{\langle D_x \rangle^r \tilde{h}(x,v)}dxdv\\ \notag
= & \ \sum_{n=0}^{+\infty}\sum_{\substack{k+l=n \\ k,l \geq 0}}\alpha_{k,l} \Big(n+\frac{1}{2}\Big)^{-\frac{s}{2}} \int_{\rr}  \langle D_x \rangle^r \big(f_k g_l \big)\overline{\langle D_x \rangle^r \tilde{h}_n}dx.
\end{align*}
It follows from (\ref{cl6}) that
\begin{align*}
|(\mathcal{H}^{-s}\Gamma(f,g),h)_{(r,0)}| \leq & \ \sum_{n=0}^{+\infty}\sum_{\substack{k+l=n \\ k,l \geq 0}}|\alpha_{k,l}| \Big(l+\frac{1}{2}\Big)^{-\frac{s}{2}} \|f_kg_l\|_{H^r(\rr_x)}\|\tilde{h}_n\|_{H^r(\rr_x)}\\
\leq & \ C_r\sum_{n=0}^{+\infty}\sum_{\substack{k+l=n \\ k,l \geq 0}}|\alpha_{k,l}| \Big(l+\frac{1}{2}\Big)^{-\frac{s}{2}} \|f_k\|_{H^r(\rr_x)}\|g_l\|_{H^r(\rr_x)}\|\tilde{h}_n\|_{H^r(\rr_x)}\\
 \leq & \ C_r \sum_{n=0}^{+\infty}\sum_{\substack{k+l=n \\ k,l \geq 0}}|\alpha_{k,l}|\|f_k\|_{H^r(\rr_x)}\|\tilde{g}_l\|_{H^r(\rr_x)}\|\tilde{h}_n\|_{H^r(\rr_x)}.
\end{align*}
By using the very same previous estimates as in (\ref{cp3.2}) in the proof of Lemma~\ref{proposition222}, we deduce that
$$|(\mathcal{H}^{-s}\Gamma(f,g),h)_{(r,0)}| \lesssim \|f\|_{(r,0)}\|\mathcal{H}^\frac{s}{2}\tilde{g}\|_{(r,0)}\|\mathcal{H}^\frac{s}{2}\tilde{h}\|_{(r,0)},$$
that is
$$|(\mathcal{H}^{-s}\Gamma(f,g),h)_{(r,0)}| \lesssim \|f\|_{(r,0)}\|g\|_{(r,0)}\|h\|_{(r,0)}.$$
This implies that
$$\|\mathcal{H}^{-s}\Gamma(f,g)\|_{(r,0)} \lesssim \|f\|_{(r,0)}\|g\|_{(r,0)}.$$
This ends the proof of Lemma~\ref{proposition2225}.
\end{proof}

\section{Hypoelliptic estimate for  the principal part of the linear  inhomogeneous Kac operator}\label{fok}

We consider the operator acting in the velocity variable
\begin{equation}\label{ind10}
P=iv \xi+a_0^w(v,D_v),
\end{equation}
with parameter $\xi \in \rr$, where the operator $A=a_0^w(v,D_v)$ stands for the pseudodifferential operator
$$a_0^w(v,D_v)u=\frac{1}{2\pi}\int_{\rr^2}e^{i(v-w)\eta}a_0\Big(\frac{v+w}{2},\eta\Big)u(w)dwd\eta,$$
defined by the Weyl quantization of the symbol
\begin{equation}\label{symba0}
a_0(v,\eta)=c_0\Big(1+\eta^2+\frac{v^2}{4}\Big)^s,
\end{equation}
with some constants $c_0>0$, $0<s<1$. This operator corresponds to the principal part of the linear  inhomogeneous Kac operator $$v\partial_x+\mathcal{K},$$
on the Fourier side in the position variable.

Let $\psi$ be a $C_0^{\infty}(\rr,[0,1])$ function satisfying
$$\psi=1 \textrm{ on } [-1,1], \quad \textrm{supp }\psi \subset [-2,2].$$
We define the real-valued symbol
\begin{equation}\label{m1fp}
g=-\frac{\xi \eta}{\lambda^{\frac{2s+2}{2s+1}}}
\psi\left(\frac{\eta^2+v^2}{\lambda^{\frac{2}{2s+1}}}\right),
\end{equation}
with
\begin{equation}\label{tg2.7}
\lambda=(1+v^2+\eta^2+\xi^2)^{\frac{1}{2}}.
\end{equation}
The variable $\eta$ stands for the Fourier dual variable of the velocity variable $v$, whereas the variable~$\xi$ denotes the Fourier dual variable of the position variable $x$.
We aim at establishing the following result:

 \begin{lemma} \label{KFP2}
Let $P$ be the operator defined in \emph{(\ref{ind10})} and $G=g^w$ the selfadjoint operator defined by the Weyl quantization of the symbol \emph{(\ref{m1fp})}. Then,
the operator $G$ is uniformly bounded on $L^2(\rr_v)$ with respect to the parameter $\xi \in \rr$, and there exist some positive constants $0<\eps_0 \leq 1$, $c_1,c_2>0$ such that for all $0<\eps \leq \eps_0$, $u \in \mathscr{S}(\rr_v)$, $\xi \in \rr$,
$$\emph{\textrm{Re}}(Pu,(1-\eps G)u)_{L^2(\rr_v)} \geq c_{1}\|\mathcal{H}^{\frac{s}{2}}u\|_{L^2(\rr_v)}^2+c_1\eps \langle \xi \rangle^{\frac{2s}{2s+1}}\|u\|_{L^2(\rr_v)}^2-c_{2}\|u\|_{L^2(\rr_v)}^2,$$
where $\mathcal{H}=-\Delta_v+\frac{v^2}{4}$ stands for the harmonic oscillator.
\end{lemma}

This lemma is an adaptation to the fractional diffusion case of the hypoelliptic estimate proven in \cite{Herau-Starov} (Proposition 2.1) for the Kramers-Fokker-Planck operator.
Following standard notations~\cite{Hor85,Le}, we consider the following metrics on the phase space~$\rr_{v,\eta}^2$,
\begin{equation}\label{f1}
\Gamma_0 = \frac{dv^2+d\eta^2}{\langle (v,\eta) \rangle^2}, \qquad \Gamma_1 = \frac{dv^2+d\eta^2}{M(v,\eta,\xi)},
\end{equation}
with
$$\langle (v,\eta) \rangle^2=1+v^2+\eta^2,$$
\begin{equation}\label{ya2}
M(v,\eta,\xi) = 1+v^2 + \eta^2 + \lambda^{\frac{2}{2s+1}}=1+v^2 + \eta^2+(1+v^2+\eta^2+\xi^2)^{\frac{1}{2s+1}}.
\end{equation}
Notice that the second metric depends on the parameter $\xi \in \rr$.

For a positive function $\mu \geq1$, we define the space $S(\mu,\Gamma_0)$ as the set of functions $a\in C^{\infty}(\real_{v,\eta}^{2},\cc)$ possibly depending on the parameter $\xi$ satisfying
$$\forall \alpha \in \nn^{2}, \exists C_{\alpha}>0, \forall (v,\eta,\xi) \in \rr^{3}, \quad  |\partial_{v,\eta}^{\alpha} a(v,\eta,\xi)| \leq C_{\alpha} \mu(v,\eta,\xi)\langle (v,\eta) \rangle^{-|\alpha|},$$
whereas the space $S(\mu,\Gamma_1)$ corresponds to $C^{\infty}(\real_{v,\eta}^{2},\cc)$ functions depending on the parameter $\xi$ satisfying
$$\forall \alpha \in \nn^{2}, \exists C_{\alpha}>0, \forall (v,\eta,\xi) \in \rr^{3}, \quad  |\partial_{v,\eta}^{\alpha} a(v,\eta,\xi)| \leq C_{\alpha} \mu(v,\eta,\xi) M(v,\eta,\xi)^{-\frac{|\alpha|}{2}}.$$
The metrics $\Gamma_0$ and $\Gamma_1$ are admissible (slowly varying, satisfying the uncertainty principle and temperate), see Appendix (Section~\ref{symbolic}). 
In addition, we need to verify some properties for the weight $\mu \geq 1$ with respect to the metric $\Gamma_j$, namely the slowly varying property of $\mu$ with respect to $\Gamma_j$, for the function space $S(\mu,\Gamma_j)$ to enjoy nice symbolic calculus properties. In the present work, we shall work in the symbol classes
$$S(\langle (v,\eta)\rangle^m,\Gamma_0), \quad S(M^m,\Gamma_1), $$
with $m \in \rr$, which enjoy nice symbolic calculus since the function $\langle (v,\eta)\rangle^m$ is a $\Gamma_0$-slowly varying weight and that the function $M$ is a $\Gamma_1$-slowly varying weight uniformly with respect to the parameter $\xi \in \rr$, see Appendix (Section~\ref{symbolic}).
The gain functions in the symbolic calculus associated to these two symbol classes $S(\langle (v,\eta)\rangle^m,\Gamma_0)$ and $S(M^m,\Gamma_1)$ are respectively given by
$$\Lambda_{\Gamma_0}=\langle (v,\eta)\rangle^2, \quad \Lambda_{\Gamma_1}=M(v,\eta,\xi).$$
On the other hand, let us notice that the following inclusion holds
\begin{equation}\label{hm1}
S(m,\Gamma_1) \subset S(m,\Gamma_0),
\end{equation}
because
\begin{equation}\label{hm1.2}
\langle (v,\eta) \rangle^2 \leq M(v,\eta,\xi).
\end{equation}

In the following, we shall frequently use the equivalence of norms
\begin{equation}\label{hm4}
\forall r \in \rr, \exists C_r>0, \ \frac{1}{C_r}\|\mathcal{H}^ru\|_{L^2} \leq  \Big\|\textrm{Op}^w\Big(\Big(1+\eta^2+\frac{v^2}{4}\Big)^{r}\Big)u\Big\|_{L^2} \leq C_r\|\mathcal{H}^ru\|_{L^2},
\end{equation}
where $\mathcal{H}=-\Delta_v+\frac{v^2}{4}$ stands for the harmonic oscillator.
This natural link between pseudodifferential calculus and functional calculus may be readily deduced from~\cite{HelNie05HES} (Proposition~4.5).
We begin by proving the following symbolic estimates:
\begin{lemma} \label{2.2fp}
For all $m \in \R$, the following symbols belong to their respective function spaces
\begin{align*}
i) & \ \langle \xi \rangle^m \in S( \lambda^m, \Gamma_1) \quad ii)\ \lambda^m \in S( \lambda^m, \Gamma_1)\quad iii) \ \psi\left(\frac{\eta^2+v^2}{\lambda^{\frac{2}{2s+1}}}\right) \in S(1,\Gamma_1) \quad
iv) \  g \in S(1,\Gamma_1) \\
v) & \  \frac{\xi^2}{\lambda^{\frac{2s+2}{2s+1}}}\psi\Big(\frac{\eta^2+v^2}{\lambda^{\frac{2}{2s+1}}}\Big) \in S(M,\Gamma_1)
\quad vi) \ \Big(1-\psi\Big(\frac{\eta^2+v^2}{\lambda^{\frac{2}{2s+1}}}\Big)\Big)\Big(1+\eta^2+\frac{v^2}{4}\Big)^{s} \in S(M,\Gamma_1)
\end{align*}
uniformly with respect to the parameter $\xi \in \R$.
\end{lemma}
\begin{proof}
The assertion $i)$ is trivial since the term $\langle \xi \rangle^m$ is independent of the variables $(v,\eta)$.
On the other hand, we easily derive from (\ref{tg2.7}) and (\ref{ya2}) that
$$\forall \alpha  \in \N^{2}, \quad |\D_{v,\eta}^\alpha(\lambda^m)| \lesssim \lambda^{m -|\alpha|} \lesssim \lambda^m M^{-\frac{|\alpha|}{2}},$$
uniformly \wrt the parameter $\xi \in \R$, since the estimate
$M^{\frac{1}{2}} \lesssim \lambda$
holds uniformly with respect to~$\xi \in \rr$. This proves the assertion $ii)$. Regarding the assertion $iii)$, we first notice that on the support of the function
$$\psi\left(\frac{\eta^2+v^2}{\lambda^{\frac{2}{2s+1}}}\right),$$
the estimate $\eta^2+v^2 \leq  2\lambda^{\frac{2}{2s+1}}$ implies that
$$M^{\frac{1}{2}} \sim \lambda^{\frac{1}{2s+1}}$$
and
$$|\D^\alpha_{v,\eta} (\eta^2 + v^2)|  \lesssim \left\{ \begin{array}{ll} \lambda^{\frac{2}{2s+1}} & \textrm{ when } |\alpha| = 0,\\  \lambda^{\frac{1}{2s+1}} & \textrm{ when } |\alpha| = 1,\\ 1 & \textrm{ when } |\alpha| = 2, \\ 0 & \textrm{ when } |\alpha| \geq 3. \end{array} \right.$$
The assertion $iii)$ then directly follows from assertion $ii)$. Next, we notice that on the support of the function
$$\psi\left(\frac{\eta^2+v^2}{\lambda^{\frac{2}{2s+1}}}\right),$$
the estimate $|\xi \eta| \lesssim \lambda^{\frac{2s+2}{2s+1}}$ implies that
$$|\D^\alpha_{v,\eta} (\xi \eta)|  \lesssim \left\{ \begin{array}{ll} \lambda^{\frac{2s+2}{2s+1}} & \textrm{ when } |\alpha| = 0,\\  \lambda & \textrm{ when } |\alpha| = 1,\\ 0 & \textrm{ when } |\alpha| \geq 2. \end{array} \right.$$
Recalling that in this region $M^{\frac{1}{2}} \sim \lambda^{\frac{1}{2s+1}}$, the assertion $iv)$ is then a direct consequence of (\ref{m1fp}) and the assertions $ii)$ and $iii)$.
Recalling that $0<s<1$, we deduce from (\ref{ya2}), $i)$, $ii)$, $iii)$ that
$$\frac{\xi^2}{\lambda^{\frac{2s+2}{2s+1}}}\psi\Big(\frac{\eta^2+v^2}{\lambda^{\frac{2}{2s+1}}}\Big) \in S(\lambda^{\frac{2s}{2s+1}},\Gamma_1) \subset S(M,\Gamma_1),$$
uniformly with respect to the parameter $\xi \in \R$. This proves the assertion $v)$. Regarding the last assertion, we first notice that
\begin{equation}\label{q1}
M^{\frac{1}{2}} \sim \langle (v,\eta) \rangle,
\end{equation}
on the support of the function
$$1-\psi\Big(\frac{\eta^2+v^2}{\lambda^{\frac{2}{2s+1}}}\Big),$$
where $\eta^2+v^2 \geq  \lambda^{\frac{2}{2s+1}}$. By using that
$$\partial_{v,\eta}^{\alpha}\Big(\Big(1+\eta^2+\frac{v^2}{4}\Big)^{s}\Big) \lesssim \langle (v,\eta) \rangle^{2s-|\alpha|},$$
we deduce from $iii)$ and (\ref{q1}) that
$$\Big(1-\psi\Big(\frac{\eta^2+v^2}{\lambda^{\frac{2}{2s+1}}}\Big)\Big)\Big(1+\eta^2+\frac{v^2}{4}\Big)^{s} \in S(M^s,\Gamma_1) \subset S(M,\Gamma_1).$$
This proves the assertion $vi)$. This ends the proof of Lemma~\ref{2.2fp}.
\end{proof}

The following lemma shows that up to controlled terms and a weight factor $\lambda^{\frac{2s+2}{2s+1}}$, the Poisson bracket
$$H_{p_1} g=\{p_1, g\}=\frac{\partial p_1}{\partial \eta}\frac{\partial g}{\partial v}-\frac{\partial p_1}{\partial v}\frac{\partial g}{\partial \eta},$$
with $p_1=v\xi$ the symbol associated to the transport part of $P$, lets appear the elliptic symbol
$$\{p_1,-\xi \eta\}=\xi^2,$$
in the region of the phase space where $\eta^2+v^2 \lesssim \lambda^{\frac{2}{2s+1}}$.

\begin{lemma} \label{mainterm}
With $p_1=v\xi$, we have
$$
H_{p_1} g =\{v\xi,g\}= \frac{\xi^2}{\lambda^{\frac{2s+2}{2s+1}}}
\psi\left(\frac{\eta^2+v^2}{\lambda^{\frac{2}{2s+1}}}\right) + r,
$$
with a remainder $r$ belonging to the symbol class $S(\langle (v,\eta) \rangle^{2s},\Gamma_0)$ uniformly with respect to the parameter $\xi \in \R$.
\end{lemma}

\begin{proof}
Recalling the definition (\ref{m1fp}), an explicit computation of the Poisson bracket
$$H_{p_1} g =\{p_1,g\}=\{v\xi,g\},$$
gives that
$$\{v\xi,g\}=-\xi\frac{\partial g}{\partial \eta}=
\frac{\xi^2}{\lambda^{\frac{2s+2}{2s+1}}}\psi\left(\frac{\eta^2+v^2}{\lambda^{\frac{2}{2s+1}}}\right)+r,$$
with
$$r=\xi^2 \eta\partial_{\eta}\big(\lambda^{-\frac{2s+2}{2s+1}}\big)\psi\left(\frac{\eta^2+v^2}{\lambda^{\frac{2}{2s+1}}}\right)
+\frac{\xi^2 \eta}{\lambda^{\frac{2s+2}{2s+1}}}\partial_{\eta}\left[\psi\left(\frac{\eta^2+v^2}{\lambda^{\frac{2}{2s+1}}}\right)\right].$$
Recalling that $M^{\frac{1}{2}} \sim \lambda^{\frac{1}{2s+1}}$ on the support of the function
$$\psi\left(\frac{\eta^2+v^2}{\lambda^{\frac{2}{2s+1}}}\right),$$
we notice that in this region
\begin{equation}\label{fili1b}
|\D^\alpha_{v,\eta} (\xi^2 \eta)|  \lesssim \left\{ \begin{array}{ll} \lambda^{\frac{4s+3}{2s+1}} & \textrm{when } |\alpha| = 0,\\
\lambda^{2} & \textrm{when } |\alpha| = 1, \\
0 & \textrm{when } |\alpha| \geq 2, \end{array} \right. \quad |\D^\alpha_{v,\eta} (\xi^2 \eta^2)|  \lesssim \left\{ \begin{array}{ll} \lambda^{\frac{4s+4}{2s+1}} & \textrm{when } |\alpha| = 0,\\
\lambda^{\frac{4s+3}{2s+1}} & \textrm{when } |\alpha| = 1, \\
\lambda^2 & \textrm{when } |\alpha| = 2,\\
0 & \textrm{when } |\alpha| \geq 3, \end{array} \right.
\end{equation}
because $\eta^2+v^2 \leq 2\lambda^{\frac{2}{2s+1}}$, and we therefore deduce from Lemma~\ref{2.2fp} and (\ref{fili1b}) that
$$\xi^2 \eta\partial_{\eta}\big(\lambda^{-\frac{2s+2}{2s+1}}\big)\psi\left(\frac{\eta^2+v^2}{\lambda^{\frac{2}{2s+1}}}\right) = -\frac{2s+2}{2s+1} \xi^2\eta^2 \lambda^{-\frac{6s+4}{2s+1}}\psi\left(\frac{\eta^2+v^2}{\lambda^{\frac{2}{2s+1}}}\right) \in S(\lambda^{-\frac{2s}{2s+1}},\Gamma_1) \subset S(1,\Gamma_1)$$
and
$$\frac{\xi^2 \eta}{\lambda^{\frac{2s+2}{2s+1}}}\partial_{\eta}\left[\psi\left(\frac{\eta^2+v^2}{\lambda^{\frac{2}{2s+1}}}\right)\right] \in S(\lambda^{\frac{2s}{2s+1}},\Gamma_1).$$
By using now that
$$\eta^2+v^2 \sim \lambda^{\frac{2}{2s+1}} \enskip \mbox{
on the support of } \enskip
\psi'\left(\frac{\eta^2+v^2}{\lambda^{\frac{2}{2s+1}}}\right),$$
it follows that
$$\frac{\xi^2 \eta}{\lambda^{\frac{2s+2}{2s+1}}}\partial_{\eta}\left[\psi\left(\frac{\eta^2+v^2}{\lambda^{\frac{2}{2s+1}}}\right)\right] \in S(\langle (v,\eta)\rangle^{2s},\Gamma_1),$$
uniformly with respect to the parameter $\xi \in \R$. This implies that the remainder $r$ belongs to the symbol class $S(\langle (v,\eta)\rangle^{2s},\Gamma_1)$ uniformly with respect to the parameter $\xi \in \R$. Finally, we deduce from (\ref{hm1}) and (\ref{hm1.2}) that the remainder $r$ belongs to the symbol class $S(\langle (v,\eta)\rangle^{2s},\Gamma_0)$ uniformly with respect to the parameter $\xi \in \R$. The proof of Lemma~\ref{mainterm} is complete.
\end{proof}

We shall now prove Lemma~\ref{KFP2}.
Let $\eps$ be a positive parameter satisfying $0<\eps \leq 1$.
We consider the multiplier $G = g^w$ defined by the Weyl quantization of the symbol (\ref{m1fp}). It follows from (\ref{ind10}) that
\begin{multline}\label{eq7fp}
\textrm{Re}(Pu,(1-\eps G)u)_{L^2(\rr_v)}=\textrm{Re}(a_0^w(v,D_v)u,u)_{L^2(\rr_v)}-\eps \textrm{Re}(iv\xi u,Gu)_{L^2(\rr_v)}\\
-\eps \textrm{Re}(a_0^w(v,D_v) u,Gu)_{L^2(\rr_v)}.
\end{multline}
We have
\begin{equation}\label{w1b}
\textrm{Re}(a_0^w(v,D_v)u,u)_{L^2(\rr_v)}=(a_0^w(v,D_v)u,u)_{L^2(\rr_v)},
\end{equation}
since the symbol $a_0$ is real-valued.
In order to estimate the first term, we use some symbolic calculus in the class $S(\langle (v,\eta)\rangle^s,\Gamma_0)$, and notice that
\begin{equation}\label{hm6}
a_0-c_0\Big(1+\eta^2+\frac{v^2}{4}\Big)^{\frac{s}{2}} \sharp^w \Big(1+\eta^2+\frac{v^2}{4}\Big)^{\frac{s}{2}} \in S(\langle (v,\eta)\rangle^{2s-2},\Gamma_0) \subset S(1,\Gamma_0),
\end{equation}
since $0<s<1$.
Then, it follows from (\ref{hm4}) that there exist some positive constants $C, \tilde{c}_0>0$ such that for all $u \in \mathscr{S}(\rr_v)$,
\begin{equation}\label{hm3}
(a_0^w(v,D_v)u,u)_{L^2} \geq c_0\Big\|\textrm{Op}^w\Big(\Big(1+\eta^2+\frac{v^2}{4}\Big)^{\frac{s}{2}}\Big)u\Big\|_{L^2}^2-C\|u\|_{L^2}^2
\geq  \tilde{c}_0\|\mathcal{H}^{\frac{s}{2}}u\|_{L^2}^2-C\|u\|_{L^2}^2.
\end{equation}
Now, it remains to estimate the two last terms appearing in (\ref{eq7fp}). We begin by noticing from Lemma~\ref{2.2fp} that the operator $G=g^w$ is uniformly bounded on $L^2(\rr_v)$ with respect to the parameter $\xi \in \rr$. By using (\ref{hm4}) and (\ref{hm6}), it follows that
\begin{multline}\label{hm7.1}
\textrm{Re}(a_0^w(v,D_v) u,Gu)_{L^2}\\
= c_0\textrm{Re}\Big(\textrm{Op}^w\Big(\Big(1+\eta^2+\frac{v^2}{4}\Big)^{\frac{s}{2}}\Big)u,\textrm{Op}^w\Big(\Big(1+\eta^2+\frac{v^2}{4}\Big)^{\frac{s}{2}}\Big)Gu\Big)_{L^2}
+\mathcal{O}(\|u\|_{L^2}^2)
=  R_1+R_2+\mathcal{O}(\|u\|_{L^2}^2),
\end{multline}
with
\begin{equation}\label{hm7}
R_1=c_0\textrm{Re}\Big(\textrm{Op}^w\Big(\Big(1+\eta^2+\frac{v^2}{4}\Big)^{\frac{s}{2}}\Big)u,G\ \textrm{Op}^w\Big(\Big(1+\eta^2+\frac{v^2}{4}\Big)^{\frac{s}{2}}\Big)u\Big)_{L^2} =\mathcal{O}(\|\mathcal{H}^{\frac{s}{2}}u\|_{L^2}^2),
\end{equation}
\begin{equation}\label{hm7.2}
R_2=c_0\textrm{Re}\Big(\textrm{Op}^w\Big(\Big(1+\eta^2+\frac{v^2}{4}\Big)^{\frac{s}{2}}\Big)u,\Big[\textrm{Op}^w\Big(\Big(1+\eta^2+\frac{v^2}{4}\Big)^{\frac{s}{2}}\Big),g^w\Big]u\Big)_{L^2}.
\end{equation}
We notice from Lemma~\ref{2.2fp}, (\ref{hm1}) and (\ref{hm1.2}) that
$$g \in S(1,\Gamma_0),$$
uniformly with respect to the parameter $\xi \in \rr$.
Then, some symbolic calculus shows that
\begin{equation}\label{hm7.3}
\Big[\textrm{Op}^w\Big(\Big(1+\eta^2+\frac{v^2}{4}\Big)^{\frac{s}{2}}\Big),g^w\Big] \in \textrm{Op}^w(S(\langle (v,\eta) \rangle^{s-2},\Gamma_0)) \subset  \textrm{Op}^w(S(1,\Gamma_0)),
\end{equation}
since $0<s<1$.
It therefore follows from (\ref{hm4}), (\ref{hm7.1}), (\ref{hm7}), (\ref{hm7.2}) and (\ref{hm7.3}) that
\begin{equation}\label{hm7.4}
\textrm{Re}(a_0^w(v,D_v) u,Gu)_{L^2}=\mathcal{O}(\|\mathcal{H}^{\frac{s}{2}}u\|_{L^2}^2),
\end{equation}
uniformly with respect to the parameter $\xi \in \rr$.
Regarding the last term, we may write
\begin{equation}\label{p1}
-\textrm{Re}(iv\xi u,Gu)_{L^2}= \frac{1}{2}([i v\xi ,G]u,u)_{L^2},
\end{equation}
since the operators $G$ and $iv\xi$ are respectively formally selfadjoint and skew-selfadjoint on~$L^2$.
Some symbolic calculus shows that the Weyl symbol of the commutator
\begin{equation}\label{p2}
2^{-1}[i v\xi ,G],
\end{equation}
is exactly given by
\begin{equation}\label{p3}
2^{-1}\{v\xi, g\}.
\end{equation}
Then, Lemma~\ref{mainterm} shows that the symbol of this commutator may be written as
\begin{equation}\label{p4}
\frac{1}{2}\{v\xi, g\}=\frac{\xi^2}{2\lambda^{\frac{2s+2}{2s+1}}}\psi\left(\frac{\eta^2+v^2}{\lambda^{\frac{2}{2s+1}}}\right) +r,
\end{equation}
with a remainder term $r \in S(\langle (v,\eta)\rangle^{2s},\Gamma_0) \subset S(\langle (v,\eta)\rangle^{2},\Gamma_0)$ uniformly with respect to the parameter $\xi \in \R$.
The symbol $r$ is therefore a first order symbol for the symbolic calculus associated to the class $S(\langle (v,\eta)\rangle^{2},\Gamma_0)$.
By noticing that
$$|r| \lesssim \Big(1+\eta^2+\frac{v^2}{4}\Big)^{s},$$
we deduce from the G\aa rding inequality (see e.g.~\cite{Le}, Theorem~2.5.4) applied in the class
$$S(\langle (v,\eta)\rangle^{2},\Gamma_0),$$
that
\begin{equation}\label{hm8.1}
|(r^w u, u)_{L^2}|\leq \Big(\textrm{Op}^w\Big(\Big(1+\eta^2+\frac{v^2}{4}\Big)^{s}\Big)u,u\Big)_{L^2}+\mathcal{O}(\|u\|_{L^2}^2),
\end{equation}
uniformly with respect to the parameter $\xi \in \rr$.
We deduce from (\ref{symba0}), (\ref{hm4}), (\ref{hm6}) and (\ref{hm8.1}) that
\begin{equation}\label{hm8}
|(r^w u, u)_{L^2}| \lesssim \Big\|\textrm{Op}^w\Big(\Big(1+\eta^2+\frac{v^2}{4}\Big)^{\frac{s}{2}}\Big)u\Big\|_{L^2}^2+\|u\|_{L^2}^2=\mathcal{O}(\|\mathcal{H}^{\frac{s}{2}}u\|_{L^2}^2),
\end{equation}
uniformly with respect to the parameter $\xi \in \rr$.
Setting
\begin{equation}\label{fili5}
\Psi=\frac{\xi^2}{2\lambda^{\frac{2s+2}{2s+1}}}\psi\left(\frac{\eta^2+v^2}{\lambda^{\frac{2}{2s+1}}}\right) ,
\end{equation}
it follows from (\ref{p1}), (\ref{p2}), (\ref{p3}), (\ref{p4}) and (\ref{hm8}) that there exists a positive constant $c>0$ such that for all $u \in \mathscr{S}(\rr_v)$, $\xi \in \rr$,
\begin{equation} \label{333}
-\textrm{Re}(iv\xi u,Gu)_{L^2} \geq (\Psi^w u,u)_{L^2}-c\|\mathcal{H}^{\frac{s}{2}}u\|_{L^2}^2.
\end{equation}
Then, we deduce from (\ref{eq7fp}), (\ref{w1b}), (\ref{hm3}), (\ref{hm7.4}) and (\ref{333}) that
$$\textrm{Re}(Pu,(1-\eps G)u)_{L^2} \geq \tilde{c}_0\|\mathcal{H}^{\frac{s}{2}}u\|_{L^2}^2+\eps(\Psi^wu,u)_{L^2}-C\|u\|_{L^2}^2-\eps\mathcal{O}(\|\mathcal{H}^{\frac{s}{2}}u\|_{L^2}^2),$$
uniformly with respect to the parameter $\xi \in \rr$.
We can therefore find a value of the parameter $0 <\eps_0 \leq 1$ and a new positive constant $C>0$ such that for all $u \in \mathscr{S}(\rr_v)$, $0<\eps \leq \eps_0$, $\xi \in \rr$,
\begin{equation}\label{eq7fp2}
\textrm{Re}(Pu,(1-\eps G)u)_{L^2} \geq \frac{1}{2}\tilde{c}_0\|\mathcal{H}^{\frac{s}{2}}u\|_{L^2}^2+\eps(\Psi^wu,u)_{L^2}-C\|u\|_{L^2}^2.
\end{equation}
By considering separately the two regions of the phase space where
$$\eta^2+v^2 \lesssim \lambda^{\frac{2}{2s+1}}, \qquad  \eta^2+v^2 \gtrsim \lambda^{\frac{2}{2s+1}},$$
according to the support of the function
$$\psi\Big(\frac{\eta^2+v^2}{\lambda^{\frac{2}{2s+1}}}\Big),$$
we notice that one can find a positive constant $c_1 >0$ such that for all $(v,\eta,\xi) \in \rr^{3}$,
\begin{equation} \label{crucial}
\frac{\xi^2}{2\lambda^{\frac{2s+2}{2s+1}}}\psi\Big(\frac{\eta^2+v^2}{\lambda^{\frac{2}{2s+1}}}\Big) +\Big(1-\psi\Big(\frac{\eta^2+v^2}{\lambda^{\frac{2}{2s+1}}}\Big)\Big)\Big(1+\eta^2+\frac{v^2}{4}\Big)^{s}
\geq c_1 \lambda^{\frac{2s}{2s+1}} \geq c_1 \langle \xi \rangle^{\frac{2s}{2s+1}}.
\end{equation}
We notice from Lemma~\ref{2.2fp} and (\ref{ya2}) that
$$\frac{\xi^2}{2\lambda^{\frac{2s+2}{2s+1}}}\psi\Big(\frac{\eta^2+v^2}{\lambda^{\frac{2}{2s+1}}}\Big), \quad \Big(1-\psi\Big(\frac{\eta^2+v^2}{\lambda^{\frac{2}{2s+1}}}\Big)\Big)\Big(1+\eta^2+\frac{v^2}{4}\Big)^{s}, \quad  \langle \xi \rangle^{\frac{2s}{2s+1}},$$
are all first order symbols in the class $S(M,\Gamma_1)$ uniformly with respect to the parameter $\xi \in \rr$.
It follows from (\ref{fili5}), (\ref{crucial}) and the G\aa rding inequality that there exists a positive constant $c_2>0$ such that for all $\xi \in \rr$, $u \in \mathscr{S}(\rr_v)$,
\begin{multline}\label{q2}
(\Psi^wu,u)_{L^2}+\Big(\textrm{Op}^w\Big(\Big(1-\psi\Big(\frac{\eta^2+v^2}{\lambda^{\frac{2}{2s+1}}}\Big)\Big)\Big(1+\eta^2+\frac{v^2}{4}\Big)^{s}\Big)u,u\Big)_{L^2} \\ \geq c_1 \langle \xi \rangle^{\frac{2s}{2s+1}}\|u\|_{L^2}^2-c_2\|u\|_{L^2}^2.
\end{multline}
On the other hand, we notice from Lemma~\ref{2.2fp}, (\ref{hm1}) and (\ref{hm1.2}) that
$$1-\psi\Big(\frac{\eta^2+v^2}{\lambda^{\frac{2}{2s+1}}}\Big) \in S(1,\Gamma_1) \subset S(1,\Gamma_0),$$
uniformly with respect to the parameter $\xi \in \rr$.
By using that
$$\Big(1+\eta^2+\frac{v^2}{4}\Big)^{s}  \in S(\langle (v,\eta) \rangle^{2s},\Gamma_0)\subset S(\langle (v,\eta) \rangle^{2},\Gamma_0),$$
since $0<s<1$, it follows that
$$\Big(1-\psi\Big(\frac{\eta^2+v^2}{\lambda^{\frac{2}{2s+1}}}\Big)\Big)\Big(1+\eta^2+\frac{v^2}{4}\Big)^{s}  \in  S(\langle (v,\eta) \rangle^{2},\Gamma_0),$$
uniformly with respect to the parameter $\xi \in \rr$.
The two symbols
$$\Big(1+\eta^2+\frac{v^2}{4}\Big)^{s}, \quad \Big(1-\psi\Big(\frac{\eta^2+v^2}{\lambda^{\frac{2}{2s+1}}}\Big)\Big)\Big(1+\eta^2+\frac{v^2}{4}\Big)^{s},$$
are therefore first order symbols in the class $S(\langle (v,\eta) \rangle^{2},\Gamma_0)$.
Starting from the estimate
$$\Big(1+\eta^2+\frac{v^2}{4}\Big)^{s} \geq \Big(1-\psi\Big(\frac{\eta^2+v^2}{\lambda^{\frac{2}{2s+1}}}\Big)\Big)\Big(1+\eta^2+\frac{v^2}{4}\Big)^{s},$$
another use of the G\aa rding inequality  shows that there exists a positive constant $c_3>0$ such that for all $\xi \in \rr$, $u \in \mathscr{S}(\rr_v)$,
$$\Big(\textrm{Op}^w\Big(\Big(1+\eta^2+\frac{v^2}{4}\Big)^{s}\Big)u,u\Big)_{L^2} \geq \Big(\textrm{Op}^w\Big(\Big(1-\psi\Big(\frac{\eta^2+v^2}{\lambda^{\frac{2}{2s+1}}}\Big)\Big)\Big(1+\eta^2+\frac{v^2}{4}\Big)^{s}\Big)u,u\Big)_{L^2} -c_3\|u\|_{L^2}^2.$$
By proceeding as in (\ref{hm8.1}) and (\ref{hm8}), we obtain that there exists a positive constant $c_4>0$ such that for all $\xi \in \rr$, $u \in \mathscr{S}(\rr_v)$,
\begin{equation}\label{q6}
\Big(\textrm{Op}^w\Big(\Big(1-\psi\Big(\frac{\eta^2+v^2}{\lambda^{\frac{2}{2s+1}}}\Big)\Big)\Big(1+\eta^2+\frac{v^2}{4}\Big)^{s}\Big)u,u\Big)_{L^2} \leq c_4\|\mathcal{H}^{\frac{s}{2}}u\|_{L^2}^2.
\end{equation}
We deduce from (\ref{eq7fp2}), (\ref{q2}) and (\ref{q6}) that for all $u \in \mathscr{S}(\rr_v)$, $0<\eps \leq \eps_0$, $\xi \in \rr$,
$$\textrm{Re}(Pu,(1-\eps G)u)_{L^2} \geq (2^{-1}\tilde{c}_0-c_4\eps)\|\mathcal{H}^{\frac{s}{2}}u\|_{L^2}^2+c_1\eps \langle \xi \rangle^{\frac{2s}{2s+1}}\|u\|_{L^2}^2-(C+c_2 \eps)\|u\|_{L^2}^2.$$
We can therefore find some new positive constants $0<\eps_0 \leq 1$, $c_1, c_2>0$ such that for all $0<\eps \leq \eps_0$, $u \in \mathscr{S}(\rr_v)$, $\xi \in \rr$,
$$\textrm{Re}(Pu,(1-\eps G)u)_{L^2} \geq c_1\|\mathcal{H}^{\frac{s}{2}}u\|_{L^2}^2+c_1\eps \langle \xi \rangle^{\frac{2s}{2s+1}}\|u\|_{L^2}^2-c_2\|u\|_{L^2}^2.$$
This ends the proof of Lemma~\ref{KFP2}.
\section{Local existence and uniqueness result}\label{loun}
Following~\cite{AM}, we aim at establishing the local existence and the uniqueness for the Cauchy problem (\ref{w7}) with small initial $H^{(1,0)}(\rr_{x,v}^2)$-fluctuations. To that end, we
begin by considering a linear equation with a source.
\subsection{Local existence for a linear equation}
We begin by proving the following existence result:
\begin{lemma}\label{tr11}
There exist some positive constants $c_0 > 1$, $\eps_0>0$ such that for all $T>0$, $g_0 \in H^{(1,0)}(\rr_{x,v}^2)$, $f \in L^{\infty}([0,T],H^{(1,0)}(\rr_{x,v}^2))$ satisfying
$$\|f\|_{L^{\infty}([0,T],H^{(1,0)}(\rr_{x,v}^2))} \leq \eps_0,$$
the Cauchy problem
$$\begin{cases}
\partial_tg+v\partial_{x}g+\mathcal{K}g=\Gamma(f,g),\\
g|_{t=0}=g_0,
\end{cases}$$
admits a weak solution $g \in L^{\infty}([0,T],H^{(1,0)}(\rr_{x,v}^2))$ satisfying
$$\|g\|_{L^{\infty}([0,T],H^{(1,0)}(\rr_{x,v}^2))}+\|\mathcal{H}^{\frac{s}{2}}g\|_{L^{2}([0,T],H^{(1,0)}(\rr_{x,v}^2))} \leq c_0e^{3T}\|g_0\|_{(1,0)}.$$
\end{lemma}

\begin{proof}
Let $r>1/2$, $T>0$.
We consider
$$\mathcal{Q}=-\partial_t+(v\partial_{x}+\mathcal{K}-\Gamma(f,\cdot))^*,$$
where the adjoint operator is taken with respect to the scalar product in $H^{(r,0)}(\rr_{x,v}^2)$. We deduce from (\ref{tr4}) and Lemma~\ref{proposition222} that for all $h \in C^{\infty}([0,T],\mathscr{S}(\rr_{x,v}^2))$, with $h(T)=0$ and~$0 \leq t \leq T$,
\begin{align*}
& \ \textrm{Re}\big(h(t),\mathcal{Q}h(t)\big)_{(r,0)}\\ \notag
=& \ -\frac{1}{2}\frac{d}{dt}(\|h\|_{(r,0)}^2)+\textrm{Re}(v\partial_{x}h,h)_{(r,0)}+\textrm{Re}(\mathcal{K}h,h)_{(r,0)}-\textrm{Re}(\Gamma(f,h),h)_{(r,0)} \\ \notag
\geq& \ -\frac{1}{2}\frac{d}{dt}\big(\|h(t)\|_{(r,0)}^2\big)+\frac{1}{C}\|\mathcal{H}^{\frac{s}{2}}h(t)\|_{(r,0)}^2-\|h(t)\|_{(r,0)}^2-C_r\|f(t)\|_{(r,0)}\|\mathcal{H}^{\frac{s}{2}}h(t)\|_{(r,0)}^2,
\end{align*}
since $\mathcal{K}$ is a selfadjoint operator and $\textrm{Re}(v\partial_{x}h,h)_{(r,0)}=0$.
When
\begin{equation}\label{tr5}
\|f\|_{L^{\infty}([0,T],H^{(r,0)}(\rr_{x,v}^2))} \leq \frac{1}{4CC_r},
\end{equation}
it follows from the Cauchy-Schwarz inequality that for all $0 \leq t \leq T$,
$$-\frac{d}{dt}\big(\|h(t)\|_{(r,0)}^2\big)+\frac{3}{2C}\|\mathcal{H}^{\frac{s}{2}}h(t)\|_{(r,0)}^2 \leq 2\|h(t)\|_{(r,0)}\|\mathcal{Q}h(t)\|_{(r,0)}+2\|h(t)\|_{(r,0)}^2,$$
that is
$$-\frac{d}{dt}\big(e^{2t}\|h(t)\|_{(r,0)}^2\big)+\frac{3}{2C}e^{2t}\|\mathcal{H}^{\frac{s}{2}}h(t)\|_{(r,0)}^2 \leq 2e^{2t}\|h(t)\|_{(r,0)}\|\mathcal{Q}h(t)\|_{(r,0)}.$$
We obtain that for all $0 \leq t \leq T$,
\begin{align*}
& \ \|h(t)\|_{(r,0)}^2+\frac{3}{2C}\|\mathcal{H}^{\frac{s}{2}}h\|_{L^{2}([t,T],H^{(r,0)}(\rr_{x,v}^2))}^2 \leq  \|h(t)\|_{(r,0)}^2+\frac{3}{2C}\int_{t}^Te^{2(\tau-t)}\|\mathcal{H}^{\frac{s}{2}}h(\tau)\|_{(r,0)}^2d\tau \\
\leq & \  2\int_t^Te^{2(\tau-t)}\|h(\tau)\|_{(r,0)}\|\mathcal{Q}h(\tau)\|_{(r,0)}d\tau \leq 2e^{2T}\|h\|_{L^{\infty}([0,T],H^{(r,0)}(\rr_{x,v}^2))}\|\mathcal{Q}h\|_{L^{1}([0,T],H^{(r,0)}(\rr_{x,v}^2))},
\end{align*}
since $h(T)=0$.
We deduce that for all $h \in C^{\infty}([0,T],\mathscr{S}(\rr_{x,v}^2))$, with $h(T)=0$,
\begin{equation}\label{tr7}
\|h\|_{L^{\infty}([0,T],H^{(r,0)}(\rr_{x,v}^2))} \leq 2e^{2T}\|\mathcal{Q}h\|_{L^1([0,T],H^{(r,0)}(\rr_{x,v}^2))}.
\end{equation}
We consider the vector subspace
$$\mathcal{W}=\{w=\mathcal{Q}h : h \in C^{\infty}([0,T],\mathscr{S}(\rr_{x,v}^2)), \ h(T)=0\} \subset L^{1}([0,T],H^{(r,0)}(\rr_{x,v}^2)).$$
This inclusion holds since according to Lemma~\ref{proposition2225},
\begin{multline*}
|(\Gamma(f,\cdot)^*h,g)_{(r,0)}|=|(h,\Gamma(f,g))_{(r,0)}|=|(\mathcal{H}^{s}h,\mathcal{H}^{-s}\Gamma(f,g))_{(r,0)}|\\ \leq \|\mathcal{H}^{-s}\Gamma(f,g)\|_{(r,0)}\|\mathcal{H}^{s}h\|_{(r,0)}
 \leq c_r\|f\|_{(r,0)} \|g\|_{(r,0)}\|\mathcal{H}^{s}h\|_{(r,0)},
\end{multline*}
we have
$$\|\Gamma(f,\cdot)^*h\|_{(r,0)}  \leq c_r\|f\|_{(r,0)}\|\mathcal{H}^{s}h\|_{(r,0)}.$$
Let $g_0 \in H^{(r,0)}(\rr_{x,v}^2)$. We define the linear functional
\begin{align*}
\mathcal{G} \ : \qquad &\mathcal{W} \longrightarrow \cc\\ \notag
w=&\mathcal{Q}h \mapsto (g_0,h(0))_{(r,0)},
\end{align*}
where $h \in C^{\infty}([0,T],\mathscr{S}(\rr_{x,v}^2))$, with $h(T)=0$.
According to (\ref{tr7}), the operator $\mathcal{Q}$ is injective. The linear functional $\mathcal{G}$ is therefore well-defined. It follows from (\ref{tr7}) that $\mathcal{G}$ is a continuous linear form on $(\mathcal{W},\|\cdot\|_{L^{1}([0,T],H^{(r,0)}(\rr_{x,v}^2))})$,
\begin{multline*}
|\mathcal{G}(w)| \leq \|g_0\|_{(r,0)}\|h(0)\|_{(r,0)} \leq \|g_0\|_{(r,0)}\|h\|_{L^{\infty}([0,T],H^{(r,0)}(\rr_{x,v}^2))} \\
\leq 2e^{2T}\|g_0\|_{(r,0)}\|\mathcal{Q}h\|_{L^{1}([0,T],H^{(r,0)}(\rr_{x,v}^2))}= 2e^{2T}\|g_0\|_{(r,0)}\|w\|_{L^{1}([0,T],H^{(r,0)}(\rr_{x,v}^2))}.
\end{multline*}
By using the Hahn-Banach theorem, $\mathcal{G}$ may be extended as a continuous linear form on
$$L^{1}([0,T],H^{(r,0)}(\rr_{x,v}^2)),$$
with a norm smaller than $2e^{2T}\|g_0\|_{(r,0)}$. It follows that there exists $g \in L^{\infty}([0,T],H^{(r,0)}(\rr_{x,v}^2))$ satisfying
$$\|g\|_{L^{\infty}([0,T],H^{(r,0)}(\rr_{x,v}^2))} \leq 2e^{2T}\|g_0\|_{(r,0)},$$
such that
$$\forall w \in L^{1}([0,T],H^{(r,0)}(\rr_{x,v}^2)), \quad \mathcal{G}(w)=\int_0^T(g(t),w(t))_{(r,0)}dt.$$
This implies in particular that for all $h \in \mathscr{F}=C_0^{\infty}((-\infty,T),\mathscr S(\rr_{x,v}^2))$,
$$\mathcal{G}(\mathcal{Q}h)=\int_0^T(g(t),\mathcal{Q}h(t))_{(r,0)}dt=(g_0,h(0))_{(r,0)}.$$
This shows that $g \in L^{\infty}([0,T],H^{(r,0)}(\rr_{x,v}^2))$ is a weak solution of the Cauchy problem
\begin{equation}\label{cl1ff1}
\begin{cases}
\partial_tg+v\partial_{x}g+\mathcal{K}g=\Gamma(f,g),\\
g|_{t=0}=g_0.
\end{cases}
\end{equation}
We deduce from (\ref{sal1}), (\ref{3.llkkb}), (\ref{tr5}) and Lemma~\ref{proposition2225} that
\begin{equation}\label{li1}
\mathcal{H}^{-s}\mathcal{K}g \in L^{\infty}([0,T],H^{(r,0)}(\rr_{x,v}^2)), \quad \mathcal{H}^{-s}\Gamma(f,g) \in L^{\infty}([0,T],H^{(r,0)}(\rr_{x,v}^2)),
\end{equation}
since $f \in L^{\infty}([0,T],H^{(r,0)}(\rr_{x,v}^2))$. We define
\begin{equation}\label{li2}
g_{\delta}=(1+\delta \sqrt{\mathcal{H}}+\delta \langle D_x \rangle)^{-1}g, \quad 0<\delta \leq 1.
\end{equation}
We notice that
$$(1+\delta \sqrt{\mathcal{H}}+\delta \langle D_x \rangle)g_{\delta} \in L^{\infty}([0,T],H^{(r,0)}(\rr_{x,v}^2)) \subset L^{2}([0,T],H^{(r,0)}(\rr_{x,v}^2)).$$
On the other hand, we deduce from (\ref{cl1ff1}) that
\begin{multline*}
(1+\delta \sqrt{\mathcal{H}}+\delta \langle D_x \rangle)^{-1}\partial_tg_{\delta}=(1+\delta \sqrt{\mathcal{H}}+\delta \langle D_x \rangle)^{-2}\mathcal{H}^{s}\mathcal{H}^{-s}\Gamma(f,g)\\ -(1+\delta \sqrt{\mathcal{H}}+\delta \langle D_x \rangle)^{-2}(v\partial_{x}g) -(1+\delta \sqrt{\mathcal{H}}+\delta \langle D_x \rangle)^{-2}\mathcal{H}^{s}\mathcal{H}^{-s}\mathcal{K}g.
\end{multline*}
It follows from (\ref{li1}) that
$$(1+\delta \sqrt{\mathcal{H}}+\delta \langle D_x \rangle)^{-1}\partial_tg_{\delta} \in  L^{\infty}([0,T],H^{(r,0)}(\rr_{x,v}^2)) \subset L^{2}([0,T],H^{(r,0)}(\rr_{x,v}^2)),$$
since $0<s<1$. A direct adaptation of Theorem~3 in~\cite{evans} (Section 5.9) shows that the mapping
$$t \mapsto \|g_{\delta}(t)\|_{(r,0)}^2,$$
is absolutely continuous with
\begin{equation}\label{li6}
\frac{d}{dt}\|g_{\delta}(t)\|_{(r,0)}^2=2\textrm{Re}(\partial_tg_{\delta}(t),g_{\delta}(t))_{(r,0)}.
\end{equation}
By using the multiplier in $H^{(r,0)}(\rr_{x,v}^2)$,
$$h_{\delta}=(1+\delta \sqrt{\mathcal{H}}+\delta \langle D_x \rangle)^{-2}g,$$
we deduce from (\ref{sal1}), (\ref{li2}) and (\ref{li6}) that
\begin{align*}
& \ \frac{1}{2}\frac{d}{dt}(\|g_{\delta}(t)\|_{(r,0)}^2)+\textrm{Re}(\mathcal{K}g_{\delta},g_{\delta})_{(r,0)}+\textrm{Re}(v\partial_{x}g_{\delta},g_{\delta})_{(r,0)} \\
& \ +\textrm{Re}\big([(1+\delta \sqrt{\mathcal{H}}+\delta \langle D_x \rangle)^{-1},v](1+\delta \sqrt{\mathcal{H}}+\delta \langle D_x \rangle)\partial_{x}g_{\delta},g_{\delta}\big)_{(r,0)}\\
= & \ \textrm{Re}\big((1+\delta \sqrt{\mathcal{H}}+\delta \langle D_x \rangle)^{-1}\Gamma\big(f,(1+\delta \sqrt{\mathcal{H}}+\delta \langle D_x \rangle)g_{\delta}\big),g_{\delta}\big)_{(r,0)},
\end{align*}
since $[(1+\delta \sqrt{\mathcal{H}}+\delta \langle D_x \rangle)^{-1},\mathcal{K}]=0$.
We deduce from (\ref{tr4}) that for all $0 \leq t \leq T$, $0<\delta \leq 1$,
\begin{align}\label{li7}
& \ \frac{1}{2}\frac{d}{dt}(\|g_{\delta}(t)\|_{(r,0)}^2)+\frac{1}{C}\|\mathcal{H}^{\frac{s}{2}}g_{\delta}(t)\|_{(r,0)}^2-\|g_{\delta}(t)\|_{(r,0)}^2 \\ \notag
\leq & \ \big|\big([(1+\delta \sqrt{\mathcal{H}}+\delta \langle D_x \rangle)^{-1},v](1+\delta \sqrt{\mathcal{H}}+\delta \langle D_x \rangle)\partial_{x}g_{\delta},g_{\delta}\big)_{(r,0)}\big|\\ \notag
+ & \ \big|\big((1+\delta \sqrt{\mathcal{H}}+\delta \langle D_x \rangle)^{-1}\Gamma\big(f,(1+\delta \sqrt{\mathcal{H}})g_{\delta}\big),g_{\delta}\big)_{(r,0)}\big|\\ \notag
+ & \ \big|\big((1+\delta \sqrt{\mathcal{H}}+\delta \langle D_x \rangle)^{-1}\Gamma\big(f,\delta \langle D_x \rangle g_{\delta}\big),g_{\delta}\big)_{(r,0)}\big|,
\end{align}
since $\mathcal{K}$ is a selfadjoint operator and $\textrm{Re}(v\partial_{x}g_{\delta},g_{\delta})_{(r,0)}=0$.
We deduce from Lemma~\ref{proposition222} with $t=0$, $\delta_1=\delta$, $j_1=0$ and $j_2=1$ that for all $0<\delta \leq 1$,
\begin{equation}\label{li8}
 \big|\big((1+\delta \sqrt{\mathcal{H}}+\delta \langle D_x \rangle)^{-1}\Gamma\big(f,(1+\delta \sqrt{\mathcal{H}})g_{\delta}\big),g_{\delta}\big)_{(r,0)}\big| \leq C_r\|f\|_{(r,0)}\|\mathcal{H}^\frac{s}{2}g_{\delta}\|_{(r,0)}^2.
\end{equation}
On the other hand, it follows from Lemma~\ref{proposition222} that for all $0<\delta \leq 1$,
\begin{align}\label{li9}
& \ \big|\big((1+\delta \sqrt{\mathcal{H}}+\delta \langle D_x \rangle)^{-1}\Gamma\big(f, \delta \langle D_x \rangle g_{\delta}\big),g_{\delta}\big)_{(r,0)}\big|\\ \notag
= & \ \big|\big(\Gamma\big(f,\langle D_x \rangle g_{\delta}\big),\delta\langle D_x \rangle^{2r}(1+\delta \sqrt{\mathcal{H}}+\delta \langle D_x \rangle)^{-1}g_{\delta}\big)_{L^2(\rr_{x,v}^2)}\big| \\ \notag
\leq & \ C_r\|f\|_{(r,0)}\|\langle D_x \rangle \mathcal{H}^\frac{s}{2}g_{\delta}\|_{L^2(\rr_{x,v}^2)}\|\delta\langle D_x \rangle^{2r}(1+\delta \sqrt{\mathcal{H}}+\delta \langle D_x \rangle)^{-1}\mathcal{H}^\frac{s}{2}g_{\delta}\|_{L^2(\rr_{x,v}^2)} \\ \notag
\leq & \ C_r\|f\|_{(r,0)}\|\mathcal{H}^\frac{s}{2}g_{\delta}\|_{(1,0)}\|\mathcal{H}^\frac{s}{2}g_{\delta}\|_{(2r-1,0)},
\end{align}
since
$$\|\delta\langle D_x \rangle(1+\delta \sqrt{\mathcal{H}}+\delta \langle D_x \rangle)^{-1}\|_{\mathcal{L}(L^2(\rr_{x,v}^2))} \leq 1.$$
Next, we check that the operator $[(1+\delta \sqrt{\mathcal{H}}+\delta \langle D_x \rangle)^{-1},v](1+\delta \sqrt{\mathcal{H}}+\delta \langle D_x \rangle)\partial_{x}$ is uniformly bounded on $L^2(\rr_{x,v}^2)$ with respect to the parameter $0<\delta \leq 1$. Let $f \in \mathscr{S}(\rr_{x,v}^2)$. We decompose this function into the Hermite basis in the velocity variable
$$f(x,v)=\sum_{n=0}^{+\infty}f_n(x) \psi_n(v), \quad f_n(x)=(f(x,\cdot),\psi_n)_{L^2(\rr_v)}.$$
By using the creation and annihilation operators $v=A_++A_-,$
we deduce from (\ref{ge1}) and (\ref{ge2}) that
\begin{align*}
& \ [(1+\delta \sqrt{\mathcal{H}}+\delta \langle D_x \rangle)^{-1},v](1+\delta \sqrt{\mathcal{H}}+\delta \langle D_x \rangle)\partial_{x}f\\
= &\ \sum_{n=0}^{+\infty}\sqrt{n+1}\frac{\delta(\sqrt{n+\frac{1}{2}}-\sqrt{n+\frac{3}{2}})}{1+\delta \sqrt{n+\frac{3}{2}}+\delta \langle D_x \rangle}\partial_{x}f_n(x) \psi_{n+1}(v)\\
+ & \ \sum_{n=0}^{+\infty}\sqrt{n}\frac{\delta(\sqrt{n+\frac{1}{2}}-\sqrt{n-\frac{1}{2}})}{1+\delta \sqrt{n-\frac{1}{2}}+\delta \langle D_x \rangle}\partial_{x}f_n(x) \psi_{n-1}(v).
\end{align*}
It follows that
\begin{align*}
& \ \|[(1+\delta \sqrt{\mathcal{H}}+\delta \langle D_x \rangle)^{-1},v](1+\delta \sqrt{\mathcal{H}}+\delta \langle D_x \rangle)\partial_{x}f\|_{L^2(\rr_{x,v}^2)}^2 \\
\leq & \ \sum_{n=0}^{+\infty}\frac{2n+2}{(\sqrt{n+\frac{1}{2}}+\sqrt{n+\frac{3}{2}})^2}\Big\|\Big(1+\delta \sqrt{n+\frac{3}{2}}+\delta \langle D_x \rangle\Big)^{-1}\delta \partial_{x}f_n\Big\|_{L^2(\rr_x)}^2\\
+ & \ \sum_{n=1}^{+\infty}\frac{2n}{(\sqrt{n+\frac{1}{2}}+\sqrt{n-\frac{1}{2}})^2}\Big\|\Big(1+\delta \sqrt{n-\frac{1}{2}}+\delta \langle D_x \rangle\Big)^{-1}\delta \partial_{x}f_n\Big\|_{L^2(\rr_x)}^2.
\end{align*}
This implies that
\begin{equation}\label{li10}
\|[(1+\delta \sqrt{\mathcal{H}}+\delta \langle D_x \rangle)^{-1},v](1+\delta \sqrt{\mathcal{H}}+\delta \langle D_x \rangle)\partial_{x}f\|_{L^2(\rr_{x,v}^2)}^2 \leq 4\|f\|_{L^2(\rr_{x,v}^2)}^2.
\end{equation}
By taking $r=1$, we deduce from (\ref{li7}), (\ref{li8}), (\ref{li9}) and (\ref{li10}) that for all $0 \leq t \leq T$, $0<\delta \leq 1$,
$$\frac{1}{2}\frac{d}{dt}(\|g_{\delta}(t)\|_{(1,0)}^2)+\frac{1}{C}\|\mathcal{H}^{\frac{s}{2}}g_{\delta}(t)\|_{(1,0)}^2 \leq  3\|g_{\delta}(t)\|_{(1,0)}^2+2C_1\|f(t)\|_{(1,0)}\|\mathcal{H}^\frac{s}{2}g_{\delta}(t)\|_{(1,0)}^2.$$
When the source satisfies (\ref{tr5}) with $r=1$, we obtain that for all $0 \leq t \leq T$, $0<\delta \leq 1$,
$$\frac{d}{dt}(\|g_{\delta}(t)\|_{(1,0)}^2)+\frac{1}{C}\|\mathcal{H}^{\frac{s}{2}}g_{\delta}(t)\|_{(1,0)}^2 \leq 6\|g_{\delta}(t)\|_{(1,0)}^2.$$
It follows that for all $0 \leq t \leq T$, $0<\delta \leq 1$,
$$\|g_{\delta}(t)\|_{(1,0)}^2+\frac{1}{C}\int_0^te^{6(t-\tau)}\|\mathcal{H}^{\frac{s}{2}}g_{\delta}(\tau)\|_{(1,0)}^2d\tau \leq e^{6t}\|(1+\delta \sqrt{\mathcal{H}}+\delta \langle D_x \rangle)^{-1}g_0\|_{(1,0)}^2.$$
This implies that for all $0<\delta \leq 1$,
$$\|g_{\delta}\|_{L^{\infty}([0,T],H^{(1,0)}(\rr_{x,v}^2))}^2+\|\mathcal{H}^{\frac{s}{2}}g_{\delta}\|_{L^{2}([0,T],H^{(1,0)}(\rr_{x,v}^2))}^2 \leq (C+1)e^{6T}\|g_0\|_{(1,0)}^2.$$
Finally, by writing
$$\|g_{\delta}(t)\|_{(1,0)}^2=\frac{1}{2\pi}\sum_{n=0}^{+\infty}\int_{\rr}\langle \xi \rangle^2 \Big(1+\delta \sqrt{n+\frac{1}{2}}+\delta \langle \xi \rangle\Big)^{-2}|\mathcal{F}_xg_n(t,\xi)|^2d\xi,$$
$$\|\mathcal{H}^{\frac{s}{2}}g_{\delta}\|_{L^{2}([0,T],H^{(1,0)}(\rr_{x,v}^2))}^2=\int_0^T\sum_{n=0}^{+\infty}\Big(n+\frac{1}{2}\Big)^s
\int_{\rr}\langle \xi \rangle^2 \Big(1+\delta \sqrt{n+\frac{1}{2}}+\delta \langle \xi \rangle\Big)^{-2}|\mathcal{F}_xg_n(t,\xi)|^2\frac{d\xi dt}{2\pi},$$
with $g_n=(g(t,x,\cdot),\psi_n)_{L^2(\rr_v)}$, where $\mathcal{F}_x$ denotes the partial Fourier transform in the position variable,
we deduce from the monotone convergence theorem by passing to the limit when $\delta \to 0_+$ that
$$\|g\|_{L^{\infty}([0,T],H^{(1,0)}(\rr_{x,v}^2))}^2+\|\mathcal{H}^{\frac{s}{2}}g\|_{L^{2}([0,T],H^{(1,0)}(\rr_{x,v}^2))}^2 \leq (C+1)e^{6T}\|g_0\|_{(1,0)}^2.$$
This ends the proof of Lemma~\ref{tr11}.
\end{proof}

\subsection{Convergence of approximate solutions}

We prove the existence of a local solution for the Cauchy problem associated to the spatially inhomogeneous Kac equation with small initial fluctuations belonging to $H^{(1,0)}(\rr_{x,v}^2)$ as the limit of a sequence of approximate solutions:

\begin{theorem}\label{qw13}
Let $T>0$. Then, there exist some positive constants $c_0 > 1$, $\eps_0>0$ such that for all $g_0 \in H^{(1,0)}(\rr_{x,v}^2)$ satisfying
$$\|g_0\|_{(1,0)} \leq \eps_0,$$
the Cauchy problem associated to the spatially inhomogeneous Kac equation
\begin{equation}\label{qw13.1}
\begin{cases}
\partial_tg+v\partial_{x}g+\mathcal{K}g=\Gamma(g,g),\\
g|_{t=0}=g_0,
\end{cases}
\end{equation}
admits a weak solution $g \in L^{\infty}([0,T],H^{(1,0)}(\rr_{x,v}^2)) $ satisfying
$$\|g\|_{L^{\infty}([0,T],H^{(1,0)}(\rr_{x,v}^2))}+\|\mathcal{H}^{\frac{s}{2}}g\|_{L^{2}([0,T],H^{(1,0)}(\rr_{x,v}^2))} \leq c_0\|g_0\|_{(1,0)}.$$
\end{theorem}

\begin{proof}
Let $0<\lambda<1$, $T>0$ and $g_0  \in H^{(1,0)}(\rr_{x,v}^2)$ be an initial fluctuation satisfying
\begin{equation}\label{tr13}
\|g_0\|_{(1,0)} \leq \tilde{\eps}_0, \quad \textrm{with } \quad 0<\tilde{\eps}_0=\inf\Big(\frac{\eps_0}{c_0e^{3T}},\frac{1}{4CC_1c_0e^{3T}},\frac{\lambda}{2\sqrt{2C}C_1c_0e^{6T}}\Big) \leq \eps_0,
\end{equation}
where $c_0 > 1$, $\eps_0, C_1, C>0$ are the constants defined in (\ref{tr4}) and Lemmas~\ref{proposition222}, \ref{tr11}. We define
\begin{equation}\label{tr14}
\tilde{g}_0(t)=\exp\big(-\delta t\big(\sqrt{\mathcal{H}}+\langle D_x \rangle\big)^{\frac{2s}{2s+1}}\big)g_0, \quad 0 \leq t \leq T,
\end{equation}
with $0 \leq \delta \leq 1$.
We notice that
\begin{equation}\label{tr15}
\|\tilde{g}_0\|_{L^{\infty}([0,T],H^{(1,0)}(\rr_{x,v}^2))} \leq \|g_0\|_{(1,0)} \leq \eps_0.
\end{equation}
We deduce from (\ref{tr15}) and Lemma~\ref{tr11} that we can construct  a sequence of solutions $(\tilde{g}_n)_{n \geq 0}$ belonging to
$L^{\infty}([0,T],H^{(1,0)}(\rr_{x,v}^2))$ and satisfying the Cauchy problem
\begin{equation}\label{tr16}
\begin{cases}
\partial_t\tilde{g}_{n+1}+v\partial_{x}\tilde{g}_{n+1}+\mathcal{K}\tilde{g}_{n+1}=\Gamma(\tilde{g}_{n},\tilde{g}_{n+1}), \quad n \geq 0,\\
\tilde{g}_{n+1}|_{t=0}=g_0,
\end{cases}
\end{equation}
together with the estimates
\begin{equation}\label{tr16.1}
\|\tilde{g}_{n}\|_{L^{\infty}([0,T],H^{(1,0)}(\rr_{x,v}^2))}+\|\mathcal{H}^{\frac{s}{2}}\tilde{g}_{n}\|_{L^{2}([0,T],H^{(1,0)}(\rr_{x,v}^2))} \leq c_0e^{3T}\|g_0\|_{(1,0)} \leq \eps_0,
\end{equation}
for all $n \geq 1$.
Indeed, if we assume that for some $n \geq 0$,
$$\|\tilde{g}_n\|_{L^{\infty}([0,T],H^{(1,0)}(\rr_{x,v}^2))} \leq \eps_0,$$
we deduce from (\ref{tr13}) and Lemma~\ref{tr11} that there exists a solution
$$\tilde{g}_{n+1} \in L^{\infty}([0,T],H^{(1,0)}(\rr_{x,v}^2)),$$
to the Cauchy problem (\ref{tr16}) satisfying
$$\|\tilde{g}_{n+1}\|_{L^{\infty}([0,T],H^{(1,0)}(\rr_{x,v}^2))}+\|\mathcal{H}^{\frac{s}{2}}\tilde{g}_{n+1}\|_{L^{2}([0,T],H^{(1,0)}(\rr_{x,v}^2))} \leq c_0e^{3T}\|g_0\|_{(1,0)} \leq \eps_0.$$
Then, we consider the difference
\begin{equation}\label{tr17}
w_n=\tilde{g}_{n+1}-\tilde{g}_n,
\end{equation}
for any $n \geq 0$.
We deduce from (\ref{tr16}) that for all $n \geq 1$,
\begin{equation}\label{tr18}
\begin{cases}
\partial_tw_n+v\partial_{x}w_n+\mathcal{K}w_n=\Gamma(\tilde{g}_{n},w_{n})+\Gamma(w_{n-1},\tilde{g}_{n}),\\
w_n|_{t=0}=0.
\end{cases}
\end{equation}
We define
\begin{equation}\label{li2.1}
w_{n,\delta}=(1+\delta \sqrt{\mathcal{H}}+\delta \langle D_x \rangle)^{-1}w_n, \quad 0<\delta \leq 1.
\end{equation}
By using the multiplier $(1+\delta \sqrt{\mathcal{H}}+\delta \langle D_x \rangle)^{-2}w_n$ in $H^{(1,0)}(\rr_{x,v}^2)$,
we deduce from (\ref{sal1}) and (\ref{li2.1}) that
\begin{align*}
& \ \frac{1}{2}\frac{d}{dt}(\|w_{n,\delta}\|_{(1,0)}^2)+\textrm{Re}(\mathcal{K}w_{n,\delta},w_{n,\delta})_{(1,0)}+\textrm{Re}(v\partial_{x}w_{n,\delta},w_{n,\delta})_{(1,0)}
\\
& \ +\textrm{Re}\big([(1+\delta \sqrt{\mathcal{H}}+\delta \langle D_x \rangle)^{-1},v](1+\delta \sqrt{\mathcal{H}}+\delta \langle D_x \rangle)\partial_{x}w_{n,\delta},w_{n,\delta}\big)_{(1,0)}\\
 = & \  \textrm{Re}\big((1+\delta \sqrt{\mathcal{H}}+\delta \langle D_x \rangle)^{-1}\Gamma\big(\tilde{g}_{n},(1+\delta \sqrt{\mathcal{H}}+\delta \langle D_x \rangle)w_{n,\delta}\big),w_{n,\delta}\big)_{(1,0)}\\
 & \ + \textrm{Re}\big((1+\delta \sqrt{\mathcal{H}}+\delta \langle D_x \rangle)^{-1}\Gamma\big((1+\delta \sqrt{\mathcal{H}}+\delta \langle D_x \rangle)w_{n-1,\delta},\tilde{g}_{n}\big),w_{n,\delta}\big)_{(1,0)},
\end{align*}
since $[(1+\delta \sqrt{\mathcal{H}}+\delta \langle D_x \rangle)^{-1},\mathcal{K}]=0$.
We deduce from (\ref{tr4}) and (\ref{li10}) that for all $0 \leq t \leq T$, $0<\delta \leq 1$,
\begin{align}\label{li7.1}
& \ \frac{1}{2}\frac{d}{dt}(\|w_{n,\delta}(t)\|_{(1,0)}^2)+\frac{1}{C}\|\mathcal{H}^{\frac{s}{2}}w_{n,\delta}(t)\|_{(1,0)}^2-\|w_{n,\delta}(t)\|_{(1,0)}^2 \\ \notag
\leq & \ 2\|w_{n,\delta}(t)\|_{(1,0)}^2+ \big|\big((1+\delta \sqrt{\mathcal{H}}+\delta \langle D_x \rangle)^{-1}\Gamma\big(\tilde{g}_{n},(1+\delta \sqrt{\mathcal{H}})w_{n,\delta}\big),w_{n,\delta}\big)_{(1,0)}\big|\\ \notag
+ & \ \big|\big((1+\delta \sqrt{\mathcal{H}}+\delta \langle D_x \rangle)^{-1}\Gamma\big(\tilde{g}_{n},\delta \langle D_x \rangle w_{n,\delta}\big),w_{n,\delta}\big)_{(1,0)}\big| \\ \notag
+ & \ \big|\big((1+\delta \sqrt{\mathcal{H}}+\delta \langle D_x \rangle)^{-1}\Gamma\big((1+\delta \sqrt{\mathcal{H}})w_{n-1,\delta},\tilde{g}_{n}\big),w_{n,\delta}\big)_{(1,0)}\big|\\ \notag
+ & \ \big|\big((1+\delta \sqrt{\mathcal{H}}+\delta \langle D_x \rangle)^{-1}\Gamma\big(\delta \langle D_x \rangle w_{n-1,\delta},\tilde{g}_{n}\big),w_{n,\delta}\big)_{(1,0)}\big|,
\end{align}
since $\mathcal{K}$ is a selfadjoint operator and $\textrm{Re}(v\partial_{x}w_{n,\delta},w_{n,\delta})_{(1,0)}=0$.
We deduce from Lemma~\ref{proposition222} with $t=0$, $\delta_1=\delta$ and $(j_1,j_2)=(0,1)$ that for all $0<\delta \leq 1$,
\begin{equation}\label{li8.1}
\big|\big((1+\delta \sqrt{\mathcal{H}}+\delta \langle D_x \rangle)^{-1}\Gamma\big(\tilde{g}_{n},(1+\delta \sqrt{\mathcal{H}})w_{n,\delta}\big),w_{n,\delta}\big)_{(1,0)}\big|
 \leq C_1\|\tilde{g}_{n}\|_{(1,0)}\|\mathcal{H}^\frac{s}{2}w_{n,\delta}\|_{(1,0)}^2.
\end{equation}
We also deduce from Lemma~\ref{proposition222} with $t=0$, $\delta_1=\delta$ and $(j_1,j_2)=(1,0)$ that for all $0<\delta \leq 1$,
\begin{multline}\label{li8.1b}
\big|\big((1+\delta \sqrt{\mathcal{H}}+\delta \langle D_x \rangle)^{-1}\Gamma\big((1+\delta \sqrt{\mathcal{H}})w_{n-1,\delta},\tilde{g}_{n}\big),w_{n,\delta}\big)_{(1,0)}\big| \\
\leq C_1\|w_{n-1,\delta}\|_{(1,0)}\|\mathcal{H}^\frac{s}{2}\tilde{g}_{n}\|_{(1,0)}\|\mathcal{H}^\frac{s}{2}w_{n,\delta}\|_{(1,0)}.
\end{multline}
On the other hand, it follows from Lemma~\ref{proposition222} with $r=1$ that for all $0<\delta \leq 1$,
\begin{align}\label{li9.1}
& \ \big|\big((1+\delta \sqrt{\mathcal{H}}+\delta \langle D_x \rangle)^{-1}\Gamma\big(\tilde{g}_{n},\delta \langle D_x \rangle w_{n,\delta}\big),w_{n,\delta}\big)_{(1,0)}\big|\\ \notag
= & \ \big|\big(\Gamma\big(\tilde{g}_{n},\langle D_x \rangle w_{n,\delta}\big),\delta\langle D_x \rangle^{2}(1+\delta \sqrt{\mathcal{H}}+\delta \langle D_x \rangle)^{-1}w_{n,\delta}\big)_{L^2(\rr_{x,v}^2)}\big| \\ \notag
\leq & \ C_1\|\tilde{g}_{n}\|_{(1,0)}\|\langle D_x \rangle \mathcal{H}^\frac{s}{2}w_{n,\delta}\|_{L^2(\rr_{x,v}^2)}\|
\delta\langle D_x \rangle^{2}(1+\delta \sqrt{\mathcal{H}}+\delta \langle D_x \rangle)^{-1}\mathcal{H}^\frac{s}{2}w_{n,\delta}\|_{L^2(\rr_{x,v}^2)} \\ \notag
\leq & \ C_1\|\tilde{g}_{n}\|_{(1,0)}\|\mathcal{H}^\frac{s}{2}w_{n,\delta}\|_{(1,0)}^2,
\end{align}
since
\begin{equation}\label{li10.1}
\|\delta\langle D_x \rangle(1+\delta \sqrt{\mathcal{H}}+\delta \langle D_x \rangle)^{-1}\|_{\mathcal{L}(L^2(\rr_{x,v}^2))} \leq 1.
\end{equation}
We also deduce from (\ref{li10.1}) and Lemma~\ref{proposition222} with $r=1$ that for all $0<\delta \leq 1$,
\begin{align}\label{li11.1}
& \ \big|\big((1+\delta \sqrt{\mathcal{H}}+\delta \langle D_x \rangle)^{-1}\Gamma\big(\delta \langle D_x \rangle w_{n-1,\delta},\tilde{g}_{n}\big),w_{n,\delta}\big)_{(1,0)}\big|\\ \notag
= & \ \big|\big(\Gamma\big(\langle D_x \rangle w_{n-1,\delta},\tilde{g}_{n}\big),\delta \langle D_x \rangle^2(1+\delta \sqrt{\mathcal{H}}+\delta \langle D_x \rangle)^{-1}w_{n,\delta}\big)_{L^2(\rr_{x,v}^2)}\big| \\ \notag
\leq & \ C_1\|w_{n-1,\delta}\|_{(1,0)}\|\mathcal{H}^\frac{s}{2}\tilde{g}_{n}\|_{(1,0)}\|\delta \langle D_x \rangle^2(1+\delta \sqrt{\mathcal{H}}+\delta \langle D_x \rangle)^{-1}\mathcal{H}^\frac{s}{2}w_{n,\delta}\|_{L^2(\rr_{x,v}^2)}\\ \notag
\leq & \ C_1\|w_{n-1,\delta}\|_{(1,0)}\|\mathcal{H}^\frac{s}{2}\tilde{g}_{n}\|_{(1,0)}\|\mathcal{H}^\frac{s}{2}w_{n,\delta}\|_{(1,0)}.
\end{align}
It follows from (\ref{li7.1}), (\ref{li8.1}), (\ref{li8.1b}), (\ref{li9.1}) and (\ref{li11.1}) that
\begin{multline}\label{li12.1}
\frac{1}{2}\frac{d}{dt}(\|w_{n,\delta}(t)\|_{(1,0)}^2)+\frac{1}{C}\|\mathcal{H}^{\frac{s}{2}}w_{n,\delta}(t)\|_{(1,0)}^2  \leq  3\|w_{n,\delta}(t)\|_{(1,0)}^2
\\ +2C_1\|\tilde{g}_{n}(t)\|_{(1,0)}\|\mathcal{H}^\frac{s}{2}w_{n,\delta}(t)\|_{(1,0)}^2
+ 2C_1\|w_{n-1,\delta}(t)\|_{(1,0)}\|\mathcal{H}^\frac{s}{2}\tilde{g}_{n}(t)\|_{(1,0)}\|\mathcal{H}^\frac{s}{2}w_{n,\delta}(t)\|_{(1,0)}.
\end{multline}
We deduce from (\ref{tr16.1}) that
\begin{multline*}
\frac{1}{2}\frac{d}{dt}(\|w_{n,\delta}(t)\|_{(1,0)}^2)+\frac{3}{4C}\|\mathcal{H}^{\frac{s}{2}}w_{n,\delta}(t)\|_{(1,0)}^2
\leq 3\|w_{n,\delta}(t)\|_{(1,0)}^2 \\ + 2e^{3T}c_0C_1\|g_0\|_{(1,0)}\|\mathcal{H}^{\frac{s}{2}}w_{n,\delta}(t)\|_{(1,0)}^2
+4CC_1^2\|w_{n-1,\delta}(t)\|_{(1,0)}^2\|\mathcal{H}^{\frac{s}{2}}\tilde{g}_{n}(t)\|_{(1,0)}^2.
\end{multline*}
It follows from (\ref{tr13}) that
$$\frac{d}{dt}(\|w_{n,\delta}(t)\|_{(1,0)}^2)+\frac{1}{2C}\|\mathcal{H}^{\frac{s}{2}}w_{n,\delta}(t)\|_{(1,0)}^2
\leq 6\|w_{n,\delta}(t)\|_{(1,0)}^2
+8CC_1^2\|w_{n-1,\delta}(t)\|_{(1,0)}^2\|\mathcal{H}^{\frac{s}{2}}\tilde{g}_{n}(t)\|_{(1,0)}^2.$$
We deduce from (\ref{tr13}), (\ref{tr16.1}), (\ref{tr18}) and (\ref{li2.1}) that for all $0 \leq t \leq T$, $0<\delta \leq 1$,
\begin{align*}
& \ \|w_{n,\delta}(t)\|_{(1,0)}^2+\frac{1}{2C}\int_{0}^te^{6(t-\tau)}\|\mathcal{H}^{\frac{s}{2}}w_{n,\delta}(\tau)\|_{(1,0)}^2d\tau\\
\leq  & \ 8CC_1^2\int_{0}^te^{6(t-\tau)}\|w_{n-1,\delta}(\tau)\|_{(1,0)}^2\|\mathcal{H}^{\frac{s}{2}}\tilde{g}_{n}(\tau)\|_{(1,0)}^2d\tau\\
\leq  & \ 8CC_1^2e^{6T}\|w_{n-1,\delta}\|_{L^{\infty}([0,T],H^{(1,0)}(\rr_{x,v}^2))}^2\|\mathcal{H}^{\frac{s}{2}}\tilde{g}_{n}\|_{L^{2}([0,T],H^{(1,0)}(\rr_{x,v}^2))}^2\\
\leq & \ \lambda^2 \|w_{n-1,\delta}\|_{L^{\infty}([0,T],H^{(1,0)}(\rr_{x,v}^2))}^2.
\end{align*}
It follows that for all $n \geq 1$, $0<\delta \leq 1$,
$$\|w_{n,\delta}\|_{L^{\infty}([0,T],H^{(1,0)}(\rr_{x,v}^2))} \leq \lambda \|w_{n-1,\delta}\|_{L^{\infty}([0,T],H^{(1,0)}(\rr_{x,v}^2))},$$
$$\|\mathcal{H}^{\frac{s}{2}}w_{n,\delta}\|_{L^{2}([0,T],H^{(1,0)}(\rr_{x,v}^2))} \leq \sqrt{2C}\lambda \|w_{n-1,\delta}\|_{L^{\infty}([0,T],H^{(1,0)}(\rr_{x,v}^2))}.$$
Recalling (\ref{li2.1}), we obtain that for all $n \geq 1$, $0<\delta \leq 1$,
\begin{multline*}
\|w_{n,\delta}\|_{L^{\infty}([0,T],H^{(1,0)}(\rr_{x,v}^2))} \leq \lambda^n \|(1+\delta \sqrt{\mathcal{H}}+\delta \langle D_x \rangle)^{-1}w_{0}\|_{L^{\infty}([0,T],H^{(1,0)}(\rr_{x,v}^2))}\\
\leq \lambda^n \|w_{0}\|_{L^{\infty}([0,T],H^{(1,0)}(\rr_{x,v}^2))},
\end{multline*}
\begin{multline*}
\|\mathcal{H}^{\frac{s}{2}}w_{n,\delta}\|_{L^{2}([0,T],H^{(1,0)}(\rr_{x,v}^2))} \leq \sqrt{2C}\lambda^n \|(1+\delta \sqrt{\mathcal{H}}+\delta \langle D_x \rangle)^{-1}w_{0}\|_{L^{\infty}([0,T],H^{(1,0)}(\rr_{x,v}^2))}\\
\leq \sqrt{2C}\lambda^n \|w_{0}\|_{L^{\infty}([0,T],H^{(1,0)}(\rr_{x,v}^2))}.
\end{multline*}
By passing to the limit when $\delta \to 0_+$, it follows from (\ref{li2.1}) and the monotone convergence theorem that for all $n \geq 1$,
\begin{equation}\label{wwe1}
\|w_{n}\|_{L^{\infty}([0,T],H^{(1,0)}(\rr_{x,v}^2))} \leq \lambda^n \|w_{0}\|_{L^{\infty}([0,T],H^{(1,0)}(\rr_{x,v}^2))},
\end{equation}
\begin{equation}\label{wwe2}
\|\mathcal{H}^{\frac{s}{2}}w_{n}\|_{L^{2}([0,T],H^{(1,0)}(\rr_{x,v}^2))} \leq \sqrt{2C}\lambda^n \|w_{0}\|_{L^{\infty}([0,T],H^{(1,0)}(\rr_{x,v}^2))}.
\end{equation}
We deduce from (\ref{tr17}), (\ref{wwe1}) and (\ref{wwe2}) the convergence of the sequences
\begin{equation}\label{qw0.0}
g=\lim_{n \to +\infty}\tilde{g}_n \quad \textrm{ in } L^{\infty}([0,T],H^{(1,0)}(\rr_{x,v}^2)),
\end{equation}
\begin{equation}\label{qw0.1}
G=\lim_{n \to +\infty}\mathcal{H}^{\frac{s}{2}}\tilde{g}_n \quad \textrm{ in } L^{2}([0,T],H^{(1,0)}(\rr_{x,v}^2)).
\end{equation}
Let $H=\un_{\R_+}$ be the Heaviside function.
The convergences
$$H(t)g=\lim_{n \to +\infty}H(t)\tilde{g}_n \quad \textrm{ in } L^{\infty}((-\infty,T],H^{(1,0)}(\rr_{x,v}^2)),$$
$$H(t)G=\lim_{n \to +\infty}H(t)\mathcal{H}^{\frac{s}{2}}\tilde{g}_n \quad \textrm{ in } L^{2}((-\infty,T],H^{(1,0)}(\rr_{x,v}^2)),$$
imply
\begin{equation}\label{qw2}
H(t)g=\lim_{n \to +\infty}H(t)\tilde{g}_n, \qquad H(t)G=\lim_{n \to +\infty}H(t)\mathcal{H}^{\frac{s}{2}}\tilde{g}_n \quad \textrm{ in }  \mathscr{D}'((-\infty,T),\mathscr S'(\rr_{x,v}^2)).
\end{equation}
We obtain that for all $\varphi \in \mathscr{F}=C_0^{\infty}((-\infty,T),\mathscr S(\rr_{x,v}^2))$,
\begin{multline*}
\langle  H(t)G, \varphi \rangle_{\mathscr{F}^*,\mathscr{F}}=\lim_{n \to +\infty}\langle H(t)\mathcal{H}^{\frac{s}{2}}\tilde{g}_n, \varphi \rangle_{\mathscr{F}^*,\mathscr{F}}=\lim_{n \to +\infty}\langle H(t)\tilde{g}_n, \mathcal{H}^{\frac{s}{2}}\varphi \rangle_{\mathscr{F}^*,\mathscr{F}}\\ =\langle H(t)g, \mathcal{H}^{\frac{s}{2}}\varphi \rangle_{\mathscr{F}^*,\mathscr{F}}
=\langle H(t) \mathcal{H}^{\frac{s}{2}}g,\varphi \rangle_{\mathscr{F}^*,\mathscr{F}},
\end{multline*}
where $\mathscr{F}^*$ stands for the anti-dual (anti-linear forms) of $\mathscr{F}$ and where $\langle\cdot, \cdot \rangle_{\mathscr{F}^*,\mathscr{F}}$ denotes the duality bracket.
It follows that
\begin{equation}\label{qqw1}
G=\mathcal{H}^{\frac{s}{2}}g.
\end{equation}
By passing to the limit in the estimate (\ref{tr16.1}), we deduce from (\ref{qw0.0}), (\ref{qw0.1}) and (\ref{qqw1}) that
$$\|g\|_{L^{\infty}([0,T],H^{(1,0)}(\rr_{x,v}^2))}+\|\mathcal{H}^{\frac{s}{2}}g\|_{L^{2}([0,T],H^{(1,0)}(\rr_{x,v}^2))} \leq c_0e^{3T}\|g_0\|_{(1,0)}.$$
We find as well from (\ref{qw2}) that
\begin{equation}\label{qw3}
\partial_t(H(t)g)
=\lim_{n \to +\infty}\partial_t(H(t)\tilde{g}_n),
\qquad
H(t)v\partial_xg=\lim_{n \to +\infty}H(t)v\partial_x\tilde{g}_n,
\end{equation}
with limits in $\mathscr{F}^*$.
It follows from (\ref{sal1}), (\ref{3.llkkb}) and (\ref{qw2}) that
\begin{equation}\label{qw5}
H(t) \mathcal K g
=\lim_{n \to +\infty} H(t) \mathcal{K}\tilde{g}_n,
\end{equation}
with limit in $\mathscr{F}^*$.
On the other hand, we obtain that for all $\varphi \in \mathscr{F}$,
\begin{multline}\label{qw7}
|\langle  H(t) \Gamma(\tilde{g}_n,\tilde{g}_{n+1}),\varphi\rangle_{\mathscr{F}^*,\mathscr{F}}-\langle  H(t) \Gamma(g,g),\varphi \rangle_{ \mathscr{F}^*, \mathscr{F}}| \\ \leq |\langle  H(t) \Gamma(\tilde{g}_n-g,\tilde{g}_{n+1}),\varphi\rangle_{ \mathscr{F}^*, \mathscr{F}}|+|\langle  H(t) \Gamma(g,\tilde{g}_{n+1}-g),\varphi \rangle_{ \mathscr{F}^*, \mathscr{F}}|.
\end{multline}
We deduce from Lemma~\ref{proposition222} with $r=1$ that
\begin{multline*}
|\langle  H(t) \Gamma(\tilde{g}_n-g,\tilde{g}_{n+1}),\varphi\rangle_{\mathscr{F}^*,\mathscr{F}}| \leq \int_{\rr} H(t)|( \Gamma(\tilde{g}_n-g,\tilde{g}_{n+1}),\varphi)_{L^2(\rr_{x,v}^2)}|dt \\
\leq  C_1\int_0^T\|\tilde{g}_n(t)-g(t)\|_{(1,0)}\|\mathcal{H}^{\frac{s}{2}}\tilde{g}_{n+1}(t)\|_{(1,0)}\|\mathcal{H}^{\frac{s}{2}}\varphi(t)\|_{L^2(\rr_{x,v}^2)}dt.
\end{multline*}
It follows from (\ref{tr16.1}) that
\begin{align}\label{qw8}
& \  |\langle H(t) \Gamma(\tilde{g}_n-g,\tilde{g}_{n+1}),\varphi\rangle_{\mathscr{F}^*,\mathscr{F}}| \\ \notag
 \leq & \  C_1\Vert \tilde{g}_n-g\|_{L^{\infty}([0,T],H^{(1,0)}(\rr_{x,v}^2))}  \|\mathcal{H}^{\frac{s}{2}}\tilde{g}_{n+1}\|_{L^{2}([0,T],H^{(1,0)}(\rr_{x,v}^2))} \|\mathcal{H}^{\frac{s}{2}}\varphi\|_{L^{2}([0,T],L^2(\rr_{x,v}^2))} \\ \notag
 \leq & \ \eps_0 C_1\|\mathcal{H}^{\frac{s}{2}}\varphi\|_{L^{2}([0,T],L^2(\rr_{x,v}^2))}
 \|\tilde{g}_n-g\|_{L^{\infty}([0,T],H^{(1,0)}(\rr_{x,v}^2))}.
\end{align}
We deduce from (\ref{qw0.0}) and (\ref{qw8}) that
\begin{equation}\label{qw9}
\lim_{n \to +\infty}\langle H(t) \Gamma(\tilde{g}_n-g,\tilde{g}_{n+1}),\varphi\rangle_{\mathscr{F}^*,\mathscr{F}}=0.
\end{equation}
We deduce from Lemma~\ref{proposition222} with $r=1$ that
\begin{multline*}
|\langle H(t) \Gamma(g,\tilde{g}_{n+1}-g),\varphi \rangle_{\mathscr{F}^*,\mathscr{F}}| \leq \int_{\rr}H(t)
|( \Gamma(g,\tilde{g}_{n+1}-g),\varphi)_{L^2(\rr_{x,v}^2)}|dt \\
\leq  C_1\int_0^T\|g(t)\|_{(1,0)}\|\mathcal{H}^{\frac{s}{2}}\tilde{g}_{n+1}(t)-\mathcal{H}^{\frac{s}{2}}g(t)\|_{(1,0)}\|\mathcal{H}^{\frac{s}{2}}\varphi(t)\|_{L^2(\rr_{x,v}^2)}dt.
\end{multline*}
It follows that
\begin{multline}\label{qw10}
|\langle H(t) \Gamma(g,\tilde{g}_{n+1}-g),\varphi \rangle_{\mathscr{F}^*,\mathscr{F}}|\\
 \leq   C_1\|g\|_{L^{\infty}([0,T],H^{(1,0)}(\rr_{x,v}^2))} \|\mathcal{H}^{\frac{s}{2}}\varphi\|_{L^{2}([0,T],L^2(\rr_{x,v}^2))} \|\mathcal{H}^{\frac{s}{2}}\tilde{g}_{n+1}-\mathcal{H}^{\frac{s}{2}}g\|_{L^{2}([0,T],H^{(1,0)}(\rr_{x,v}^2))}.
\end{multline}
We deduce from (\ref{qw0.1}) and (\ref{qw10}) that
\begin{equation}\label{qw11}
\lim_{n \to +\infty}H(t) \Gamma(g,\tilde{g}_{n+1}-g)=0,
\end{equation}
with limit in $\mathscr F^*$.
It follows from (\ref{qw7}), (\ref{qw9}) and (\ref{qw11}) that
\begin{equation}\label{qw12}
\lim_{n \to +\infty}H(t) \Gamma(\tilde{g}_n,\tilde{g}_{n+1})=H(t) \Gamma(g,g),
\end{equation}
with limit in $\mathscr F^*$.
We deduce from (\ref{tr16}), (\ref{qw3}), (\ref{qw5})
and (\ref{qw12}) that
$$\partial_t(H(t)g)+v\partial_{x}H(t)g+\mathcal{K}H(t)g=H(t)\Gamma(g,g)+\delta_0(t)\otimes g_0,$$
that is,
$$\begin{cases}\partial_tg+v\partial_{x}g+\mathcal{K}g=\Gamma(g,g),\\
g_{\vert t=0}=g_0.
\end{cases}$$
This ends the proof of Theorem~\ref{qw13}.
\end{proof}

\subsection{Uniqueness}

The following result provides the uniqueness for the Cauchy problem (\ref{qw13.1}) in Theorem~\ref{qw13} when the initial fluctuation is sufficiently small
$\|g_0\|_{(1,0)} \ll 1$.

\begin{theorem}\label{qw13.66}
Let $T>0$. Then, there exists a positive constant $\tilde{\eps}_0>0$ such that if
$$g_1, g_2 \in L^{\infty}([0,T],H^{(1,0)}(\rr_{x,v}^2)) ,$$
are two solutions of the Cauchy problem associated to the spatially inhomogeneous Kac equation
\begin{equation}\label{qw13.16}
\begin{cases}
\partial_tg+v\partial_{x}g+\mathcal{K}g=\Gamma(g,g),\\
g|_{t=0}=g_0,
\end{cases}
\end{equation}
with the same initial fluctuation
$$g_0 \in H^{(1,0)}(\rr_{x,v}^2), \quad \|g_0\|_{(1,0)} \leq \tilde{\eps}_0,$$
satisfying
$$\|g_j\|_{L^{\infty}([0,T],H^{(1,0)}(\rr_{x,v}^2))} \leq \tilde{\eps}_0, \quad   \mathcal{H}^{\frac{s}{2}}g_j \in L^{2}([0,T],H^{(1,0)}(\rr_{x,v}^2)), \quad j=1,2.$$
Then, the two solutions are identical
$$\forall 0 \leq t \leq T, \quad g_1(t)=g_2(t).$$
\end{theorem}

\begin{proof}
Let $g_1, g_2 \in L^{\infty}([0,T],H^{(1,0)}(\rr_{x,v}^2)) $ be two solutions of the Cauchy problem
$$\begin{cases}
\partial_tg+v\partial_{x}g+\mathcal{K}g=\Gamma(g,g),\\
g|_{t=0}=g_0,
\end{cases}$$
associated to the same initial datum
$$g_0\in H^{(1,0)}(\rr_{x,v}^2), \quad \|g_0\|_{(1,0)} \leq \eps_0,$$
satisfying
\begin{equation}\label{qw13.2e}
\|g_j\|_{L^{\infty}([0,T],H^{(1,0)}(\rr_{x,v}^2))} \leq \frac{3}{8CC_1}, \quad   \mathcal{H}^{\frac{s}{2}}g_j \in L^{2}([0,T],H^{(1,0)}(\rr_{x,v}^2)), \quad j=1,2,
\end{equation}
where $C, C_1,\eps_0>0$ are the constants defined in (\ref{tr4}), Lemma~\ref{proposition222} and Theorem~\ref{qw13}.
We consider the difference $f=g_1-g_2$. This function satisfies the Cauchy problem
\begin{equation}\label{po1}
\begin{cases}
\partial_tf+v\partial_{x}f+\mathcal{K}f=\Gamma(g_1,f)+\Gamma(f,g_2),\\
f|_{t=0}=0.
\end{cases}
\end{equation}
We define
\begin{equation}\label{li2.2}
f_{\delta}=(1+\delta \sqrt{\mathcal{H}}+\delta \langle D_x \rangle)^{-1}f, \quad 0<\delta \leq 1.
\end{equation}
By proceeding as in (\ref{tr18}) and (\ref{li2.1}), we use the multiplier $(1+\delta \sqrt{\mathcal{H}}+\delta \langle D_x \rangle)^{-2}f$ in $H^{(1,0)}(\rr_{x,v}^2)$. The very same arguments allow to prove as in (\ref{li12.1}) that for all $0<\delta \leq 1$,
\begin{multline*}
\frac{1}{2}\frac{d}{dt}\|f_{\delta}(t)\|_{(1,0)}^2+\frac{1}{C}\|\mathcal{H}^{\frac{s}{2}}f_{\delta}(t)\|_{(1,0)}^2  \leq  3\|f_{\delta}(t)\|_{(1,0)}^2
\\ +2C_1\|g_{1}(t)\|_{(1,0)}\|\mathcal{H}^\frac{s}{2}f_{\delta}(t)\|_{(1,0)}^2
+ 2C_1\|f_{\delta}(t)\|_{(1,0)}\|\mathcal{H}^\frac{s}{2}g_{2}(t)\|_{(1,0)}\|\mathcal{H}^\frac{s}{2}f_{\delta}(t)\|_{(1,0)}.
\end{multline*}
It follows from (\ref{qw13.2e}) that for all $0<\delta \leq 1$,
\begin{align*}
& \ \frac{1}{2}\frac{d}{dt}\|f_{\delta}(t)\|^2_{(1, 0)}+\frac{3}{4C}\|{\mathcal{H}}^{\frac{s}{2}}f_{\delta}(t)\|^2_{(1, 0)}\\
\leq & \ 3\|f_{\delta}(t)\|^2_{(1, 0)}+2C_1 \|g_1\|_{L^{\infty}([0,T],H^{(1,0)}(\rr_{x,v}^2))}\|{\mathcal{H}}^{\frac{s}{2}}f_{\delta}(t)\|^2_{(1, 0)}
+4CC_1^2\|f_{\delta}(t)\|_{(1, 0)}^2\|{\mathcal{H}}^{\frac{s}{2}}g_2(t)\|_{(1, 0)}^2\\
\leq & \ 3\|f_{\delta}(t)\|^2_{(1, 0)}+\frac{3}{4C}\|{\mathcal{H}}^{\frac{s}{2}}f_{\delta}(t)\|^2_{(1, 0)}
+4CC_1^2\|f_{\delta}(t)\|_{(1, 0)}^2\|{\mathcal{H}}^{\frac{s}{2}}g_2(t)\|_{(1, 0)}^2.
\end{align*}
We deduce that for all $0<\delta \leq 1$,
$$\frac{d}{dt}\|f_{\delta}(t)\|^2_{(1, 0)}\leq 8\|f_{\delta}(t)\|^2_{(1, 0)}(1+CC_1^2\|{\mathcal{H}}^{\frac{s}{2}}g_2(t)\|_{(1, 0)}^2),$$
that is,
$$\frac{d}{dt}\Big(\|f_{\delta}(t)\|^2_{(1, 0)}\exp\Big(-8t -8CC_1^2\int_0^t\|{\mathcal{H}}^{\frac{s}{2}}g_2(\tau)\|_{(1, 0)}^2d\tau\Big)\Big) \leq 0.$$
It follows from (\ref{qw13.2e}), (\ref{po1}) and (\ref{li2.2}) that for all $0<\delta \leq 1$,
\begin{align*}
\|f_{\delta}(t)\|^2_{(1, 0)}\leq & \ \|(1+\delta \sqrt{\mathcal{H}}+\delta \langle D_x \rangle)^{-1}f(0)\|^2_{(1, 0)}\exp\Big(8t+8CC_1^2\int_0^t\|{\mathcal{H}}^{\frac{s}{2}}g_2(\tau)\|_{(1, 0)}^2d\tau\Big) \\
\leq & \  \|f(0)\|^2_{(1, 0)}\exp\big(8T+8CC_1^2\|{\mathcal{H}}^{\frac{s}{2}}g_2\|_{L^{2}([0,T],H^{(1,0)}(\rr_{x,v}^2))}^2\big)=0,
\end{align*}
since $f(0)=0$. This proves that $f_{\delta}(t)=0$ for all $0\leq t \leq T$. According to (\ref{li2.2}), this ends the proof of Theorem~\ref{qw13.66}.
\end{proof}

\section{Gelfand-Shilov and Gevrey regularizing effect}\label{gelf}
We aim at establishing that the Cauchy problem (\ref{qw13.16}) enjoys some
Gelfand-Shilov regularizing properties with respect to the velocity variable and Gevrey regularizing properties with respect to the position variable.

\subsection{A priori estimates with exponential weights}

We begin by establishing some a priori estimates with exponential weights satisfied by the sequence of approximate solutions $(\tilde{g}_n)_{n \geq 0}$ defined in \eqref{tr16} for sufficiently small initial data:

\begin{lemma}\label{sake2}
Let $T>0$. Then, there exist some positive constants $c, \eps_1>0$, $0<\delta_0 \leq 1$ such that for all initial data $\|g_0\|_{(1,0)} \leq \eps_1$, the sequence of approximate solutions $(\tilde{g}_n)_{n \geq 0}$ defined in \eqref{tr16} satisfies for all $0<\delta_1 \leq 1$, $0 \leq \delta \leq \delta_0$, $n \geq 1$,
\begin{multline}\label{sake6}
\|M_{\delta_1}(\delta t)\tilde{g}_n\|_{L^{\infty}([0,T],H^{(1,0)}(\rr_{x,v}^2))}^2+\|\mathcal{H}^{\frac{s}{2}}M_{\delta_1}(\delta t)\tilde{g}_n\|_{L^{2}([0,T],H^{(1,0)}(\rr_{x,v}^2))}^2\\
+\|\langle D_x \rangle^{\frac{s}{2s+1}}M_{\delta_1}(\delta t)\tilde{g}_n\|_{L^{2}([0,T],H^{(1,0)}(\rr_{x,v}^2))}^2 \leq ce^{cT}\|g_0\|_{(1,0)}^2,
\end{multline}
with
$$M_{\delta_1}(t)=\frac{\exp\big(t(\sqrt{\mathcal{H}}+\langle D_x \rangle)^{\frac{2s}{2s+1}}\big)}{1+\delta_1 \exp\big(t(\sqrt{\mathcal{H}}+\langle D_x \rangle)^{\frac{2s}{2s+1}}\big)}.$$
\end{lemma}

Let $0 \leq \delta \leq 1$ and $0<\delta_1 \leq 1$. We define \begin{equation}\label{as3}
h_{n,\delta,\delta_1}=M_{\delta_1}(\delta t)\tilde{g}_n, \quad n \geq 0.
\end{equation}
The functions $h_{n,\delta,\delta_1}$ depend on the parameters $0 \leq \delta \leq 1$ and $0<\delta_1 \leq 1$. For simplicity, we omit this dependence in the notation and write $h_n$ for $h_{n,\delta,\delta_1}$.
We notice from (\ref{tr14}) that
$$h_0(t)=\big(1+\delta_1 \exp\big(\delta t(\sqrt{\mathcal{H}}+\langle D_x \rangle)^{\frac{2s}{2s+1}}\big)\big)^{-1}g_0, \quad 0\leq t \leq T,$$
satisfies
\begin{equation}\label{as1}
\|h_0\|_{L^{\infty}([0,T],H^{(1,0)}(\rr_{x,v}^2))} \leq  \|g_0\|_{(1,0)}.
\end{equation}
By using that
$$\tilde{g}_n=\big(M_{\delta_1}(\delta t)\big)^{-1}h_n=\big(\delta_1+ \exp\big(-\delta t(\sqrt{\mathcal{H}}+\langle D_x \rangle)^{\frac{2s}{2s+1}}\big)\big)h_n,$$
the equation
$$\partial_t\tilde{g}_{n+1}+v\partial_{x}\tilde{g}_{n+1}+\mathcal{K}\tilde{g}_{n+1}=\Gamma(\tilde{g}_{n},\tilde{g}_{n+1}),$$
reads as
\begin{multline*}
\big(M_{\delta_1}(\delta t)\big)^{-1}\partial_th_{n+1}-\delta (\sqrt{\mathcal{H}}+\langle D_x \rangle)^{\frac{2s}{2s+1}}\exp\big(-\delta t(\sqrt{\mathcal{H}}+\langle D_x \rangle)^{\frac{2s}{2s+1}}\big)h_{n+1}+v\partial_{x}
\big(M_{\delta_1}(\delta t)\big)^{-1}h_{n+1}\\
+\big(M_{\delta_1}(\delta t)\big)^{-1}\mathcal{K}h_{n+1}=\Gamma\big(\big(M_{\delta_1}(\delta t)\big)^{-1}h_n,\big(M_{\delta_1}(\delta t)\big)^{-1}h_{n+1}\big),
\end{multline*}
since according to (\ref{sal1}), the linearized Kac operator $\mathcal{K}=f(\mathcal{H})$ is a function of the harmonic oscillator acting only in the velocity variable, which therefore commutes with the exponential weight $\big(M_{\delta_1}(\delta t)\big)^{-1}$. It follows that
\begin{multline}\label{cl2}
\partial_th_{n+1}-\frac{\delta (\sqrt{\mathcal{H}}+\langle D_x \rangle)^{\frac{2s}{2s+1}}}{1+\delta_1 \exp\big(\delta t(\sqrt{\mathcal{H}}+\langle D_x \rangle)^{\frac{2s}{2s+1}}\big)}h_{n+1}+
M_{\delta_1}(\delta t)v\big(M_{\delta_1}(\delta t)\big)^{-1}\partial_{x}h_{n+1}\\  +\mathcal{K}h_{n+1}=
M_{\delta_1}(\delta t)\Gamma\big( \big(M_{\delta_1}(\delta t)\big)^{-1}h_n,\big(M_{\delta_1}(\delta t)\big)^{-1}h_{n+1}\big).
\end{multline}
By integrating with respect to the $\xi$-variable and coming back to the direct side, we deduce from Lemma~\ref{KFP2} and (\ref{m1fp}) that we can choose the positive parameter $0<\eps \leq \eps_0$ in order to ensure that the multiplier
\begin{equation}\label{qsym}
Q=Q(v,D_v,D_x)=1-\eps g^w(v,D_v,D_x),
\end{equation}
is a positive bounded isomorphism on $L^2(\rr_{x,v}^2)$ such that for all $u \in \mathscr{S}(\rr_{x,v}^2)$,
\begin{multline}\label{cp3}
\textrm{Re}((v\partial_x+a_0^w(v,D_v))u,Qu)_{L^2(\rr_{x,v}^2)} \geq  \\ c_{1}\|\mathcal{H}^{\frac{s}{2}}u\|_{L^2(\rr_{x,v}^2)}^2  +c_1\eps \|\langle D_x \rangle^{\frac{s}{2s+1}}u\|_{L^2(\rr_{x,v}^2)}^2 -c_{2}\|u\|_{L^2(\rr_{x,v}^2)}^2.
\end{multline}
Furthermore, we notice from Lemma~\ref{2.2fp}, (\ref{hm1}) and (\ref{hm1.2}) that the symbol $Q(\cdot,\xi)$ belongs to the symbol class $S(1,\Gamma_0)$ uniformly with respect to the parameter $\xi \in \rr$. The operator $Q$ is therefore commuting with any operator of the type $f(D_x)$.
It follows from (\ref{cp4}), (\ref{cp5}) and (\ref{cp3}) that there exist some positive constants $c_3,c_4>0$ such that for all $u \in \mathscr{S}(\rr_{x,v}^2)$,
\begin{equation}\label{cp3.1}
\textrm{Re}((v\partial_x+\mathcal{K})u,Qu)_{L^2(\rr_{x,v}^2)} \geq  c_3\|\mathcal{H}^{\frac{s}{2}}u\|_{L^2(\rr_{x,v}^2)}^2+c_3 \|\langle D_x \rangle^{\frac{s}{2s+1}}u\|_{L^2(\rr_{x,v}^2)}^2 -c_4\|u\|_{L^2(\rr_{x,v}^2)}^2.
\end{equation}
Applying the estimate (\ref{cp3.1}) to the function $\langle D_x \rangle u$, we obtain that  for all $u \in \mathscr{S}(\rr_{x,v}^2)$,
\begin{equation}\label{cp6}
\textrm{Re}((v\partial_x+\mathcal{K})u,\langle D_x \rangle^{2}Qu)_{L^2(\rr_{x,v}^2)} \geq  c_3\|\mathcal{H}^{\frac{s}{2}}u\|_{(1,0)}^2+c_3 \|\langle D_x \rangle^{\frac{s}{2s+1}}u\|_{(1,0)}^2 -c_4\|u\|_{(1,0)}^2,
\end{equation}
since the operator $\mathcal{K}$ only acts in the velocity variable.
We define
\begin{equation}\label{li2.3}
h_{n,\delta_2}=(1+\delta_2 \sqrt{\mathcal{H}}+\delta_2 \langle D_x \rangle)^{-1}h_n, \quad 0<\delta_2 \leq 1.
\end{equation}
By using the multiplier
$$\langle D_x \rangle^{2}(1+\delta_2 \sqrt{\mathcal{H}}+\delta_2 \langle D_x \rangle)^{-1}Q(1+\delta_2 \sqrt{\mathcal{H}}+\delta_2 \langle D_x \rangle)^{-1}h_{n+1},$$ in $L^2(\rr_{x,v}^2),$
we deduce from (\ref{cl2}) that
\begin{align*}
& \ \frac{1}{2}\frac{d}{dt}\|Q^{1/2}h_{n+1,\delta_2}\|_{(1,0)}^2+  \textrm{Re}(\mathcal{K}h_{n+1,\delta_2},\langle D_x \rangle^{2}Qh_{n+1,\delta_2})_{L^2(\rr_{x,v}^2)} \\ 
+ & \   \textrm{Re}\big((1+\delta_2 \sqrt{\mathcal{H}}+\delta_2 \langle D_x \rangle)^{-1}M_{\delta_1}(\delta t)v\big(M_{\delta_1}(\delta t)\big)^{-1}\partial_{x}h_{n+1},\langle D_x \rangle^{2}Qh_{n+1,\delta_2}\big)_{L^2(\rr_{x,v}^2)}\\ 
- & \   \textrm{Re}\Big(\frac{ \delta(\sqrt{\mathcal{H}}+\langle D_x \rangle)^{\frac{2s}{2s+1}}}{1+\delta_1 \exp\big(\delta t(\sqrt{\mathcal{H}}+\langle D_x \rangle)^{\frac{2s}{2s+1}}\big)}h_{n+1,\delta_2},\langle D_x \rangle^{2}Qh_{n+1,\delta_2}\Big)_{L^2(\rr_{x,v}^2)}\\ 
=& \ \textrm{Re}\big((1+\delta_2 \sqrt{\mathcal{H}}+\delta_2 \langle D_x \rangle)^{-1}M_{\delta_1}(\delta t)\Gamma\big(\big(M_{\delta_1}(\delta t)\big)^{-1}h_n,\big(M_{\delta_1}(\delta t)\big)^{-1}h_{n+1}\big),Qh_{n+1,\delta_2}\big)_{(1,0)},
\end{align*}
since $[\mathcal{K},(1+\delta_2 \sqrt{\mathcal{H}}+\delta_2 \langle D_x \rangle)^{-1}]=0$.
We obtain that
\begin{align}\label{cl10.3}
& \ \frac{1}{2}\frac{d}{dt}\|Q^{1/2}h_{n+1,\delta_2}\|_{(1,0)}^2+  \textrm{Re}\big((v\partial_x+\mathcal{K})h_{n+1,\delta_2},\langle D_x \rangle^{2}Qh_{n+1,\delta_2}\big)_{L^2(\rr_{x,v}^2)} \\ \notag
+ & \   \textrm{Re}\big(\big[A_{\delta_1,\delta_2}(\delta t)v\big(A_{\delta_1,\delta_2}(\delta t)\big)^{-1}-v\big]\partial_{x}h_{n+1,\delta_2},\langle D_x \rangle^{2}Qh_{n+1,\delta_2}\big)_{L^2(\rr_{x,v}^2)}\\ \notag
\leq & \ \delta\|(\sqrt{\mathcal{H}}+\langle D_x \rangle)^{\frac{s}{2s+1}}h_{n+1,\delta_2}\|_{(1,0)}\|(\sqrt{\mathcal{H}}+\langle D_x \rangle)^{\frac{s}{2s+1}}Qh_{n+1,\delta_2}\|_{(1,0)}\\ \notag
+ & \ \big|\big((1+\delta_2 \sqrt{\mathcal{H}}+\delta_2 \langle D_x \rangle)^{-1}M_{\delta_1}(\delta t)\Gamma\big(\big(M_{\delta_1}(\delta t)\big)^{-1}h_n,\big(M_{\delta_1}(\delta t)\big)^{-1}h_{n+1}\big),Qh_{n+1,\delta_2}\big)_{(1,0)}\big|,
\end{align}
with
\begin{equation}\label{ro1}
A_{\delta_1,\delta_2}(\delta t)=\frac{M_{\delta_1}(\delta t)}{1+\delta_2 \sqrt{\mathcal{H}}+\delta_2 \langle D_x \rangle }=\frac{e^{\delta t(\sqrt{\mathcal{H}}+\langle D_x \rangle)^{\frac{2s}{2s+1}}}}{(1+\delta_2 \sqrt{\mathcal{H}}+\delta_2 \langle D_x \rangle)(1+\delta_1 e^{\delta t(\sqrt{\mathcal{H}}+\langle D_x \rangle)^{\frac{2s}{2s+1}}})}.
\end{equation}
It follows from Lemma~\ref{proposition222} with $r=1$, (\ref{cp6}) and (\ref{cl10.3}) that for all $0 \leq \delta \leq 1$, $0<\delta_1 \leq 1$, $0<\delta_2 \leq 1$, $0 \leq t \leq T$,
\begin{align}\label{cl11b}
& \ \frac{1}{2}\frac{d}{dt}\|Q^{1/2}h_{n+1,\delta_2}\|_{(1,0)}^2+c_3\|\mathcal{H}^{\frac{s}{2}}h_{n+1,\delta_2}\|_{(1,0)}^2+c_3\|\langle D_x \rangle^{\frac{s}{2s+1}}h_{n+1,\delta_2}\|_{(1,0)}^2 \\ \notag
 \leq & \ C_1\|h_n\|_{(1,0)}\|\mathcal{H}^{\frac{s}{2}}h_{n+1}\|_{(1,0)}\|\mathcal{H}^{\frac{s}{2}}Qh_{n+1,\delta_2}\|_{(1,0)}+c_4\|h_{n+1,\delta_2}\|_{(1,0)}^2\\ \notag
 + & \ \delta\|(\sqrt{\mathcal{H}}+\langle D_x \rangle)^{\frac{s}{2s+1}}h_{n+1,\delta_2}\|_{(1,0)}\|(\sqrt{\mathcal{H}}+\langle D_x \rangle)^{\frac{s}{2s+1}}Qh_{n+1,\delta_2}\|_{(1,0)}\\ \notag
 + & \ \big|\big(\big[A_{\delta_1,\delta_2}(\delta t)v\big(A_{\delta_1,\delta_2}(\delta t)\big)^{-1}-v\big]\partial_{x}h_{n+1,\delta_2},\langle D_x \rangle^{2}Qh_{n+1,\delta_2}\big)_{L^2(\rr_{x,v}^2)}\big|.
\end{align}
We use the following instrumental lemmas:

\begin{lemma}\label{tt1}
There exists a positive constant $c_5>0$ such that for all $f \in \mathscr{S}(\rr_{x,v}^2)$,
$$\|(\sqrt{\mathcal{H}}+\langle D_x \rangle)^{\frac{s}{2s+1}}f\|_{(1,0)} \leq c_5\|\mathcal{H}^{\frac{s}{2}}f\|_{(1,0)}+
c_5\|\langle D_x \rangle^{\frac{s}{2s+1}}f\|_{(1,0)}.$$
\end{lemma}

\begin{proof}
Let $f \in \mathscr{S}(\rr_{x,v}^2)$. We decompose this function into the Hermite basis in the velocity variable
$$f(x,v)=\sum_{n=0}^{+\infty}f_n(x) \psi_n(v),$$
with
$$f_n(x)=(f(x,\cdot),\psi_n)_{L^2(\rr_v)}.$$
We deduce from (\ref{cl3}) that
\begin{align*}
& \ \|(\sqrt{\mathcal{H}}+\langle D_x \rangle)^{\frac{s}{2s+1}}f\|_{(1,0)}=\Big(\frac{1}{2\pi}\sum_{n=0}^{+\infty}\int_{\rr}\langle \xi \rangle^{2}\Big(\sqrt{n+\frac{1}{2}}+\langle \xi \rangle\Big)^{\frac{2s}{2s+1}}|\hat{f}_n(\xi)|^2d\xi\Big)^{\frac{1}{2}}\\
\leq & \ \Big(\sum_{n=0}^{+\infty}\int_{\rr}\langle \xi \rangle^{2}\Big[\Big(n+\frac{1}{2}\Big)^{\frac{s}{2s+1}}+\langle \xi \rangle^{\frac{2s}{2s+1}}\Big]|\hat{f}_n(\xi)|^2\frac{d\xi}{2\pi}\Big)^{\frac{1}{2}}=(\|\mathcal{H}^{\frac{s}{4s+2}}f\|_{(1,0)}^2+\|\langle D_x \rangle^{\frac{s}{2s+1}}f\|_{(1,0)}^2)^{\frac{1}{2}}\\
\lesssim & \ \|\mathcal{H}^{\frac{s}{4s+2}}f\|_{(1,0)}+\|\langle D_x \rangle^{\frac{s}{2s+1}}f\|_{(1,0)} \lesssim \|\mathcal{H}^{\frac{s}{2}}f\|_{(r,0)}+\|\langle D_x \rangle^{\frac{s}{2s+1}}f\|_{(1,0)}.
\end{align*}
This ends the proof of Lemma~\ref{tt1}.
\end{proof}

\begin{lemma}\label{tt2}
For all $0 \leq m\leq 1$, there exists a positive constant $\tilde{c}_m>0$ such that for all $f \in \mathscr{S}(\rr_{x,v}^2)$,
$$\|\mathcal{H}^{m}Qf\|_{(1,0)} \leq \tilde{c}_m\|\mathcal{H}^{m}f\|_{(1,0)}.$$
\end{lemma}

\begin{proof}
We notice from (\ref{hm4}) that
\begin{multline}\label{ip1}
\|\mathcal{H}^mQf\|_{(1,0)} \lesssim \Big\|\textrm{Op}^w\Big(\Big(1+\eta^2+\frac{v^2}{4}\Big)^{m}\Big)Q \langle D_x \rangle f\Big\|_{L^2(\rr_{x,v}^2)} \\ \leq
\Big\|Q \textrm{Op}^w\Big(\Big(1+\eta^2+\frac{v^2}{4}\Big)^{m}\Big)\langle D_x \rangle f\Big\|_{L^2(\rr_{x,v}^2)}+
\Big\|\Big[\textrm{Op}^w\Big(\Big(1+\eta^2+\frac{v^2}{4}\Big)^{m}\Big),Q\Big] \langle D_x \rangle f\Big\|_{L^2(\rr_{x,v}^2)},
\end{multline}
since $[Q,\langle D_x \rangle]=0$.
By using again (\ref{hm4}), it follows from the fact that the multiplier $Q$ is a bounded operator on $L^2(\rr_{x,v}^2)$ that
\begin{multline}\label{ip2}
\Big\|Q \textrm{Op}^w\Big(\Big(1+\eta^2+\frac{v^2}{4}\Big)^{m}\Big)\langle D_x \rangle f\Big\|_{L^2(\rr_{x,v}^2)}  \lesssim
\Big\|\textrm{Op}^w\Big(\Big(1+\eta^2+\frac{v^2}{4}\Big)^{m}\Big)\langle D_x \rangle f\Big\|_{L^2(\rr_{x,v}^2)}\\ \lesssim \|\mathcal{H}^{m}\langle D_x \rangle f\|_{L^2(\rr_{x,v}^2)}= \|\mathcal{H}^{m}f\|_{(1,0)}.
\end{multline}
On the other hand, we notice from (\ref{hm1}), (\ref{hm1.2}), (\ref{qsym}) and Lemma~\ref{2.2fp} that
$$\Big[\textrm{Op}^w\Big(\Big(1+\eta^2+\frac{v^2}{4}\Big)^{m}\Big),Q\Big] \in \textrm{Op}^w\big(S(\langle (v,\eta) \rangle^{2m-2},\Gamma_0)\big) \subset \textrm{Op}^w\big(S(1,\Gamma_0)\big),$$
uniformly with respect to the parameter $\xi \in \rr$, because $0 \leq m \leq 1$. This implies that
\begin{equation}\label{ip3}
\Big\|\Big[\textrm{Op}^w\Big(\Big(1+\eta^2+\frac{v^2}{4}\Big)^{m}\Big),Q\Big] \langle D_x \rangle f\Big\|_{L^2(\rr_{x,v}^2)} \lesssim \|\langle D_x \rangle f\|_{L^2(\rr_{x,v}^2)}=\|f\|_{(1,0)}.
\end{equation}
We deduce from (\ref{ip1}), (\ref{ip2}) and (\ref{ip3}) that
$$\|\mathcal{H}^mQf\|_{(1,0)} \lesssim \|\mathcal{H}^mf\|_{(1,0)}.$$
This ends the proof of Lemma~\ref{tt2}.
\end{proof}

We deduce from Lemmas~\ref{tt1} and~\ref{tt2} that
\begin{align*}
& \ \|(\sqrt{\mathcal{H}}+\langle D_x \rangle)^{\frac{s}{2s+1}}f\|_{(1,0)}\|(\sqrt{\mathcal{H}}+\langle D_x \rangle)^{\frac{s}{2s+1}}Qf\|_{(1,0)} \\ 
\lesssim & \ (\|\mathcal{H}^{\frac{s}{2}}f\|_{(1,0)}+\|\langle D_x \rangle^{\frac{s}{2s+1}}f\|_{(1,0)})(\|\mathcal{H}^{\frac{s}{2}}Qf\|_{(1,0)}+\|\langle D_x \rangle^{\frac{s}{2s+1}}Qf\|_{(1,0)}) \\ 
\lesssim & \ \|\mathcal{H}^{\frac{s}{2}}f\|_{(1,0)}^2+\|\langle D_x \rangle^{\frac{s}{2s+1}}f\|_{(1,0)}^2,
\end{align*}
since the operator $Q$ is bounded on $L^2(\rr_{x,v}^2)$ and commutes with the operator $\langle D_x \rangle^{\frac{s}{2s+1}+1}$.
It follows from (\ref{cl11b}) and Lemma~\ref{tt2} that there exist some positive constants $0<\delta_0 \leq 1$, $c_6, c_7>0$ such that for all $0\leq \delta \leq \delta_0$, $0<\delta_1\leq 1$, $0<\delta_2\leq 1$, $0 \leq t \leq T$,
\begin{align}\label{tt4}
& \ \frac{1}{2}\frac{d}{dt}\|Q^{1/2}h_{n+1,\delta_2}\|_{(1,0)}^2+c_6\|\mathcal{H}^{\frac{s}{2}}h_{n+1,\delta_2}\|_{(1,0)}^2+c_6\|\langle D_x \rangle^{\frac{s}{2s+1}}h_{n+1,\delta_2}\|_{(1,0)}^2  \\ \notag
\leq   & \ c_7\|h_n\|_{(1,0)}\|\mathcal{H}^{\frac{s}{2}}h_{n+1}\|_{(1,0)}\|\mathcal{H}^{\frac{s}{2}}h_{n+1,\delta_2}\|_{(1,0)}  +c_7\|h_{n+1,\delta_2}\|_{(1,0)}^2\\ \notag
& \quad +   \big|\big(\big[A_{\delta_1,\delta_2}(\delta t)v\big(A_{\delta_1,\delta_2}(\delta t)\big)^{-1}-v\big]\partial_{x}h_{n+1,\delta_2},\langle D_x \rangle^{2}Qh_{n+1,\delta_2}\big)_{L^2(\rr_{x,v}^2)}\big|.
\end{align}
We use the two following instrumental lemmas:

\begin{lemma}\label{tt5}
There exists a positive constant $\tilde{c}_{1}>0$ such that for all $f \in \mathscr{S}(\rr_{x,v}^2)$, $0 \leq \delta \leq 1$, $0 < \delta_1 \leq 1$, $t \geq 0$,
$$\big\|\langle D_x \rangle^{\frac{s+1}{2s+1}}\big(M_{\delta_1}(\delta t)v\big(M_{\delta_1}(\delta t)\big)^{-1}-v\big)f\big\|_{(1,0)} \leq \tilde{c}_{1}\delta  t e^{\tilde{c}_{1} t}\|\langle D_x\rangle^{\frac{s}{2s+1}}f\|_{(1,0)}.$$

\end{lemma}

\begin{proof}
Let $f \in \mathscr{S}(\rr_{x,v}^2)$. We decompose this function into the Hermite basis in the velocity variable
$$f(x,v)=\sum_{n=0}^{+\infty}f_n(x) \psi_n(v), \quad f_n(x)=(f(x,\cdot),\psi_n)_{L^2(\rr_v)}.$$
By using the identities (\ref{ge1}) satisfied by the creation and annihilation operators
$$A_+\psi_n=\Big(\frac{v}{2}-\partial_v\Big)\psi_n=\sqrt{n+1}\psi_{n+1}, \quad A_-\psi_n=\Big(\frac{v}{2}+\partial_v\Big)\psi_n=\sqrt{n}\psi_{n-1}, \quad v=A_++A_-,$$
we notice that
\begin{align}\label{lk-1}
 & \ \left(\frac{\exp\big(\delta t(\sqrt{\mathcal{H}}+\langle \xi \rangle)^{\frac{2s}{2s+1}}\big)}{1+\delta_1 \exp\big(\delta t(\sqrt{\mathcal{H}}+\langle \xi \rangle)^{\frac{2s}{2s+1}}\big)}v\frac{1+\delta_1\exp\big(\delta t(\sqrt{\mathcal{H}}+\langle \xi \rangle)^{\frac{2s}{2s+1}}\big)}{\exp\big(\delta t(\sqrt{\mathcal{H}}+\langle \xi \rangle)^{\frac{2s}{2s+1}}\big)}-v\right)\mathcal{F}_xf \\ \notag
 \notag
=  & \ \sum_{n=0}^{+\infty}\hat{f}_n(\xi)\sqrt{n+1}
\frac{\exp\Big(\delta t\Big(\sqrt{n+\frac{3}{2}}+\langle \xi \rangle\Big)^{\frac{2s}{2s+1}}-\delta t\Big(\sqrt{n+\frac{1}{2}}+\langle \xi \rangle\Big)^{\frac{2s}{2s+1}}\Big)-1}{\Big(1+\delta_1\exp\Big(\delta t(\sqrt{n+\frac{3}{2}}+\langle \xi \rangle)^{\frac{2s}{2s+1}}\Big)\Big)}\psi_{n+1}\\  \notag
+ & \  \sum_{n=0}^{+\infty}\hat{f}_n(\xi)\sqrt{n}
\frac{\exp\Big(\delta t\Big(\sqrt{n-\frac{1}{2}}+\langle \xi \rangle\Big)^{\frac{2s}{2s+1}}-\delta t\Big(\sqrt{n+\frac{1}{2}}+\langle \xi \rangle\Big)^{\frac{2s}{2s+1}}\Big)-1}{\Big(1+\delta_1\exp\Big(\delta t\Big(\sqrt{n-\frac{1}{2}}+\langle \xi \rangle\Big)^{\frac{2s}{2s+1}}\Big)\Big)}\psi_{n-1},
\end{align}
where $\mathcal{F}_x$ stands for the partial Fourier transform with respect to the position variable.
We notice that for all $x>0$,
\begin{multline}\label{ff1}
0 \leq (\sqrt{x+1}+\langle \xi \rangle)^{\frac{2s}{2s+1}}-(\sqrt{x}+\langle \xi \rangle)^{\frac{2s}{2s+1}}=\frac{s}{2s+1}\int_{x}^{x+1}\frac{dy}{\sqrt{y}(\sqrt{y}+\langle \xi \rangle)^{\frac{1}{2s+1}}} \\ \leq \frac{s}{(2s+1)\sqrt{x}(\sqrt{x}+\langle \xi \rangle)^{\frac{1}{2s+1}}}.
\end{multline}
It follows from (\ref{ff1}) for all $n \geq 0$, $t \geq 0$, $0 \leq \delta \leq 1$,
\begin{multline*}
0 \leq  \exp\Big(\delta t\Big(\sqrt{n+\frac{3}{2}}+\langle \xi \rangle\Big)^{\frac{2s}{2s+1}}-\delta t\Big(\sqrt{n+
\frac{1}{2}}+\langle \xi \rangle\Big)^{\frac{2s}{2s+1}}\Big)-1=  \int_0^{A_1}e^ydy \leq  \int_0^{A_2}e^ydy,
\end{multline*}
with
$$A_1=\delta t\Big[\Big(\sqrt{n+\frac{3}{2}}+\langle \xi \rangle\Big)^{\frac{2s}{2s+1}}-\Big(\sqrt{n+
\frac{1}{2}}+\langle \xi \rangle\Big)^{\frac{2s}{2s+1}}\Big],$$
$$A_2=\frac{s \delta t}{(2s+1)\sqrt{n+\frac{1}{2}}\Big(\sqrt{n+\frac{1}{2}}+\langle \xi \rangle\Big)^{\frac{1}{2s+1}}}.$$
This implies that there exists a positive constant $c_8>0$ such that for all $n \geq 0$, $t \geq 0$, $0 \leq \delta \leq 1$,
\begin{multline}\label{qs2}
0 \leq  \exp\Big( \delta t\Big(\sqrt{n+\frac{3}{2}}+\langle \xi \rangle\Big)^{\frac{2s}{2s+1}}-\delta t\Big(\sqrt{n+
\frac{1}{2}}+\langle \xi \rangle\Big)^{\frac{2s}{2s+1}}\Big)-1\\
\leq  \frac{s \delta t}{(2s+1)\sqrt{n+\frac{1}{2}}\langle \xi \rangle^{\frac{1}{2s+1}}}\exp\Big({\frac{2^{\frac{1}{2}}s t}{2s+1}}\Big)
\leq c_8 \delta t e^{c_8t}\langle n \rangle^{-\frac{1}{2}}\langle \xi \rangle^{-\frac{1}{2s+1}}.
\end{multline}
On the other hand, we have for all $n \geq 1$, $t \geq 0$, $0 \leq \delta \leq 1$,
$$0 \leq  1-\exp\Big(-\delta t\Big[\Big(\sqrt{n+\frac{1}{2}}+\langle \xi \rangle\Big)^{\frac{2s}{2s+1}}-\Big(\sqrt{n-
\frac{1}{2}}+\langle \xi \rangle\Big)^{\frac{2s}{2s+1}}\Big]\Big)=  \int_{-A_3}^{0}e^ydy,$$
with
$$A_3= \delta t\Big[\Big(\sqrt{n+\frac{1}{2}}+\langle \xi \rangle\Big)^{\frac{2s}{2s+1}}-\Big(\sqrt{n-
\frac{1}{2}}+\langle \xi \rangle\Big)^{\frac{2s}{2s+1}}\Big].$$
It follows from (\ref{ff1}) that
\begin{multline*}
0 \leq  1-\exp\Big(- \delta t\Big[\Big(\sqrt{n+\frac{1}{2}}+\langle \xi \rangle\Big)^{\frac{2s}{2s+1}}-\Big(\sqrt{n-
\frac{1}{2}}+\langle \xi \rangle\Big)^{\frac{2s}{2s+1}}\Big]\Big)\\
\leq  \delta t\Big[\Big(\sqrt{n+\frac{1}{2}}+\langle \xi \rangle\Big)^{\frac{2s}{2s+1}}-\Big(\sqrt{n-
\frac{1}{2}}+\langle \xi \rangle\Big)^{\frac{2s}{2s+1}}\Big]
\leq  \frac{s \delta t}{(2s+1)\sqrt{n-\frac{1}{2}}\Big(\sqrt{n-\frac{1}{2}}+\langle \xi \rangle\Big)^{\frac{1}{2s+1}}}.
\end{multline*}
This implies that there exists a positive constant $c_9>0$ such that for all $n \geq 1$, $t \geq 0$, $0 \leq \delta \leq 1$,
\begin{equation}\label{qs3}
0 \leq  1-\exp\Big(- \delta t\Big[\Big(\sqrt{n+\frac{1}{2}}+\langle \xi \rangle\Big)^{\frac{2s}{2s+1}}-\Big(\sqrt{n-
\frac{1}{2}}+\langle \xi \rangle\Big)^{\frac{2s}{2s+1}}\Big]\Big)\\
\leq  c_9 \delta  t \langle n \rangle^{-\frac{1}{2}}\langle \xi \rangle^{-\frac{1}{2s+1}}.
\end{equation}
By noticing that
\begin{multline*}
\sqrt{2\pi}\big\|\langle D_x \rangle^{\frac{s+1}{2s+1}} \big(M_{\delta_1}(\delta t)v\big(M_{\delta_1}(\delta t)\big)^{-1}-v\big)f\big\|_{(1,0)}=\\
\left\|\langle \xi \rangle^{\frac{s+1}{2s+1}+1}\left(\frac{\exp\big(\delta t(\sqrt{\mathcal{H}}+\langle \xi \rangle)^{\frac{2s}{2s+1}}\big)}{1+\delta_1 \exp\big(\delta t(\sqrt{\mathcal{H}}+\langle \xi \rangle)^{\frac{2s}{2s+1}}\big)}v
\frac{1+\delta_1 \exp\big(\delta t(\sqrt{\mathcal{H}}+\langle \xi \rangle)^{\frac{2s}{2s+1}}\big)}{\exp\big(\delta t(\sqrt{\mathcal{H}}+\langle \xi \rangle)^{\frac{2s}{2s+1}}\big)}-v\right)\mathcal{F}_xf\right\|_{L^2(\rr_{\xi,v}^2)},
\end{multline*}
we deduce from (\ref{lk-1}), (\ref{qs2}) and (\ref{qs3}) that for all $0 \leq \delta \leq 1$, $0<\delta_1 \leq 1$, $t \geq 0$,
\begin{multline*}
\sqrt{2\pi}\big\|\langle D_x \rangle^{\frac{s+1}{2s+1}}\big(M_{\delta_1}(\delta t)v\big(M_{\delta_1}(\delta t)\big)^{-1}-v\big)f\big\|_{(1,0)}\\
\leq c_8 \delta t e^{c_8t}\Big(\sum_{n=0}^{+\infty}\|f_n\|_{H^{\frac{s}{2s+1}+1}(\rr_x)}^2\frac{n+1}{\langle n \rangle}\Big)^{\frac{1}{2}}+
c_9\delta  t \Big(\sum_{n=0}^{+\infty}\|f_n\|_{H^{\frac{s}{2s+1}+1}(\rr_x)}^2\frac{n}{\langle n \rangle}\Big)^{\frac{1}{2}}.
\end{multline*}
It follows that there exists a positive constant $c_{10}>0$ such that for all $0 \leq \delta \leq 1$, $0<\delta_1 \leq 1$, $t \geq 0$,
\begin{multline*}
\big\|\langle D_x \rangle^{\frac{s+1}{2s+1}}\big(M_{\delta_1}(\delta t)v\big(M_{\delta_1}(\delta t)\big)^{-1}-v\big)f\big\|_{(1,0)}
\leq c_{10} \delta t e^{c_{10}t}\Big(\sum_{n=0}^{+\infty}\|f_n\|_{H^{\frac{s}{2s+1}+1}(\rr_x)}^2\Big)^{\frac{1}{2}}\\
\leq c_{10}\delta t e^{c_{10}t}\|\langle D_x\rangle^{\frac{s}{2s+1}}f\|_{(1,0)}.
\end{multline*}
\end{proof}

\begin{lemma}\label{tt5.7}
There exists a positive constant $\tilde{c}_{2}>0$ such that for all $f \in \mathscr{S}(\rr_{x,v}^2)$, $0 \leq \delta \leq 1$, $0 < \delta_1 \leq 1$, $0 \leq t \leq T$,
$$\big\|M_{\delta_1}(\delta t)[v,\sqrt{\mathcal{H}}]\big(M_{\delta_1}(\delta t)\big)^{-1}f\big\|_{(1,0)} \leq \tilde{c}_{2}(\delta Te^{\tilde{c}_2T}+1)\|f\|_{(1,0)}.$$

\end{lemma}

\begin{proof}
Recalling that $v=A_++A_-$, we may write
\begin{multline}\label{ro2}
M_{\delta_1}(\delta t)[v,\sqrt{\mathcal{H}}]\big(M_{\delta_1}(\delta t)\big)^{-1}=M_{\delta_1}(\delta t)[A_+,\sqrt{\mathcal{H}}]\big(M_{\delta_1}(\delta t)\big)^{-1}\\
+M_{\delta_1}(\delta t)[A_-,\sqrt{\mathcal{H}}]\big(M_{\delta_1}(\delta t)\big)^{-1}.
\end{multline}
Let $f \in \mathscr{S}(\rr_{x,v}^2)$. We decompose this function into the Hermite basis in the velocity variable
$$f(x,v)=\sum_{n=0}^{+\infty}f_n(x) \psi_n(v), \quad f_n(x)=(f(x,\cdot),\psi_n)_{L^2(\rr_v)}.$$
By proceeding as in the proof of Lemma~\ref{tt5}, we notice that
\begin{multline*}
\mathcal{F}_x\big( M_{\delta_1}(\delta t)[A_+,\sqrt{\mathcal{H}}]\big(M_{\delta_1}(\delta t)\big)^{-1}f\big) \\
=  \sum_{n=0}^{+\infty}\hat{f}_n(\xi)\sqrt{n+1}\Big(\sqrt{n+\frac{1}{2}}-\sqrt{n+\frac{3}{2}}\Big)
\Big(\frac{e^{\delta t(\sqrt{n+\frac{3}{2}}+\langle \xi \rangle)^{\frac{2s}{2s+1}}-\delta t(\sqrt{n+\frac{1}{2}}+\langle \xi \rangle)^{\frac{2s}{2s+1}}}-1}{1+\delta_1e^{\delta t(\sqrt{n+\frac{3}{2}}+\langle \xi \rangle)^{\frac{2s}{2s+1}}}}+1\Big)\psi_{n+1},
\end{multline*}
\begin{multline*}
\mathcal{F}_x\big(M_{\delta_1}(\delta t)[A_-,\sqrt{\mathcal{H}}]\big(M_{\delta_1}(\delta t)\big)^{-1}f\big) \\
=  \sum_{n=0}^{+\infty}\hat{f}_n(\xi)\sqrt{n}\Big(\sqrt{n+\frac{1}{2}}-\sqrt{n-\frac{1}{2}}\Big)
\Big(\frac{e^{\delta t(\sqrt{n-\frac{1}{2}}+\langle \xi \rangle)^{\frac{2s}{2s+1}}-\delta t(\sqrt{n+\frac{1}{2}}+\langle \xi \rangle)^{\frac{2s}{2s+1}}}-1}{1+\delta_1e^{\delta t(\sqrt{n-\frac{1}{2}}+\langle \xi \rangle)^{\frac{2s}{2s+1}}}}+1\Big)\psi_{n-1},
\end{multline*}
where $\mathcal{F}_x$ stands for the partial Fourier transform with respect to the position variable.
We deduce from (\ref{qs2}) and (\ref{qs3}) that
\begin{multline}\label{ro5}
\big\| M_{\delta_1}(\delta t)[A_+,\sqrt{\mathcal{H}}]\big(M_{\delta_1}(\delta t)\big)^{-1}f\big\|_{(1,0)}^2 \\
\leq  \sum_{n=0}^{+\infty}\frac{2n+2}{4n+2}
(c_8^2 \delta^2 t^2 e^{2c_8t}\langle n \rangle^{-1}\langle \xi \rangle^{-\frac{2}{2s+1}}+1)\|f_n\|_{H^1(\rr_x)}^2
\leq (c_8^2 \delta^2 t^2 e^{2c_8t}+1)\|f\|_{(1,0)}^2,
\end{multline}
\begin{multline}\label{ro6}
\big\|M_{\delta_1}(\delta t)[A_-,\sqrt{\mathcal{H}}]\big(M_{\delta_1}(\delta t)\big)^{-1}f\big\|_{(1,0)}^2 \\
\leq  \sum_{n=1}^{+\infty} \frac{2n}{4n-2}
(c_9^2 \delta^2  t^2 \langle n \rangle^{-1}\langle \xi \rangle^{-\frac{2}{2s+1}}+1)\|f_n\|_{H^1(\rr_x)}^2 \leq (c_9^2 \delta^2  t^2+1)\|f\|_{(1,0)}^2.
\end{multline}
It follows from (\ref{ro2}), (\ref{ro5}) and (\ref{ro6}) that there exists a positive constant $\tilde{c}_{2}>0$ such that for all $f \in \mathscr{S}(\rr_{x,v}^2)$, $0 \leq \delta \leq 1$, $0 < \delta_1 \leq 1$, $0 \leq t \leq T$,
$$\big\|M_{\delta_1}(\delta t)[v,\sqrt{\mathcal{H}}]\big(M_{\delta_1}(\delta t)\big)^{-1}f\big\|_{(1,0)} \leq \tilde{c}_{2}(\delta Te^{\tilde{c}_2T}+1)\|f\|_{(1,0)}.$$
This ends the proof of Lemma~\ref{tt5.7}.
\end{proof}

\noindent
We notice from (\ref{ro1}) that
\begin{multline*}
A_{\delta_1,\delta_2}(\delta t)v\big(A_{\delta_1,\delta_2}(\delta t)\big)^{-1}-v=M_{\delta_1}(\delta t)v\big(M_{\delta_1}(\delta t)\big)^{-1}-v\\
+\frac{\delta_2 \langle D_x \rangle}{1+\delta_2\sqrt{\mathcal{H}}+\delta_2 \langle D_x \rangle}M_{\delta_1}(\delta t)[v,\sqrt{\mathcal{H}}]\big(M_{\delta_1}(\delta t) \big)^{-1}\langle D_x \rangle^{-1}.
\end{multline*}
It follows that
\begin{align*}
& \ \big|\big(\big[A_{\delta_1,\delta_2}(\delta t)v\big(A_{\delta_1,\delta_2}(\delta t)\big)^{-1}-v\big]\partial_{x}h_{n+1,\delta_2},\langle D_x \rangle^{2}Qh_{n+1,\delta_2}\big)_{L^2(\rr_{x,v}^2)}\big|\\
\leq & \ \big|\big(\big[M_{\delta_1}(\delta t)v\big(M_{\delta_1}(\delta t)\big)^{-1}-v\big]\partial_{x}h_{n+1,\delta_2},Qh_{n+1,\delta_2}\big)_{(1,0)}\big|\\
+& \  \Big|\Big(\frac{\delta_2 \langle D_x \rangle}{1+\delta_2\sqrt{\mathcal{H}}+\delta_2 \langle D_x \rangle}M_{\delta_1}(\delta t)[v,\sqrt{\mathcal{H}}]\big(M_{\delta_1}(\delta t) \big)^{-1}\langle D_x \rangle^{-1}\partial_{x}h_{n+1,\delta_2},Qh_{n+1,\delta_2}\Big)_{(1,0)}\Big|,
\end{align*}
that is
\begin{align*}
& \ \big|\big(\big[A_{\delta_1,\delta_2}(\delta t)v\big(A_{\delta_1,\delta_2}(\delta t)\big)^{-1}-v\big]\partial_{x}h_{n+1,\delta_2},\langle D_x \rangle^{2}Qh_{n+1,\delta_2}\big)_{L^2(\rr_{x,v}^2)}\big|\\
\leq & \ \big|\big(\langle D_x \rangle^{\frac{s+1}{2s+1}}\big[M_{\delta_1}(\delta t)v\big(M_{\delta_1}(\delta t)\big)^{-1}-v\big]h_{n+1,\delta_2},Q\langle D_x \rangle^{-\frac{s+1}{2s+1}}\partial_{x}h_{n+1,\delta_2}\big)_{(1,0)}\big|\\
+& \  \Big|\Big(\frac{\delta_2 \langle D_x \rangle}{1+\delta_2\sqrt{\mathcal{H}}+\delta_2 \langle D_x \rangle}M_{\delta_1}(\delta t)[v,\sqrt{\mathcal{H}}]\big(M_{\delta_1}(\delta t) \big)^{-1}\langle D_x \rangle^{-1}\partial_{x}h_{n+1,\delta_2},Qh_{n+1,\delta_2}\Big)_{(1,0)}\Big|,
\end{align*}
since according to (\ref{qsym}), $Q$ is commuting with any function of the operator~$D_x$.
We deduce from Lemmas~\ref{tt5} and~\ref{tt5.7} that there exists a positive constant $c_{11}>0$ such that for all $0 \leq \delta \leq \delta_0$, $0<\delta_1 \leq 1$, $0<\delta_2 \leq 1$, $0 \leq t \leq T$,
\begin{multline}\label{tt10}
\big|\big(\big[A_{\delta_1,\delta_2}(\delta t)v\big(A_{\delta_1,\delta_2}(\delta t)\big)^{-1}-v\big]\partial_{x}h_{n+1,\delta_2},\langle D_x \rangle^{2}Qh_{n+1,\delta_2}\big)_{L^2(\rr_{x,v}^2)}\big|\\
\leq  c_{11}\delta T e^{c_{11} T}\|\langle D_x\rangle^{\frac{s}{2s+1}}h_{n+1,\delta_2}\|_{(1,0)}^2
+c_{11}\|h_{n+1,\delta_2}\|_{(1,0)}^2,
\end{multline}
since $Q$ is a bounded operator on $L^2(\rr_{x,v}^2)$ and
$$\Big\|\frac{\delta_2 \langle D_x \rangle}{1+\delta_2\sqrt{\mathcal{H}}+\delta_2 \langle D_x \rangle}\Big\|_{\mathcal{L}(L^2(\rr_{x,v}^2))} \leq 1.$$
It follows from (\ref{tt4}) and (\ref{tt10}) that for all $0\leq \delta \leq \delta_0$, $0<\delta_1\leq 1$, $0<\delta_2\leq 1$, $0 \leq t \leq T$,
\begin{align*}
& \ \frac{1}{2}\frac{d}{dt}\|Q^{1/2}h_{n+1,\delta_2}\|_{(1,0)}^2+c_6\|\mathcal{H}^{\frac{s}{2}}h_{n+1,\delta_2}\|_{(1,0)}^2+(c_6-c_{11}\delta T e^{c_{11}T})\|\langle D_x \rangle^{\frac{s}{2s+1}}h_{n+1,\delta_2}\|_{(1,0)}^2  \\ 
\leq   & \ c_7\|h_n\|_{{L^{\infty}([0,T],H^{(1,0)}(\rr_{x,v}^2))}}\|\mathcal{H}^{\frac{s}{2}}h_{n+1}\|_{(1,0)}^2  +(c_7+c_{11})\|h_{n+1,\delta_2}\|_{(1,0)}^2,
\end{align*}
since according to (\ref{li2.3}), we have
$$\|\mathcal{H}^{\frac{s}{2}}h_{n+1,\delta_2}\|_{(1,0)} \leq \|\mathcal{H}^{\frac{s}{2}}h_{n+1}\|_{(1,0)}.$$
We may assume that the constant $0<\delta_0 \leq 1$ is chosen sufficiently small so that
$$c_{11} \delta_0 T e^{c_{11}T} \leq \frac{c_6}{2}.$$
Under the assumption
$$\|h_n\|_{L^{\infty}([0,T],H^{(1,0)}(\rr_{x,v}^2))} \leq \frac{c_6}{2c_7},$$
we obtain that for all $0 \leq \delta \leq \delta_0$, $0<\delta_1 \leq 1$, $0<\delta_2 \leq 1$,  $0 \leq t \leq T$,
\begin{multline*}
\frac{d}{dt}\|Q^{1/2}h_{n+1,\delta_2}\|_{(1,0)}^2+2c_6\|\mathcal{H}^{\frac{s}{2}}h_{n+1,\delta_2}\|_{(1,0)}^2+c_6 \|\langle D_x \rangle^{\frac{s}{2s+1}}h_{n+1,\delta_2}\|_{(1,0)}^2 \\
\leq c_6\|\mathcal{H}^{\frac{s}{2}}h_{n+1}\|_{(1,0)}^2+ 2c_{12}\|Q^{1/2}h_{n+1,\delta_2}\|_{(r,0)}^2,
\end{multline*}
with $c_{12}=(c_7+c_{11})\|(Q^{1/2})^{-1}\|_{\mathcal{L}(L^2)}>0.$
We deduce from (\ref{tr16}), (\ref{as3}) and (\ref{li2.3}) that for all $0 \leq \delta \leq \delta_0$, $0<\delta_1 \leq 1$, $0<\delta_2 \leq 1$, $0 \leq t \leq T$,
\begin{align*}
& \ \|Q^{1/2}h_{n+1,\delta_2}(t)\|_{(1,0)}^2+c_6\int_0^te^{2c_{12}(t-\tau)}(2\|\mathcal{H}^{\frac{s}{2}}h_{n+1,\delta_2}(\tau)\|_{(1,0)}^2+\|\langle D_x \rangle^{\frac{s}{2s+1}}h_{n+1,\delta_2}(\tau)\|_{(1,0)}^2)d\tau \\
\leq & \ e^{2c_{12}t}\|Q^{1/2}(1+\delta_2\sqrt{\mathcal{H}}+\delta_2 \langle D_x \rangle)^{-1}g_0\|_{(1,0)}^2+c_6\int_0^te^{2c_{12}(t-\tau)}\|\mathcal{H}^{\frac{s}{2}}h_{n+1}(\tau)\|_{(1,0)}^2d\tau\\
\leq & \ e^{2c_{12}t}\|Q^{1/2}\|_{\mathcal{L}(L^2)}^2\|g_0\|_{(1,0)}^2+c_6\int_0^te^{2c_{12}(t-\tau)}\|\mathcal{H}^{\frac{s}{2}}h_{n+1}(\tau)\|_{(1,0)}^2d\tau,
\end{align*}
since
$$h_{n+1,\delta_2}(0)=(1+\delta_2\sqrt{\mathcal{H}}+\delta_2 \langle D_x \rangle)^{-1}(1+\delta_1)^{-1}g_0.$$
By passing to the limit when $\delta_2 \to 0_+$, it follows from (\ref{li2.3}) and the convergence monotone theorem that for all $0 \leq \delta \leq \delta_0$, $0<\delta_1 \leq 1$, $0 \leq t \leq T$,
\begin{align*}
& \ \|Q^{1/2}h_{n+1}(t)\|_{(1,0)}^2+c_6\int_0^te^{2c_{12}(t-\tau)}(2\|\mathcal{H}^{\frac{s}{2}}h_{n+1}(\tau)\|_{(1,0)}^2+\|\langle D_x \rangle^{\frac{s}{2s+1}}h_{n+1}(\tau)\|_{(1,0)}^2)d\tau \\
\leq & \ e^{2c_{12}t}\|Q^{1/2}\|_{\mathcal{L}(L^2)}^2\|g_0\|_{(1,0)}^2+c_6\int_0^te^{2c_{12}(t-\tau)}\|\mathcal{H}^{\frac{s}{2}}h_{n+1}(\tau)\|_{(1,0)}^2d\tau.
\end{align*}
We obtain that for all $0 \leq \delta \leq \delta_0$, $0<\delta_1 \leq 1$,
\begin{multline}\label{ip8.4}
\|h_{n+1}\|_{L^{\infty}([0,T],H^{(1,0)}(\rr_{x,v}^2))}^2+\|\mathcal{H}^{\frac{s}{2}}h_{n+1}\|_{L^{2}([0,T],H^{(1,0)}(\rr_{x,v}^2))}^2\\
+\|\langle D_x \rangle^{\frac{s}{2s+1}}h_{n+1}\|_{L^{2}([0,T],H^{(1,0)}(\rr_{x,v}^2))}^2 \leq c_{13}e^{2c_{12}T}\|g_0\|_{(1,0)}^2,
\end{multline}
with
$$c_{13}=\|Q^{1/2}\|_{\mathcal{L}(L^2)}^2\|(Q^{1/2})^{-1}\|_{\mathcal{L}(L^2)}^2+\frac{\|Q^{1/2}\|_{\mathcal{L}(L^2)}^2}{c_6}>0.$$
Under the assumption
\begin{equation}\label{ty1}
\|g_0\|_{(1,0)} \leq \eps_1, \textrm{ with } 0<\eps_1=\inf\Big(\tilde{\eps}_0,\frac{c_6}{2c_7},\frac{c_6}{2c_7\sqrt{c_{13}}e^{c_{12}T}}\Big) \leq \tilde{\eps}_0,
\end{equation}
where the positive parameter $\tilde{\eps}_0>0$ is defined in (\ref{tr13}), the sequence of approximate solutions $(\tilde{g}_n)_{n \geq 0}$ is well-defined.
Then, we consider the sequence $(h_n)_{n \geq 0}$ defined in (\ref{as3}) and we notice from (\ref{as1}) that
$$\|h_0\|_{L^{\infty}([0,T],H^{(1,0)}(\rr_{x,v}^2))}\leq \|g_0\|_{(1,0)} \leq \frac{c_6}{2c_7}.$$
On the other hand, we deduce from (\ref{ip8.4}) and (\ref{ty1}) that the condition
$$\|h_n\|_{L^{\infty}([0,T],H^{(1,0)}(\rr_{x,v}^2))} \leq \frac{c_6}{2c_7},$$
implies that
$$\|h_{n+1}\|_{L^{\infty}([0,T],H^{(1,0)}(\rr_{x,v}^2))} \leq \frac{c_6}{2c_7}.$$
We may therefore deduce from (\ref{ip8.4}) that for all $0 \leq \delta \leq \delta_0$, $0<\delta_1 \leq 1$, $n \geq 1$,
\begin{multline*}
\|h_{n}\|_{L^{\infty}([0,T],H^{(1,0)}(\rr_{x,v}^2))}^2+\|\mathcal{H}^{\frac{s}{2}}h_{n}\|_{L^{2}([0,T],H^{(1,0)}(\rr_{x,v}^2))}^2\\
+\|\langle D_x \rangle^{\frac{s}{2s+1}}h_{n}\|_{L^{2}([0,T],H^{(1,0)}(\rr_{x,v}^2))}^2  \leq c_{13}e^{2c_{12}T}\|g_0\|_{(1,0)}^2.
\end{multline*}
This ends the proof of Lemma~\ref{sake2}. By passing to the limit when $\delta_1 \to 0_+$ in the estimate (\ref{sake6}), we deduce from the monotone convergence theorem the following result:

\begin{lemma}\label{sake2.p}
Let $T>0$. Then, there exist some positive constants $c, \eps_1>0$, $0< \delta_0 \leq 1$ such that for all initial data $\|g_0\|_{(1,0)} \leq \eps_1$, the sequence of approximate solutions $(\tilde{g}_n)_{n \geq 0}$ defined in \eqref{tr16} satisfies for all $0 \leq \delta \leq \delta_0$, $n \geq 1$,
\begin{multline*}
\|M_{0}(\delta t)\tilde{g}_n\|_{L^{\infty}([0,T],H^{(1,0)}(\rr_{x,v}^2))}^2+\|\mathcal{H}^{\frac{s}{2}}M_{0}(\delta t)\tilde{g}_n\|_{L^{2}([0,T],H^{(1,0)}(\rr_{x,v}^2))}^2\\
+\|\langle D_x \rangle^{\frac{s}{2s+1}}M_{0}(\delta t)\tilde{g}_n\|_{L^{2}([0,T],H^{(1,0)}(\rr_{x,v}^2))}^2  \leq ce^{cT}\|g_0\|_{(1,0)}^2,
\end{multline*}
with
$$M_{0}(t)=\exp\big(t(\sqrt{\mathcal{H}}+\langle D_x \rangle)^{\frac{2s}{2s+1}}\big).$$
\end{lemma}

\subsection{Gelfand-Shilov and Gevrey regularities}

We begin by noticing from the Cauchy-Schwarz inequality and (\ref{wr3}) that for all $0 \leq \delta \leq \delta_0$, $0<\delta_1 \leq 1$,
\begin{multline*}
\|M_{\delta_1}(\delta t)f(t)\|_{(1,0)}^2=\Big(\frac{\exp\big(2\delta t(\sqrt{\mathcal{H}}+\langle D_x \rangle)^{\frac{2s}{2s+1}}\big)}{\big(1+\delta_1\exp\big(\delta t(\sqrt{\mathcal{H}}+\langle D_x \rangle)^{\frac{2s}{2s+1}}\big)\big)^2}f(t),f(t)\Big)_{(1,0)}\\
\leq \|M_{0}(2\delta t)f(t)\|_{(1,0)}\|f(t)\|_{(1,0)}.
\end{multline*}
By passing to the limit when $\delta_1 \to 0_+$ in this estimate, we deduce from the monotone convergence theorem that for all $0 \leq \delta \leq \delta_0$,
$$\|M_{0}(\delta t)f(t)\|_{(1,0)}^2\leq \|M_{0}(2\delta t)f(t)\|_{(1,0)}\|f(t)\|_{(1,0)}.$$
This implies that for all $0 \leq \delta \leq \delta_0$,
\begin{equation}\label{sake3}
\|M_{0}(\delta t)f\|_{L^{\infty}([0,T],H^{(1,0)}(\rr_{x,v}^2))}^2 \leq \|M_{0}(2\delta t)f\|_{L^{\infty}([0,T],H^{(1,0)}(\rr_{x,v}^2))}\|f\|_{L^{\infty}([0,T],H^{(1,0)}(\rr_{x,v}^2))}.
\end{equation}
By using that $(\tilde{g}_n)_{n \geq 1}$ is a Cauchy sequence in $L^{\infty}([0,T],H^{(1,0)}(\rr_{x,v}^2))$,
we deduce from Lemma~\ref{sake2.p} and (\ref{sake3}) that  $(M_{0}(\delta t)\tilde{g}_n)_{n \geq 1}$ is a Cauchy sequence in $L^{\infty}([0,T],H^{(1,0)}(\rr_{x,v}^2))$,
$$\|M_{0}(\delta t)\tilde{g}_{n+p}-M_{0}(\delta t)\tilde{g}_{n}\|_{L^{\infty}([0,T],H^{(1,0)}(\rr_{x,v}^2))}^2 \leq 2\sqrt{c}e^{\frac{cT}{2}}\|g_0\|_{(1,0)}\|\tilde{g}_{n+p}-\tilde{g}_{n}\|_{L^{\infty}([0,T],H^{(1,0)}(\rr_{x,v}^2))},$$
for all $0 \leq \delta \leq \frac{\delta_0}{2}$.
Let $h$ be the limit of the Cauchy sequence $\big(M_{0}\big(\frac{\delta_0}{2}t\big)\tilde{g}_n\big)_{n \geq 1}$ in the space $L^{\infty}([0,T],H^{(1,0)}(\rr_{x,v}^2))$. By noticing that
\begin{align*}
& \ \Big\|\tilde{g}_n-\Big(M_{0}\Big(\frac{\delta_0}{2}t\Big)\Big)^{-1}h\Big\|_{L^{\infty}([0,T],H^{(1,0)}(\rr_{x,v}^2))}\\
=& \ \Big\|\Big(M_{0}\Big(\frac{\delta_0}{2}t\Big)\Big)^{-1}\Big(M_{0}\Big(\frac{\delta_0}{2}t\Big)\tilde{g}_n-h\Big)\Big\|_{L^{\infty}([0,T],H^{(1,0)}(\rr_{x,v}^2))}\\
\leq & \ \Big\|M_{0}\Big(\frac{\delta_0}{2}t\Big)\tilde{g}_n-h\Big\|_{L^{\infty}([0,T],H^{(1,0)}(\rr_{x,v}^2))},
\end{align*}
we deduce from (\ref{qw0.0}) and the uniqueness of the limit that  $g \in L^{\infty}([0,T],H^{(1,0)}(\rr_{x,v}^2))$
the solution to the Cauchy problem (\ref{qw13.1}) is equal to
\begin{equation}\label{ui1}
g=\Big(M_{0}\Big(\frac{\delta_0}{2}t\Big)\Big)^{-1}h=\exp\Big(-\frac{\delta_0}{2} t(\sqrt{\mathcal{H}}+\langle D_x \rangle)^{\frac{2s}{2s+1}}\Big)h.
\end{equation}
On the other hand, it follows from Lemma~\ref{sake2.p} that for all $n \geq 1$,
\begin{equation}\label{kh1}
\Big\|M_{0}\Big(\frac{\delta_0}{2} t\Big)\tilde{g}_n\Big\|_{L^{\infty}([0,T],H^{(1,0)}(\rr_{x,v}^2))} \leq \sqrt{c}e^{\frac{cT}{2}}\|g_0\|_{(1,0)}.
\end{equation}
By passing to the limit in the estimate (\ref{kh1}) when $n \to +\infty$, we deduce from (\ref{ui1}) that
\begin{multline}\label{ty32}
\|h\|_{L^{\infty}([0,T],H^{(1,0)}(\rr_{x,v}^2))}=\big\|\exp\big(\delta_1 t(\sqrt{\mathcal{H}}+\langle D_x \rangle)^{\frac{2s}{2s+1}}\big)g\big\|_{L^{\infty}([0,T],H^{(1,0)}(\rr_{x,v}^2))}\\ \leq \sqrt{c}e^{\frac{cT}{2}}\|g_0\|_{(1,0)},
\end{multline}
with $\delta_1=\frac{\delta_0}{2}>0$.
Next, we notice that
\begin{multline}\label{ty35}
\forall x,c>0, \quad x^k\exp\Big(-\frac{2s+1}{s}cx^{\frac{s}{2s+1}}\Big)=\Big(\frac{(cx^{\frac{s}{2s+1}})^k}{c^k}e^{-cx^{\frac{s}{2s+1}}}\Big)^{\frac{2s+1}{s}}\\
=\frac{(k!)^{\frac{2s+1}{s}}}{c^{\frac{2s+1}{s}k}}\Big(\frac{(cx^{\frac{s}{2s+1}})^k}{k!}e^{-cx^{\frac{s}{2s+1}}}\Big)^{\frac{2s+1}{s}}
\leq \frac{(k!)^{\frac{2s+1}{s}}}{c^{\frac{2s+1}{s}k}},
\end{multline}
since
$$\forall k \geq 0, \quad \frac{(cx^{\frac{s}{2s+1}})^k}{k!} \leq e^{cx^{\frac{s}{2s+1}}}.$$
Let $f \in \mathscr{S}(\rr_{x,v}^2)$ be a Schwartz function. We decompose this function into the Hermite basis in the velocity variable
$$f(x,v)=\sum_{n=0}^{+\infty}f_n(x) \psi_n(v), \quad f_n(x)=(f(x,\cdot),\psi_n)_{L^2(\rr_v)}.$$
We deduce from (\ref{ty35}) that for all $k \geq 0$,
\begin{align}\label{ty36}
& \ \big\|(\sqrt{\mathcal{H}}+\langle D_x \rangle)^k\exp\big(-\delta_1 t(\sqrt{\mathcal{H}}+\langle D_x \rangle)^{\frac{2s}{2s+1}}\big)f\big\|_{(1,0)}^2\\
\notag
=& \ \frac{1}{2\pi}\sum_{n=0}^{+\infty}\int_{\rr}\langle \xi \rangle^{2}\Big(\sqrt{n+\frac{1}{2}}+\langle \xi \rangle\Big)^{2k}\exp\Big(-2\delta_1 t\Big(\sqrt{n+\frac{1}{2}}+\langle \xi \rangle\Big)^{\frac{2s}{2s+1}}\Big)|\widehat{f}_n(\xi)|^2d\xi\\ \notag
\leq & \ \Big(\frac{2s+1}{2s\delta_1}\Big)^{\frac{2s+1}{s}k} \frac{(k!)^{\frac{2s+1}{s}}}{t^{\frac{2s+1}{s}k}}\|f\|_{(1,0)}^2.
\end{align}
It follows from (\ref{ty32}) and (\ref{ty36}) that the solution to the Cauchy problem (\ref{qw13.1}) satisfies for all $0<t \leq T$, $k \geq 0$,
\begin{align*}
& \ \|(\sqrt{\mathcal{H}}+\langle D_x \rangle)^kg(t)\|_{(1,0)}\\
= & \ \big\|(\sqrt{\mathcal{H}}+\langle D_x \rangle)^k\exp\big(-\delta_1 t(\sqrt{\mathcal{H}}+\langle D_x \rangle)^{\frac{2s}{2s+1}}\big)\exp\big(\delta_1 t(\sqrt{\mathcal{H}}+\langle D_x \rangle)^{\frac{2s}{2s+1}}\big)g(t)\big\|_{(1,0)}\\
\leq & \ \Big(\frac{2s+1}{2s\delta_1}\Big)^{\frac{2s+1}{2s}k} \frac{(k!)^{\frac{2s+1}{2s}}}{t^{\frac{2s+1}{2s}k}}\sqrt{c}e^{\frac{cT}{2}}\|g_0\|_{(1,0)}.
\end{align*}
This implies that there exists a positive constant $C>1$ such that
\begin{equation}\label{ty37}
\forall 0<t \leq T, \forall k \geq 0, \quad \|(\sqrt{\mathcal{H}}+\langle D_x \rangle)^kg(t)\|_{(1,0)}\leq \frac{1}{t^{\frac{2s+1}{2s}k}}C^{k+1} (k!)^{\frac{2s+1}{2s}}\|g_0\|_{(1,0)}.
\end{equation}
It proves the Gelfand-Shilov property in Theorem~\ref{qw13.66ee}.
We may therefore notice from (\ref{ty37}) that for all $0<t \leq T$, $k \geq 0$,
$$\|g(t)\|_{H^k(\rr_{x,v}^2)} \lesssim \|(\sqrt{\mathcal{H}}+\langle D_x \rangle)^kg(t)\|_{(1,0)} <+\infty.$$
This implies in particular that
$$\forall 0<t \leq T, \quad g(t) \in C^{\infty}(\rr_{x,v}^2).$$
On the other hand, we notice that for all $p \geq 0$,
\begin{equation}\label{ge11}
\partial_x^pg(t,x,v)=\sum_{n=0}^{+\infty}\partial_x^pg_n(t,x)\psi_n(v), \quad \textrm{ with } \quad g_n(t,x)=(g(t,x,\cdot),\psi_n)_{L^2(\rr_v)},
\end{equation}
since
$$\partial_x^pg_n(t,x)=(\partial_x^pg(t,x,\cdot),\psi_n)_{L^2(\rr_v)}.$$
It follows from (\ref{ge11}), Lemma~\ref{ge3} and the Sobolev imbedding that there exist some positive constants $C_1,C_2, C_3>0$ such that for all $k,l,p \geq 0$, $\eps >0$,
\begin{align}\label{ge12}
& \ \|v^k\partial_v^l\partial_x^pg(t)\|_{L^{\infty}(\rr_x,L^2(\rr_v))} \leq \sum_{n=0}^{+\infty}\|\partial_x^pg_n(t)\|_{L^{\infty}(\rr_x)}\|v^k\partial_v^l\psi_n\|_{L^2(\rr_v)} \\ \notag
\leq & \ C_1\Big(\frac{C_2}{\inf(\eps^{\frac{2s+1}{2s}},1)}\Big)^{k+l}(k!)^{\frac{2s+1}{2s}}(l!)^{\frac{2s+1}{2s}}\sum_{n=0}^{+\infty}\|\partial_x^pg_n(t)\|_{L^{\infty}(\rr_x)}\big((1-\delta_{n,0})e^{\eps {\frac{2s+1}{2s}}n^{\frac{s}{2s+1}}}+\delta_{n,0}\big)\\ \notag
\leq & \ C_3\Big(\frac{C_2}{\inf(\eps^{\frac{2s+1}{2s}},1)}\Big)^{k+l}(k!)^{\frac{2s+1}{2s}}(l!)^{\frac{2s+1}{2s}}\sum_{n=0}^{+\infty}\|\partial_x^pg_n(t)\|_{H^{1}(\rr_x)}\big((1-\delta_{n,0})e^{\eps {\frac{2s+1}{2s}}n^{\frac{s}{2s+1}}}+\delta_{n,0}\big),
\end{align}
where $\delta_{n,0}$ stands for the Kronecker delta, i.e., $\delta_{n,0}=1$ if $n=0$, $\delta_{n,0}=0$ if $n \neq 0$.
We notice that the estimate (\ref{ty32})
\begin{multline*}
\forall 0 \leq t \leq T, \quad \big\|\exp\big(\delta_1 t(\sqrt{\mathcal{H}}+\langle D_x \rangle)^{\frac{2s}{2s+1}}\big)g(t)\big\|_{(1,0)}^2\\
=\sum_{n=0}^{+\infty}\Big\|\exp\Big(\delta_1 t\Big(\sqrt{n+\frac{1}{2}}+\langle D_x \rangle\Big)^{\frac{2s}{2s+1}}\Big)g_n(t)\Big\|_{H^{1}(\rr_x)}^2
 \leq ce^{cT}\|g_0\|_{(1,0)}^2 <+\infty,
\end{multline*}
implies that
\begin{equation}\label{ge13}
\forall 0 \leq t \leq T, \quad \sup_{n \geq 0}\Big\|\exp\Big(\delta_1 t\Big(\sqrt{n+\frac{1}{2}}+\langle D_x \rangle\Big)^{\frac{2s}{2s+1}}\Big)g_n(t)\Big\|_{H^{1}(\rr_x)} \leq  \sqrt{c}e^{\frac{cT}{2}}\|g_0\|_{(1,0)}.
\end{equation}
We have
\begin{multline*}
\|\partial_x^pg_n(t)\|_{H^{1}(\rr_x)}^2=\frac{1}{2\pi}\int_{\rr}|\xi|^{2p}\langle \xi \rangle^{2}|\widehat{g_n}(\xi)|^2d\xi\\
 = \frac{1}{2\pi}\int_{\rr}|\xi|^{2p}\exp\Big(-2\delta_1 t\Big(\sqrt{n+\frac{1}{2}}+\langle \xi \rangle\Big)^{\frac{2s}{2s+1}}\Big)\langle \xi \rangle^{2}\Big|\exp\Big(\delta_1 t\Big(\sqrt{n+\frac{1}{2}}+\langle \xi \rangle\Big)^{\frac{2s}{2s+1}}\Big)\widehat{g_n}(\xi)\Big|^2d\xi.
\end{multline*}
We obtain that
\begin{multline}\label{ge14}
\|\partial_x^pg_n(t)\|_{H^{1}(\rr_x)}^2 \leq \frac{1}{2\pi}\exp\Big(-\delta_1 t\Big(n+\frac{1}{2}\Big)^{\frac{s}{2s+1}}\Big) \\ \times
\int_{\rr}\langle\xi \rangle^{2p}e^{-\delta_1 t\langle \xi \rangle^{\frac{2s}{2s+1}}}\langle\xi \rangle^{2}\Big|\exp\Big(\delta_1 t\Big(\sqrt{n+\frac{1}{2}}+\langle \xi \rangle\Big)^{\frac{2s}{2s+1}}\Big)\widehat{g_n}(\xi)\Big|^2d\xi.
\end{multline}
By using that
\begin{multline*}
\forall 0<t \leq T, \forall p \geq 0, \forall \xi \in \rr, \quad \langle\xi \rangle^{2p}e^{-\delta_1 t\langle \xi \rangle^{\frac{2s}{2s+1}}}\\
=\Big(\frac{(\frac{s \delta_1 t}{2s+1} \langle \xi \rangle^{\frac{2s}{2s+1}} )^p}{p!} \Big)^{\frac{2s+1}{s}}e^{-\delta_1 t\langle \xi \rangle^{\frac{2s}{2s+1}}}\Big(\frac{2s+1}{s \delta_1 t}\Big)^{\frac{2s+1}{s}p}(p!)^{\frac{2s+1}{s}}
\leq \Big(\frac{2s+1}{s \delta_1 t}\Big)^{\frac{2s+1}{s}p}(p!)^{\frac{2s+1}{s}},
\end{multline*}
it follows from (\ref{ge13}) and (\ref{ge14}) that for all $0<t \leq T$, $n, p \geq 0$,
\begin{equation}\label{ge15}
\|\partial_x^pg_n(t)\|_{H^{1}(\rr_x)} \leq \sqrt{c}e^{\frac{cT}{2}}\|g_0\|_{(1,0)}\exp\Big(-\frac{\delta_1 t}{2}\Big(n+\frac{1}{2}\Big)^{\frac{s}{2s+1}}\Big)
\Big(\frac{2s+1}{s \delta_1 t}\Big)^{\frac{2s+1}{2s}p}(p!)^{\frac{2s+1}{2s}}.
\end{equation}
We deduce from (\ref{ge12}) and (\ref{ge15}) that for all $0<t \leq T$, $k,l,p \geq 0$,
\begin{align*}
 & \ \|v^k\partial_v^l\partial_x^pg(t)\|_{L^{\infty}(\rr_x,L^2(\rr_v))}  \\ 
 \leq & \   \sqrt{c}e^{\frac{cT}{2}}\|g_0\|_{(1,0)}C_3\Big(\frac{C_2}{\inf(\eps^{\frac{2s+1}{2s}},1)}\Big)^{k+l}\Big(\frac{2s+1}{s \delta_1 t}\Big)^{\frac{2s+1}{2s}p}(k!)^{\frac{2s+1}{2s}}(l!)^{\frac{2s+1}{2s}} (p!)^{\frac{2s+1}{2s}}  \\ 
 \times & \  \sum_{n=0}^{+\infty}
\Big((1-\delta_{n,0})  \exp\Big(\eps {\frac{2s+1}{2s}} n^{\frac{s}{2s+1}}-\frac{\delta_1 t}{2}\Big(n+\frac{1}{2}\Big)^{\frac{s}{2s+1}}\Big)+\delta_{n,0}\exp\Big(-\frac{\delta_1 t}{2}\Big(n+\frac{1}{2}\Big)^{\frac{s}{2s+1}}\Big)\Big).
\end{align*}
If we choose
$$\eps=\frac{s \delta_1 t}{4s+2}>0,$$
we obtain that there exist some positive constants $C_4,C_5>0$ such that for all $0<t \leq T$, $k,l,p \geq 0$,
\begin{equation}\label{ge17}
 \|v^k\partial_v^l\partial_x^pg(t)\|_{L^{\infty}(\rr_x,L^2(\rr_v))}  \leq   C_4 C_5^{k+l+p}\frac{\tilde{F}(t)}{t^{\frac{2s+1}{2s}(k+l+p)}} (k!)^{\frac{2s+1}{2s}}(l!)^{\frac{2s+1}{2s}}(p!)^{\frac{2s+1}{2s}}\|g_0\|_{(1,0)} ,
\end{equation}
where
$$\tilde{F}(x)=\sum_{n=0}^{+\infty}\exp\Big(-\frac{\delta_1 x}{4}\Big(n+\frac{1}{2}\Big)^{\frac{s}{2s+1}}\Big), \quad x > 0.$$
Let $\eta_0>0$ be a positive parameter. We notice that for all $x>0$
\begin{align}\label{saoi1}
& \ x^{\frac{2s+1}{s}+\eta_0}\tilde{F}(x)\\ \notag
=& \ \sum_{n=0}^{+\infty}\Big(\frac{\delta_1 x}{4}\Big(n+\frac{1}{2}\Big)^{\frac{s}{2s+1}}\Big)^{\frac{2s+1}{s}+\eta_0}\exp\Big(-\frac{\delta_1 x}{4}\Big(n+\frac{1}{2}\Big)^{\frac{s}{2s+1}}\Big)\frac{1}{(\frac{\delta_1}{4}(n+\frac{1}{2})^{\frac{s}{2s+1}})^{\frac{2s+1}{s}+\eta_0}}\\ \notag
\leq & \ \|y^{\frac{2s+1}{s}+\eta_0}e^{-y}\|_{L^{\infty}([0,+\infty[)}\sum_{n=0}^{+\infty}\frac{1}{(\frac{\delta_1}{4}(n+\frac{1}{2})^{\frac{s}{2s+1}})^{\frac{2s+1}{s}+\eta_0}}<+\infty.
\end{align}
It follows from (\ref{ge17}) and (\ref{saoi1}) that for any $\eta_0>0$, there exist some positive constants $C_6,C_7>0$ such that for all $0<t \leq T$,  $k,l,p \geq 0$,
\begin{equation}\label{saoi2}
 \|v^k\partial_v^l\partial_x^pg(t)\|_{L^{\infty}(\rr_x,L^2(\rr_v))} \leq   \frac{C_6C_7^{k+l+p}}{t^{\frac{2s+1}{2s}(k+l+p+2)+\eta_0}} (k!)^{\frac{2s+1}{2s}}(l!)^{\frac{2s+1}{2s}}(p!)^{\frac{2s+1}{2s}}\|g_0\|_{(1,0)}.
\end{equation}
We deduce from (\ref{saoi2}) and the Sobolev imbedding theorem that there exist some positive constants $C_8,C_9,C_{10}>0$ such that for all $0<t \leq T$, $k,l,p \geq 0$,
\begin{align*}
& \ \|v^k\partial_v^l\partial_x^pg(t)\|_{L^{\infty}(\rr_{x,v}^2)} \\ 
 \leq & \  C_{8}(\|v^k\partial_v^l\partial_x^pg(t)\|_{L^{\infty}(\rr_x,L^2(\rr_v))}+\|v^k\partial_v^{l+1}\partial_x^pg(t)\|_{L^{\infty}(\rr_x,L^2(\rr_v))}) \\ 
 \leq  & \  \frac{C_{9}C_{10}^{k+l+p}}{t^{\frac{2s+1}{2s}(k+l+p+3)+\eta_0}}
 (k!)^{\frac{2s+1}{2s}}(l!)^{\frac{2s+1}{2s}}(p!)^{\frac{2s+1}{2s}}\|g_0\|_{(1,0)}.
\end{align*}
It proves the Gevrey smoothing property in Theorem~\ref{qw13.66ee}.

\section{Appendix}\label{appendix}

\subsection{Hermite functions}\label{6.sec.harmo}
 The standard Hermite functions $(\phi_{n})_{n\in \N}$ are defined for $v \in \rr$,
 \begin{equation}\label{Ar1}
 \phi_{n}(v)=\frac{(-1)^n}{\sqrt{2^n n!\sqrt{\pi}}} e^{\frac{v^2}{2}}\frac{d^n}{dv^n}(e^{-v^2})
 =\frac{1}{\sqrt{2^n n!\sqrt{\pi}}} \Bigl(v-\frac{d}{dv}\Bigr)^n(e^{-\frac{v^2}{2}})=\frac{ a_{+}^n \phi_{0}}{\sqrt{n!}},
 \end{equation}
where $a_{+}$ is the creation operator
$$a_{+}=\frac{1}{\sqrt{2}}\Big(v-\frac{d}{dv}\Big).$$
The family $(\phi_{n})_{n\in \N}$ is an orthonormal basis of $L^2(\R)$.
We set for $n\in \N$, $v\in \R$,
\begin{equation}\label{Ar2}
\psi_n(v)=2^{-1/4}\phi_n(2^{-1/2}v),\quad \psi_{n}=\frac{1}{\sqrt{n!}}\Bigl(\frac{v}2-\frac{d}{dv}\Bigr)^n\psi_{0}.
\end{equation}
The family $(\psi_{n})_{n \in \nn}$ is an orthonormal basis of $L^2(\R)$
composed by the eigenfunctions of the harmonic oscillator
$$\mathcal{H}=-\Delta_v+\frac{v^2}{4}=\sum_{n\ge 0}\Big(n+\frac{1}{2}\Big)\mathbb P_{n},\quad 1=\sum_{n \ge 0}\mathbb P_{n},$$
where $\mathbb P_{n}$ stands for the orthogonal projection
$$\mathbb P_{n}f=(f,\psi_n)_{L^2(\rr_v)}\psi_n.$$
It satisfies the identities
\begin{equation}\label{ge1}
A_+\psi_n=\sqrt{n+1}\psi_{n+1}, \quad A_-\psi_n=\sqrt{n}\psi_{n-1},
\end{equation}
where
\begin{equation}\label{ge2}
A_{\pm}=\frac{v}{2}\mp \frac{d}{dv}.
\end{equation}
Instrumental in the core of the article are the estimates on the Hermite functions given in the following lemma which are an adaptation in a simpler setting of the analysis led in the work~\cite{langen} (Lemma~3.2).

\begin{lemma}\label{ge3}
We have
\begin{equation}\label{ge4}
\forall n, k,l \geq 0, \quad \|v^k\partial_v^l\psi_n\|_{L^2(\rr)} \leq 2^k\sqrt{\frac{(k+l+n)!}{n!}},
\end{equation}
\begin{multline}\label{ge5}
\forall r \geq \frac{1}{2}, \forall \eps >0, \forall n, k,l \geq 0,  \\ \|v^k\partial_v^l\psi_n\|_{L^2(\rr)} \leq \sqrt{2}\big((1-\delta_{n,0})\exp(\eps r n^{\frac{1}{2r}})+\delta_{n,0}\big)
\Big(\frac{2^{\frac{3}{2}+r}e^r}{\inf(\eps^r,1)}\Big)^{k+l}(k!)^r(l!)^r,
\end{multline}
where $\delta_{n,0}$ stands for the Kronecker delta, i.e., $\delta_{n,0}=1$ if $n=0$, $\delta_{n,0}=0$ if $n \neq 0$.  
\end{lemma}

\begin{proof}
The estimate (\ref{ge4}) is trivial if $k=l=0$, since the family $(\psi_{n})_{n \in \nn}$ is an orthonormal basis of $L^2(\R)$. We notice from (\ref{ge1}) and (\ref{ge2}) that
\begin{equation}\label{ge6}
v\psi_n=(A_++A_-)\psi_n=\sqrt{n+1}\psi_{n+1}+\sqrt{n}\psi_{n-1},
\end{equation}
\begin{equation}\label{ge7}
\partial_v\psi_n=\frac{1}{2}(A_--A_+)\psi_n=\frac{1}{2}\sqrt{n}\psi_{n-1}-\frac{1}{2}\sqrt{n+1}\psi_{n+1}.
\end{equation}
This implies that
$$\|v\psi_n\|_{L^2(\rr)}=\sqrt{2n+1}, \quad \|\partial_v\psi_n\|_{L^2(\rr)}=\frac{1}{2}\sqrt{2n+1},$$
since $(\psi_{n})_{n \in \nn}$ is an orthonormal basis of $L^2(\R)$.
It follows that the estimate (\ref{ge4}) holds as well when $(k,l)=(1,0)$ or $(k,l)=(0,1)$. We complete the proof of the estimate (\ref{ge4}) by induction. We assume that the estimate holds for any $k,l \geq 0$, $k+l \leq m$, with $m \geq 1$. Let $k,l \geq 0$ such that $k+l=m$. It follows from (\ref{ge6}) and (\ref{ge7}) that
$$v^{k+1}\partial_v^l\psi_n=\sqrt{n+1}v^{k}\partial_v^l\psi_{n+1}+\sqrt{n}v^{k}\partial_v^l\psi_{n-1}-lv^{k}\partial_v^{l-1}\psi_{n},$$
$$v^k\partial_v^{l+1}\psi_n=\frac{1}{2}\sqrt{n}v^{k}\partial_v^l\psi_{n-1}-\frac{1}{2}\sqrt{n+1}v^{k}\partial_v^l\psi_{n+1}.$$
We deduce from the induction hypothesis that
\begin{multline*}
\|v^{k+1}\partial_v^l\psi_n\|_{L^2(\rr)} \leq \sqrt{n+1}\|v^{k}\partial_v^l\psi_{n+1}\|_{L^2(\rr)}+\sqrt{n}\|v^{k}\partial_v^l\psi_{n-1}\|_{L^2(\rr)}+l\|v^{k}\partial_v^{l-1}\psi_{n}\|_{L^2(\rr)} \\ \leq
2^{k+1}\sqrt{\frac{(k+l+n+1)!}{n!}}\Big(\frac{1}{2}+\frac{\sqrt{n}\sqrt{n}+l}{2\sqrt{(k+l+n+1)(k+l+n)}}\Big) \leq 2^{k+1}\sqrt{\frac{(k+l+n+1)!}{n!}},
\end{multline*}
\begin{multline*}
\|v^k\partial_v^{l+1}\psi_n\|_{L^2(\rr)} \leq \frac{1}{2}\sqrt{n}\|v^{k}\partial_v^l\psi_{n-1}\|_{L^2(\rr)}+\frac{1}{2}\sqrt{n+1}\|v^{k}\partial_v^l\psi_{n+1}\|_{L^2(\rr)}\\ \leq
2^{k}\sqrt{\frac{(k+l+n+1)!}{n!}}\Big(\frac{1}{2}+\frac{\sqrt{n}\sqrt{n}}{2\sqrt{(k+l+n+1)(k+l+n)}}\Big) \leq 2^{k}\sqrt{\frac{(k+l+n+1)!}{n!}}.
\end{multline*}
This ends the proof of the estimate (\ref{ge4}). 
We then prove the estimates (\ref{ge5}).
When $n=0$, we deduce from (\ref{ge4}) that
$$\forall k,l \geq 0, \quad \|v^k\partial_v^l\psi_0\|_{L^2(\rr)} \leq 2^k\sqrt{{(k+l)!}}\le 2^{\frac{3k+l}{2}}\sqrt{{k!}}\sqrt{{l!}},$$
since 
$$C_{k+l}^k=\frac{(k+l)!}{k!l!} \leq 2^{k+l}.$$
It follows that 
$$\forall r \geq \frac{1}{2}, \forall \eps >0, \forall k,l \geq 0, \quad 
\|v^k\partial_v^l\psi_0\|_{L^2(\rr)} \leq \sqrt{2}\Big(\frac{2^{\frac{3}{2}+r}e^r}{\inf(\eps^r,1)}\Big)^{k+l}(k!)^r(l!)^r,$$
since
$$\forall r \geq \frac{1}{2}, \forall \eps >0, \quad 2^{\frac{3k+l}{2}}\sqrt{{k!}}\sqrt{{l!}} \leq (2^{\frac{3}{2}})^{k+l}({k!})^r({l!})^r \leq \sqrt{2}\Big(\frac{2^{\frac{3}{2}+r}e^r
}{\inf(\eps^r,1)}\Big)^{k+l}({k!})^r({l!})^r.$$
The estimates (\ref{ge5}) therefore hold when $n=0$.
When $k=l=0$ and $n \geq 1$, the estimates (\ref{ge5}) also hold since $\|\psi_n\|_{L^2(\rr)} =1$.
From now, we may therefore assume that $k+l \geq 1$ and $n\ge 1$.
We notice that for all $n\ge 1$,
\begin{multline*}
n!=\mathbf{\Gamma}(n+1)=\int_{0}^{+\io} e^{-t} t^{n}dt=\Big(\frac n e\Big)^{n}\int_{0}^{+\io} ne^{-(s-1)n} s^{n}ds \\ \ge \Big(\frac n e\Big)^{n}\int_{1}^{2} ne^{-(s-1)n} ds 
=\Big(\frac n e\Big)^{n}(1-e^{-n})\ge\frac12\Big(\frac n e\Big)^{n},
\end{multline*}
so that 
$$\forall n \geq 1, \quad n^{\frac n2}\le \sqrt{2} \sqrt{n!} e^{\frac{n}{2}}.$$
It follows that 
\begin{equation}\label{vbn1}
\forall r\ge \frac{1}{2}, \forall n\ge 1,\quad n^{\frac n2}\le \sqrt{2} \sqrt{n! e^{n}}\le  \sqrt{2}\bigl({n! e^{n}}\bigr)^{r}.
\end{equation}
We distinguish two cases. When $1\le k+l \leq n$, we deduce from (\ref{ge4}) that for all $r \geq 1/2$, $\eps >0$,
\begin{multline}\label{ge9}
\|v^k\partial_v^l\psi_n\|_{L^2(\rr)} \leq 2^k\sqrt{\frac{(k+l+n)!}{n!}} \leq 2^k(k+l+n)^{\frac{k+l}{2}}\leq 2^k(2n)^{\frac{k+l}{2}}\\
\leq \Big(\frac{2^{\frac{3}{2}}}{\eps^r}\Big)^{k+l}((k+l)!)^r\Big(\frac{(\eps n^{\frac{1}{2r}})^{k+l}}{(k+l)!}\Big)^r
\leq \Big(\frac{2^{\frac{3}{2}}}{\eps^r}\Big)^{k+l}\exp(\eps r n^{\frac{1}{2r}})((k+l)!)^r.
\end{multline}
When $k+l > n\ge 1$, we deduce from (\ref{ge4}) and (\ref{vbn1}) that for all $r \geq 1/2$,
\begin{multline}\label{ge10}
\|v^k\partial_v^l\psi_n\|_{L^2(\rr)} \leq 2^k\sqrt{\frac{(k+l+n)!}{n!}} \leq 2^k(k+l+n)^{\frac{k+l}{2}}\leq 2^k(2k+2l)^{\frac{k+l}{2}}\\
\leq  \sqrt 2 (2^{\frac{3}{2}}e^{r})^{k+l}((k+l)!)^r.
\end{multline}
It follows from (\ref{ge9}) and (\ref{ge10}) that for all $r \geq 1/2$, $\eps >0$, $n \geq 1$, $k+l \geq 1$, 
$$\|v^k\partial_v^l\psi_n\|_{L^2(\rr)} \leq \sqrt 2 \Big(\frac{2^{\frac{3}{2}}e^r}{\inf(\eps^r,1)}\Big)^{k+l}\exp(\eps r n^{\frac{1}{2r}})((k+l)!)^r.$$
By using that 
$$\frac{(k+l)!}{k!l!}=C_{k+l}^k \leq 2^{k+l},$$
we finally obtain that for all $r \geq 1/2$, $\eps >0$, $n \ge 1$, $k+l \geq 1$,
$$\|v^k\partial_v^l\psi_n\|_{L^2(\rr)} \leq \sqrt 2 \Big(\frac{2^{\frac{3}{2}+r}e^r}{\inf(\eps^r,1)}\Big)^{k+l}\exp(\eps r n^{\frac{1}{2r}})(k!)^r(l!)^r.$$
The estimates (\ref{ge5}) therefore hold when $k+l \geq 1$ and $n \geq 1$. The proof of Lemma~\ref{ge3} is complete.
\end{proof}

\subsection{Gelfand-Shilov regularity}\label{regularity}We refer the reader to the works~\cite{gelfand,rodino1,rodino,toft} and the references herein for extensive expositions of the Gelfand-Shilov regularity theory.  The Gelfand-Shilov spaces $S_{\nu}^{\mu}(\rr)$, with $\mu,\nu>0$, $\mu+\nu\geq 1$, are defined as the spaces of smooth functions $f \in C^{\infty}(\rr)$ satisfying
$$\exists C \geq 1, \quad |\partial_v^{p}f(v)| \leq C^{p+1}(p!)^{\mu}e^{-\frac{1}{C}|v|^{1/\nu}}, \quad v \in \rr, \ p \geq 0,$$
or, equivalently
$$\exists C \geq 1, \quad \sup_{v \in \rr}|v^q\partial_v^{p}f(v)| \leq C^{p+q+1}(p !)^{\mu}(q !)^{\nu}, \quad p, q \geq 0.$$
These Gelfand-Shilov spaces  $S_{\nu}^{\mu}(\rr)$ may also be characterized as the spaces of Schwartz functions $f \in \mathscr{S}(\rr)$ satisfying
$$\exists C>0, \eps>0, \quad |f(v)| \leq C e^{-\eps|v|^{1/\nu}}, \quad v \in \rr, \qquad |\widehat{f}(\xi)| \leq C e^{-\eps|\xi|^{1/\mu}}, \quad \xi \in \rr.$$
In particular, we notice that Hermite functions belong to the symmetric Gelfand-Shilov space  $S_{1/2}^{1/2}(\rr)$. More generally, the symmetric Gelfand-Shilov spaces $S_{\mu}^{\mu}(\rr)$, with $\mu \geq 1/2$, can be characterized through the decomposition into the Hermite basis $(\psi_{n})_{n \geq 0}$,
see e.g. \cite{toft} (Proposition~1.2),
\begin{multline*}
f \in S_{\mu}^{\mu}(\rr) \Leftrightarrow f \in L^2(\rr), \ \exists t_0>0, \ \big\|\big((f,\psi_{n})_{L^2}\exp({t_0n^{\frac{1}{2\mu}})}\big)_{n \geq 0}\big\|_{l^2(\nn)}<+\infty\\
\Leftrightarrow f \in L^2(\rr), \ \exists t_0>0, \ \|e^{t_0\mathcal{H}^{1/2\mu}}f\|_{L^2}<+\infty,
\end{multline*}
where $\mathcal{H}=-\Delta_v+\frac{v^2}{4}$ is the harmonic oscillator and $(\psi_{n})_{n \geq 0}$ stands for the Hermite basis defined in Section~\ref{6.sec.harmo}.

\subsection{The Kac collision operator}\label{kacsection}

For $\phi$ a function defined on $\R$, we denote its even part
$$\breve{\phi} (\theta)=\frac12\bigl(\phi(\theta)+\phi(-\theta)\bigr).$$
The following lemma is proved in~\cite{LMPX1} (Lemma~A.1):

\begin{lemma}\label{new003}
Let $\nu\in L^1_{loc}({\rr^*})$ be an even function such that $\theta^2\nu(\theta)\in L^1(\rr)$. Then, the mapping
$$ \phi \in C^2_{c}(\rr) \mapsto\lim_{\varepsilon\rightarrow 0_{+}}\int_{\vert\theta\vert\ge \varepsilon}
\nu(\theta)\bigl(\phi(\theta)-\phi(0)\bigr) d\theta=\int_{0}^1\int_{\rr}(1-t)\theta^2\nu(\theta)\phi''(t\theta) d\theta dt,$$
defines a distribution of order 2 denoted $\finp{(\nu)}$. The linear form $\finp{(\nu)}$ can be extended to $C^{1,1}$ functions
($C^1$ functions whose second derivative is $L^\io$).
For $\phi\in C^{1,1}$ satisfying $\phi(0)=0$, the function $\nu \breve\phi$ belongs to $L^1(\R)$ and
$$\poscal{\finp{(\nu)}}{\phi}=\int\nu(\theta)\breve{\phi}(\theta)d\theta.$$
\end{lemma}

Let $g,f \in \mathscr{S}(\rr)$ be Schwartz functions. We define
$$F_{f,g}(\underbrace{v,v_{*}}_{w})= f(v) g(v_{*}),\quad
\phi_{f,g}(\theta,v)=\int_{\rr}\bigl(F_{f,g}(R_{\theta}w)-F_{f,g}(w)\bigr) dv_{*},$$
where $R_{\theta}$ stands for the rotation of angle $\theta$ in $\R^2$,
$$R_{\theta}=\mat22{\cos\theta}{-\sin\theta}{\sin\theta}{\cos\theta}=\exp(\theta J),\quad J=R_{\frac{\pi}{2}}.$$
We have
$$F_{f,g}(R_{\theta}w)-F_{f,g}(w)=f(v\cos\theta - v_*\sin\theta)g(v\sin\theta +v_*\cos\theta)-f(v) g(v_{*}),$$
so that by using the notations $f_*'=f(v_*')$, $f'=f(v')$, $f_*=f(v_*)$, $f=f(v)$ with
$$v'=v\cos\theta - v_*\sin\theta, \quad v'_* =v\sin\theta +v_*\cos\theta, \quad v,v_* \in \rr,$$
we may write
$$\phi_{f,g}(\theta,v)=\int_{\R}(g'_{*}f'-g_{*}f)dv_{*}.$$
Furthermore, we easily check that its even part as a function of the variable $\theta$ is given by
$$\breve{\phi}_{f,g}(\theta,v)=\int_{\R}\bigl((\breve{g})'_{*}f'-g_{*}f\bigr) dv_{*}
=\int_{\R}\bigl((\breve{g})'_{*}f'-(\breve{g})_{*}f\bigr) dv_{*}.$$
Notice that for each $\theta \in \rr$, the mapping
$$(f,g) \in \mathscr S(\R)\times\mathscr S(\R) \mapsto \phi_{f,g}(\theta,\cdot)\in \mathscr S(\R),$$
is continuous uniformly with respect to~$\theta$. In fact, the function $F_{f,g}$ belongs to $\mathscr S(\R^2)$. By denoting~$\Pi_1$ the projection onto the first variable, this implies that the function
$$v^l\p_{v}^k\phi_{f,g}(\theta,v)=\int\Pi_{1}(w)^l\p_{v}^k \Phi_{f,g}(\theta,w) dv_{*},$$
is bounded since
$$\Phi_{f,g}(\theta,w)=F_{f,g}(R_{\theta}w)-F_{f,g}(w) \in \mathscr S(\R^2).$$
As a result, the function $v\mapsto \phi_{f,g}(\theta,v)$ belongs to $\mathscr S(\R)$ uniformly with respect to~$\theta$. Moreover, the second derivative with respect to~$\theta$ of the function $\Phi_{f,g}$,
$$F_{f,g}''(e^{\theta J}w)\bigl(e^{\theta J} Jw,e^{\theta J} Jw\bigr)-F_{f,g}'(e^{\theta J}w)e^{\theta J}w,$$
belongs to $\mathscr S(\R^2)$ uniformly with respect to $\theta$. This implies that the second derivative with respect to $\theta$ of the function $\phi_{f,g}$ is in $\mathscr S(\R)$ uniformly with respect to $\theta$.
We define the non-cutoff Kac operator as
$$K(g,f)(v)=\poscal{\finp(\un_{(-\frac{\pi}{4},\frac{\pi}{4})}\beta)}{\phi_{f,g}(\cdot,v)},$$
when $\beta$ is a function satisfying \eqref{w5}.
Since $\phi_{f,g}(0,v)\equiv 0$, Lemma \ref{new003} allows to replace the finite part by the absolutely converging integral
$$K(g,f)(v)=\int_{|\theta| \leq \frac{\pi}{4}}\beta(\theta)\Bigl(\int_{\RR}\big({\breve{g}}'_* f'-{\breve{g}}_*f \big) dv_*\Bigr)d\theta=K(\breve{g},f)(v).$$
It was established in~\cite{LMPX1} (Lemma~A.2) that $K(g,f)\in \mathscr{S}(\rr)$, when $g,f \in \mathscr{S}(\rr)$. We also recall the Bobylev formula providing an explicit formula for the Fourier transform of the Kac operator
\begin{equation}\label{cl11}
\widehat{K(g,f)}(\xi)=\int_{|\theta| \leq \frac{\pi}{4}}\beta(\theta) \left[\widehat{\breve{g}}(\xi \sin{\theta})\widehat{f}(\xi\cos{\theta})-\widehat g(0)\widehat{f}(\xi)\right]d\theta,
\end{equation}
when $f,g\in \mathscr S(\R)$.
The proof of this formula may be found in~\cite{LMPX1} (Lemma~A.4).

\subsection{Metrics on the phase space}\label{symbolic}
The purpose of this section is to check that the two metrics
$$\Gamma_0 = \frac{dv^2+d\eta^2}{\langle (v,\eta) \rangle^2}, \qquad \Gamma_1 = \frac{dv^2+d\eta^2}{M(v,\eta,\xi)},$$
defined in (\ref{f1}) are admissible (slowly varying, temperate, satisfying the uncertainty principle). We refer the reader to~\cite{Le} (Definition~2.2.15) for the definition of an admissible metric and the definition of an admissible weight associated to an admissible metric.
Regarding the metric $\Gamma_0$, this property is established in~\cite{Le} (Lemma~2.2.18). As powers of the gain function associated to the metric $\Gamma_0$, the functions $\langle (v,\eta) \rangle^m$, with $m \in \rr$, are admissible weights for the metric $\Gamma_0$.
Regarding the second metric
$$\Gamma_1=\frac{dv^2+d\eta^2}{M(v,\eta,\xi)},$$
we begin by checking that this metric is slowly varying. To that end, it is sufficient to check that
\begin{multline*}
\exists C>0, \exists r>0, \forall (v_1,\eta_1) \in \rr^2, \forall (v_2,\eta_2) \in \rr^2, \forall \xi \in \rr, \\
\frac{|v_1-v_2|^2+|\eta_1-\eta_2|^2}{M(v_1,\eta_1,\xi)} \leq r \Rightarrow \frac{1}{C}M(v_1,\eta_1,\xi) \leq M(v_2,\eta_2,\xi) \leq C M(v_1,\eta_1,\xi).
\end{multline*}
Indeed, when
$$\frac{|v_1-v_2|^2+|\eta_1-\eta_2|^2}{M(v_1,\eta_1,\xi)} \leq r,$$
it follows from (\ref{ya2}) that
\begin{multline}\label{f6}
M(v_2,\eta_2,\xi)=1+v_2^2+\eta_2^2+(1+v_2^2+\eta_2^2+\xi^2)^{\frac{1}{2s+1}} \lesssim 1+v_2^2+\eta_2^2+|\xi|^{\frac{2}{2s+1}} \\
\leq 1+2v_1^2+2\eta_1^2+2(v_1-v_2)^2+2(\eta_1-\eta_2)^2+|\xi|^{\frac{2}{2s+1}} \leq 2(r+1)M(v_1,\eta_1,\xi),
\end{multline}
since $0<s<1$. On the other hand, we have
\begin{multline*}
M(v_1,\eta_1,\xi)=1+v_1^2+\eta_1^2+(1+v_1^2+\eta_1^2+\xi^2)^{\frac{1}{2s+1}} \lesssim 1+v_1^2+\eta_1^2+|\xi|^{\frac{2}{2s+1}} \\
\leq 1+2v_2^2+2\eta_2^2+2(v_1-v_2)^2+2(\eta_1-\eta_2)^2+|\xi|^{\frac{2}{2s+1}} \leq 2M(v_2,\eta_2,\xi)+2rM(v_1,\eta_1,\xi).
\end{multline*}
This implies that $M(v_1,\eta_1,\xi) \lesssim M(v_2,\eta_2,\xi)$ when $0<r\ll 1$. This proves that the metric $\Gamma_1$ is slowly varying.
According to~\cite{Le} (Lemma~2.2.14), it is sufficient for checking the temperance to establish that
\begin{multline*}
\exists C>0, \exists N \geq 0, \forall (v_1,\eta_1) \in \rr^2, \forall (v_2,\eta_2) \in \rr^2, \forall \xi \in \rr, \\
\frac{M(v_2,\eta_2,\xi)}{M(v_1,\eta_1,\xi)} \leq C\big(1+M(v_1,\eta_1,\xi)(|v_1-v_2|^2+|\eta_1-\eta_2|^2)\big)^N.
\end{multline*}
Indeed, we deduce from (\ref{f6}) that
$$\frac{M(v_2,\eta_2,\xi)}{M(v_1,\eta_1,\xi)} \leq 2+\frac{2(v_1-v_2)^2+2(\eta_1-\eta_2)^2}{M(v_1,\eta_1,\xi)} \leq 2\big(1+M(v_1,\eta_1,\xi)(|v_1-v_2|^2+|\eta_1-\eta_2|^2)\big),$$
since $M(v_1,\eta_1,\xi) \geq 1$.
This proves that the metric $\Gamma_1$ is temperate. This metric also trivially satisfies the uncertainty principle since $M(v_1,\eta_1,\xi) \geq 1$. This implies that $\Gamma_1$ is an admissible metric on the phase space $\rr_{v,\eta}^2$ uniformly with respect to the parameter $\xi \in \rr$.
As powers of the gain function associated to the metric $\Gamma_1$, the functions $M^m$, with $m \in \rr$, are admissible weights for the metric $\Gamma_1$ uniformly with respect to the parameter~$\xi \in \rr$.

\bigskip
\noindent
{\bf Acknowledgements.}
The research of the second author is supported by the Grant-in-Aid for Scientific Research No.~25400160, Japan Society for the Promotion of Science. The research of the third author is supported by the ANR NOSEVOL (Project: ANR 2011 BS01019 01). The research of the last author is supported partially by ``The Fundamental Research Funds for Central Universities''
and the National Science Foundation of China No. 11171261.


\begin{thebibliography}{aa}
\bibitem{AS1} R. Alexandre, M. Elsafadi, \textit{Littlewood Paley decomposition and regularity issues in Boltzmann homogeneous equations. I. Non cutoff and Maxwell cases}, Math. Models Methods Appl. Sci. 15 (2005), no. 6, 907-920
\bibitem{AS2}R. Alexandre, M. Elsafadi, \textit{Littlewood-Paley theory and regularity issues in Boltzmann homogeneous equations. II. Non cutoff case and non
Maxwellian molecules}, Discrete Contin. Dyn. Syst. 24 (2009), 1-11
\bibitem{AMUXY2}
R. Alexandre, Y. Morimoto, S. Ukai, C.-J. Xu, T. Yang, \textit{Regularizing effect and local existence for the non-cutoff Boltzmann equation}, Arch. Ration. Mech. Anal. 198 (2010), no. 1, 39-123
\bibitem{AMUXY1}
R. Alexandre, Y. Morimoto, S. Ukai, C.-J. Xu, T. Yang, \textit{The Boltzmann equation without angular cutoff in the whole space: qualitative properties of solutions}, Arch. Ration. Mech. Anal. 202 (2011), no. 2, 599-661
\bibitem{AM}
R. Alexandre, Y. Morimoto, S. Ukai, C.-J. Xu, T. Yang, \textit{Global existence and full regularity of the Boltzmann equation without angular cutoff}, Comm. Math. Phys. 304 (2011), no. 2, 513-581
\bibitem{AMUXY-KJM} R. Alexandre, Y. Morimoto, S. Ukai, C.-J. Xu, T. Yang, \textit{Smoothing effect of weak solutions for the spatially homogeneous Boltzmann equation without angular cutoff}, Kyoto J. Math. 52 (2012), 433-463
\bibitem{MR2877343}
D. Ars{\'e}nio, N. Masmoudi, \textit{Regularity of renormalized
  solutions in the {B}oltzmann equation with long-range interactions}, Comm.
  Pure Appl. Math. 65 (2012), no.~4, 508-548
  \bibitem{MR2832590}
D. Ars{\'e}nio, L. Saint-Raymond, \textit{Compactness in kinetic
  transport equations and hypoellipticity}, J. Funct. Anal. 261
  (2011), no.~10, 3044-3098
  \bibitem{MR2594923}
V. Bagland, \textit{Well-posedness and large time behaviour for the
  non-cutoff {K}ac equation with a {G}aussian thermostat}, J. Stat. Phys.
  138 (2010), no.~4-5, 838-875
\bibitem{17}
C.~Cercignani, \textit{The Boltzmann Equation and its Applications}, Applied Mathematical Sciences, vol. 67, Springer-Verlag, New York (1988)
\bibitem{D95}
L.~Desvillettes, \textit{About the regularization properties of the non cut-off Kac equation}, Comm. Math. Phys.  168 (1995), 417-440
\bibitem{MR1407542}
L. Desvillettes, \textit{Regularization for the non-cutoff {$2$}{D} radially
  symmetric {B}oltzmann equation with a velocity dependent cross section},
  Proceedings of the {S}econd {I}nternational {W}orkshop on {N}onlinear
  {K}inetic {T}heories and {M}athematical {A}spects of {H}yperbolic {S}ystems
  ({S}anremo, 1994), vol.~25 (1996), 383-394

\bibitem{DFT} L. Desvillettes, G. Furioli, E. Terraneo, \textit{Propagation of Gevrey regularity for solutions of the Boltzmann equation for Maxwellian molecules}, Trans. Amer. Math. Soc. 361 (2009), no. 4, 1731-1747 
\bibitem{MR1720101}
L. Desvillettes, C. Graham, S. M{\'e}l{\'e}ard, 
\textit{Probabilistic interpretation and numerical approximation of a {K}ac
  equation without cutoff}, Stochastic Process. Appl. 84 (1999),
  no.~1, 115-135


\bibitem{DW}L. Desvillettes, B. Wennberg, \textit{Smoothness of the solution of the spatially homogeneous Boltzmann equation without cutoff}, Comm. Partial Differential Equations, 29 (2004), no. 1-2, 133-155



 \bibitem{MR3050180}
R. Duan,  R.-M. Strain, \textit{On the full dissipative property of the
  {V}lasov-{P}oisson-{B}oltzmann system}, Hyperbolic problems theory,
  numerics and applications. {V}olume 2, Ser. Contemp. Appl. Math. CAM,
  vol.~18, World Sci. Publishing, Singapore (2012), 398-405

\bibitem{evans}
L.C. Evans, \textit{Partial differential equations}, Graduate Studies in Mathematics, 19, American Mathematical Society, Providence, RI (1998)
\bibitem{gelfand}
I.M. Gelfand, G.E. Shilov, \textit{Generalized Functions}, Vol. 2, Academic Press, New York (1968)

\bibitem{GLX1} L. Glangetas, H. Li, C.-J. Xu, \textit{Sharp regularity properties for the non-cutoff spatially homogeneous Boltzmann equation}, preprint (2014), http://hal.archives-ouvertes.fr/hal-01088989  



\bibitem{MR2976422}
F. Golse, \textit{Homogenization and kinetic models in extended
  phase-space}, Riv. Math. Univ. Parma, 3 (2012), no.~1, 71-89

\bibitem{MR1711273}
C. Graham,  S. M{\'e}l{\'e}ard, \textit{Existence and regularity of a
  solution of a {K}ac equation without cutoff using the stochastic calculus of
  variations}, Comm. Math. Phys. 205 (1999), no.~3, 551-569

\bibitem{rodino1} T. Gramchev; S. Pilipovi\'c, L. Rodino, \textit{Classes of degenerate elliptic operators in Gelfand-Shilov spaces},
in ``New Developments in Pseudo-Differential Operators'', Oper. Theory Adv. Appl. 189, 15-31, Birkh\"auser, Basel (2009)

\bibitem{HelNie05HES}
B. Helffer, F. Nier, \textit{Hypoelliptic estimates and spectral theory for Fokker-Planck operators and Witten Laplacians}, Lecture Notes in
  Mathematics, vol. 1862, Springer-Verlag, Berlin (2005)

\bibitem{Herau-Starov}
F. H\'erau, K. Pravda-Starov, \textit{Anisotropic hypoelliptic estimates for Landau-type operators}, J. Math. Pures Appl. 95 (2011), 513-552

\bibitem{Hor85} L. H\"{o}rmander, \textit{The analysis of linear partial differential operators}, vol. I-IV, Springer-Verlag (1985)

\bibitem{HMUY} Z.H. Huo, Y. Morimoto, S. Ukai, T. Yang, \textit{Regularity of solutions for spatially homogeneous Boltzmann equation without angular cutoff}, 
Kinet. Relat. Models,  1 (2008), 453-489


\bibitem{langen}
M. Langenbruch, \textit{Hermite functions and weighted spaces of generalized functions}, Manuscripta Math. 119 (2006), 269-285


\bibitem{MR2556716}
N. Lekrine, C.-J. Xu, \textit{Gevrey regularizing effect of the
  {C}auchy problem for non-cutoff homogeneous {K}ac's equation}, Kinet. Relat.
  Models, 2 (2009), no.~4, 647-666

\bibitem{Le} N. Lerner, \textit{Metrics on the phase space and non-selfadjoint pseudo-differential operators}, Pseudo-Differential Operators, Theory and Applications, Vol. 3, Birkh\"auser (2010)

\bibitem{MR3063534}
N. Lerner, Y. Morimoto, K. Pravda-Starov, C.-J. Xu,
\textit{Phase space analysis and functional calculus for the linearized
  {L}andau and {B}oltzmann operators}, Kinet. Relat. Models, 6 (2013),
  no.~3, 625-648

\bibitem{LMPX1}
N. Lerner, Y. Morimoto, K. Pravda-Starov, C.-J. Xu, \textit{Spectral and phase space analysis of the linearized non-cutoff Kac collision operator},
J. Math. Pures Appl. 100 (2013), no. 6, 832-867

\bibitem{LMPX3}
N. Lerner, Y. Morimoto, K. Pravda-Starov, C.-J. Xu, \textit{Gelfand-Shilov smoothing properties of the radially symmetric spatially homogeneous Boltzmann equation without angular cutoff},  J. Differential Equations, 256 (2014), no. 2, 797-831

\bibitem{MUXY-DCDS}
Y. Morimoto, S. Ukai, C.-J. Xu, T. Yang, \textit{Regularity of solutions to the spatially homogeneous Boltzmann equation without angular cutoff}, Discrete Contin. Dyn. Syst. A, 24 (2009), 187-212

\bibitem{MX} Y. Morimoto, C.-J. Xu, \textit{Ultra-analytic effect of Cauchy problem for a class of kinetic equations}, J. Differential Equations, 247 (2009), 596-617

\bibitem{rodino} F. Nicola, L. Rodino, \textit{Global Pseudo-Differential Calculus on Euclidean Spaces}, Pseudo-Differential Operators, Theory and Applications, 4, Birkh\"auser Verlag, Basel (2010)

\bibitem{toft} J. Toft, A. Khrennikov, B. Nilsson, S. Nordebo, \textit{Decompositions of Gelfand-Shilov kernels into kernels of similar class},  J. Math. Anal. Appl. 396 (2012), 315-322

\bibitem{villani2}
C. Villani, \textit{A review of mathematical topics in collisional kinetic theory}, Handbook of Mathematical Fluid Dynamics, vol. I, North-Holland, Amsterdam (2002), 71-305

\bibitem{36}
B. Wennberg, \textit{Regularity in the Boltzmann equation and the Radon transform}, Comm. Partial Differential Equations, 19 (1994), 2057-2074
\end{thebibliography}
\end{document}